\documentclass{article}

\usepackage{fullpage}
\usepackage{cite}
\usepackage{amsthm,amsmath,amssymb,amsfonts}
\usepackage[ruled, linesnumbered]{algorithm2e}
\usepackage{commands}
\usepackage{color}
\usepackage{booktabs}
\usepackage{hyperref}

\title{Deconvolution with Unknown Error Distribution\\Interpreted as Blind Isotonic Regression
}
\author{
	Devavrat Shah\\
	\texttt{devavrat@mit.edu}
	\and
	Dogyoon Song\\
	\texttt{dgsong@mit.edu}
}
\date{}

\begin{document}

\maketitle

\begin{abstract}
Deconvolution is a statistical inverse problem to estimate the distribution of a random variable 
based on its noisy observations. Despite the extensive studies on the topic, deconvolution with 
unknown noise distribution remains as a notoriously hard problem. We propose a matrix-based 
viewpoint for collective deconvolution that subsumes the setup with repeated measurements as 
a special case. As the main result, we describe a simple algorithm that partially utilizes matrix 
structure to solve deconvolution problem and provide non-asymptotic error analysis for the algorithm. 
We show that the proposed algorithm achieves the minimax optimal rate for deconvolution in 
a restricted sense. We also remark the connection between the collective deconvolution and 
the so-called statistical seriation as a byproduct or our matrix viewpoint. We conjecture that the link 
suggests that collective deconvolution, as well as deconvolution with repeated measurements, 
is intrinsically much easier than usual deconvolution of a single distribution.
\end{abstract}


\tableofcontents


%


\section{Introduction}

Deconvolution is a statistical inverse problem to estimate the distribution of the underlying signal random variable $X$, 
based on the observations $\{ Z_1, \ldots, Z_n\}$ where $Z = T(X)$ for some transformation $T$. For example, $T(X) = X + N$ 
with $N$ denoting additive noise, when $Z$ represents the noisy measurement of $X$. When $X$ and $N$ are independent 
and admit densities, the density of $Z$ is given as the convolution $f_Z = f_X * f_N$. Assuming a priori knowledge of $f_N$ 
(equivalently, of $T$), one can solve the convolution equation (i.e., `deconvolve') with the empirical distribution of $Z$ to estimate $f_X$.

There is a vast literature on theory and applications of deconvolution, spanning form the early works on reflection seismology and 
optical imaging to studies on the optimal rates of deconvolution estimators. Under the common assumption of a priori knowledge on $T$, 
kernel deconvolution estimators have been widely studied to estimate the unknown density/distribution and they are known to achieve 
the minimax optimal rate \cite{Fan1991, tsybakov2009springer}. In particular, the optimal rates are determined by the smoothness class 
of the signal distribution and the noise densities.

Despite the extensive studies, the requirement of knowing $T$ remains as a major challenge in deconvolution. There have been 
various approaches proposed to overcome the difficulty, suggesting to exploit some types of side information to estimate $T$ first 
and then solve the usual deconvolution problem with estimated $\hat{T}$. For instance, \cite{Delaigle2008} consider the setup where 
one can measure the same entity multiple times and propose to utilize the repeated measurements to estimate the noise distribution.

In this paper, we study a matrix-based viewpoint for the deconvolution problem. Specifically, we consider the setup where there are 
$m$ signal random variables $X_1, \ldots, X_m$ of interest and we want to estimate the $m$ distributions simultaneously from a dataset 
that captures certain `commonality' in the distributions. Our framework subsumes the setup with repeated measurements as a special case 
where $X_i = X$ for all $i = 1, \ldots, m$. 

We summarize our contribution in this paper as follows. First, we propose a two-step algorithm for deconvolution (and matrix estimation) 
and provide a non-asymptotic error analysis for the algorithm that matches the optimal rate for deconvolution of a single distribution. 
Second, we point out the potential connection between deconvolution with repeated measurements to arguably much easier statistical 
problems, namely, the statistical seriation and the isotonic regression with latent features. The latter observation suggests the possibility of 
achieving an exponentially faster rate than the minimax optimal rate for deconvolution (which is logarithmic), hinting that deconvolution 
with repeated measurements is intrinsically much easier than usual deconvolution.

\subsection{Our Contribution}

As the main contribution of this work, we present a matrix-based viewpoint for deconvolution that enables robust extension of the works 
by \citet{Fan1991} and \citet{Delaigle2008} as noted earlier. To be precise, we let $A \in \Reals^{m \times n}$ denote the matrix we want to 
estimate and assume the latent variable model as the generative model for the matrix $A$, which is to be described in Section \ref{sec:model}. 
In addition, we assume certain `commonality' across the rows of $A$; we assume there exists a permutation of columns that rearranges entries 
in every row of $A$ to be monotone nondecreasing. Assuming the latent variable model, we reformulate the problem of estimating the distributions 
of $m$ signal random variables $X_1, \ldots, X_m$ as the problem of estimating the latent function associated with a matrix from its partial, noisy measurement.

Based on the proposed viewpoint, we describe an algorithm to estimate the distributions of $X_1, \ldots, X_m$ in the course of estimating the matrix $A$. 
The described algorithm operates in the following steps: (i) it estimates the column permutation utilizing the `commonality' (shared monotonicity as in 
\eqref{eqn:biLipschitz}) across the rows; (ii) it estimates the noise distribution using the estimated proximity between columns; (iii) it estimates the 
latent function for each row by modified kernel deconvolution estimator; and lastly, (iv) it estimates the matrix by plugging in the estimated permutation and 
the estimated latent function. We progressively develop the algorithm starting from the simplest noiseless setting in Section \ref{sec:alg_noiseless} 
to noisy measurement setup with known noise distribution in Section \ref{sec:alg_noisy_known} and then to the generic noisy measurement scenario 
with unknown noise distribution in Section \ref{sec:alg_noisy_unknown}. The fully developed algorithm is presented in Algorithm \ref{alg:noisy_unknown} 
with a subroutine for the noise estimation in Algorithm \ref{alg:setT_main}.

We provide non-asymptotic error analysis for the proposed algorithm in terms of two error metrics: (i) max row-$\ell^2$ norm and (ii) matrix maximum norm. 
Both are more stringent error metrics compared to the traditional Frobenius norm (i.e., mean squared error). We provide upper bounds on the error 
of our proposed algorithm using both error metrics; see Corollary \ref{coro:answer.1} and Corollary \ref{coro:answer.2}, respectively.

Note that the max row-$\ell^2$ norm error is closely related to the maximum (taken over the $m$ distributions) deconvolution error in the $L^2$ norm sense, 
i.e., (square root of) mean squared error of deconvolution. We discuss information-theoretic lower bounds on the squared $L^2$ error to argue the optimality 
of the obtained upper bounds; see Corollary \ref{coro:noiseless} for the lower bound for function approximation without noise and deconvolution, and see 
Corollary \ref{coro:lower_noisy} for the lower bound for deconvolution. Both corollaries are derived based on classical hardness results from function approximation 
and deconvolution literature.

Last but not least, we comment on the connection between the collective deconvolution considered in the current work and the problem of statistical seriation 
as a by-product of the matrix-based viewpoint toward deconvolution. Seriation is the problem of finding a permutation to rearrange the matrix entries to satisfy 
certain shape constraints, e.g., monotonicity and very recently, a statistical model for seriation is studied in \cite{flammarion2019optimal}. They also show that 
the least square estimator achieves the optimal rate for statistical seriation in terms of squared Frobenius norm error, which scales as $\big( \frac{\log n}{n} \big)^{2/3}$. 
In spite of the difference in the estimation objectives and the model assumptions, their results suggest the possibility of achieving exponentially faster rate for 
deconvolution by estimating the matrix first and interpreting it as the empirical distribution. Further discussion can be found in Section \ref{sec:discussion}.

We summarize the upper and lower bounds for the estimation error $\bbE_{Z}~ \sup_{i \in [m]} \| \hat{F}_i - F \|_{L^2}^2$ under three noise scenarios in 
Table \ref{tab:summary}. Observe that the proposed algorithm is optimal among the estimators that estimates $F_i$ based only on the information from row $i$. 
We also include the conjectured improved rates that are expected to be achievable by `collaborative' estimators in the last row of the table.
\begin{table}
\centering
\caption{Summary of the results and conjectured rates for estimating $\bbE_{Z}~ \sup_{i \in [m]} \| \hat{F}_i - F \|_{L^2}^2$.}
\label{tab:summary}
\begin{tabular}{l c c c}
	\toprule
					&	Noiseless	& Known Noise	&Unknown Noise\\
	\midrule
		Upper Bound	&	$\cO\Big( \frac{\log(mnp)}{np} \Big)$	& $\cO\Big( \big(\log (np) \big)^{-\frac{2}{\beta}} \Big)$	& $\cO\Big( \big(\log (np) \big)^{-\frac{2}{\beta}} \Big)$	\\
					&	(Corollary \ref{coro:noiseless})		&	(Corollary \ref{coro:lower_noisy})				&	(Corollary \ref{coro:answer.1})	\\
	\midrule
		Lower Bound	&	$\Omega\Big( \frac{1}{np}\Big)$		& $\Omega\Big( \big(\log (np) \big)^{-\frac{2}{\beta}} \Big)$	&	\\
		(non-collaborative)	&	(Corollary \ref{coro:noiseless})		&	(Corollary \ref{coro:lower_noisy})				&	\\
	\midrule
		Conjecture	&	& $\cO\Big( \big( \frac{\log (np)}{np} \big)^{\frac{2}{3}} \Big)$	& $\cO\Big( \big( \frac{\log (np)}{np} \big)^{\frac{2}{3}} \Big)$	\\
	\bottomrule
\end{tabular}
\end{table}

\subsection{Related Work}
\subsubsection{Deconvolution}
Early works in deconvolution literature focus on addressing how to estimate the signal density assuming a specific form of noise distribution distribution and 
computing the rates of convergence for the proposed methods. These early works include \cite{mendelsohn1982deconvolution, Carroll1988, devroye1989consistent, 
Stefanski1990, stefanski1990rates, fan1993adaptively} to name a few. Among the vast amount of literature, 
\cite{Fan1991} discusses how the dispersion characteristic of the noise influence the difficulty of the deconvolution problem by introducing the notion of 
ordinary smooth- and supersmooth- noise, thereby providing insights on the hardness of nonparametric deconvolution.

Subsequently, the harder problem of density estimation with unknown error density has been considered. The usual proposal was to estimate the error density 
from side information such as the samples of the error itself \cite{johannes2009deconvolution}. In particular, the setup with replicated measurements\footnote{That is 
to say, the observer is allowed to measure the same signal with independent measurement error multiple times.} for each inherently different samples drew much 
attention \cite{diggle1993fourier, neumann1997effect}, for example. \cite{Delaigle2008} argues that a modified kernel deconvolution estimator using the estimated 
error density achieves the same first order property as the original kernel deconvolution estimator considered in \cite{Carroll1988, Fan1991}.

In this paper, we restrict ourselves to supersmooth noise and `nice' distribution functions and thus we are able to estimate distribution and quantile function of 
the signal from the estimated density using `plug-in' estimator, as discussed in \cite{Fan1991}. However, estimation of distributions, moments, quantiles, etc. can be 
more complicated in general and does not follow as an immediate consequence of density estimation \cite{hall2008estimation, dattner2011deconvolution, dattner2016adaptive}.


\subsubsection{Isotonic Regression}

Our work also has a similar flavor with so-called isotonic regression, whose goal is in estimating an unknown function under a shape constraint. 
Isotonic regression is a classical topic in the field of nonparametric statistics and has drawn many researchers' interests on its own. In the simplest form, 
one assumes the response variables $Y_i$ and covariates $X_i$ satisfy $Y_i = f(X_i) + N_i$, $1 \leq i \leq n$ for some nondecreasing regression function $f$, 
where $N_i$'s are i.i.d. noises. The objective is in estimating a nondecreasing function $\hat{f}_n$ that minimizes the average loss at design points. 
Since the least squares type methods for isotonic estimation were proposed by \cite{ayer1955empirical, vaneeden1956maximum, grenander1956theory}, 
there has been an extensive study to develop algorithms and analyze the risk bounds. In early works, the convergence in distribution at a fixed point with 
the rate no slower than $n^{-1/3}$ was established \cite{rao1969estimation, brunk1969estimation}. In subsequent works, the same $n^{-1/3}$-rate for 
the convergence in probability was achieved for the least square estimator under the sub-Gaussian noise assumption \cite{van1990estimating, van1993hellinger}. 
Then Donoho obtained the $n^{-1/3}$ upper bound on the mean squared error ($L^2$ risk) for i.i.d. Gaussian noise \cite{donoho1990gelfand}, and 
this i.i.d. Gaussian assumption is weakened to the finiteness of some exponential moment by Birg\'e \cite{birge1993rates}. In more recent works, other types of 
risk bounds and techniques have been studied, e.g., Stein's method for mean squared error \cite{meyer2000degrees} and general $l^p$ risk based on martingale 
method \cite{zhang2002risk}. We refer interested readers to \cite{barlow1972statistical, barlow1972isotonic, grenander1981abstract, robertson1988order, groeneboom2012information} for a more general discussion on statistical methods with order restrictions.

If we treat the measurements in a single row of the matrix as the covariate, the connection to isotonic regression is evident as distribution function is always 
nondecreasing. However, there is a significant difference that covariates are corrupted with noise in our setup. This already sets a major obstacle in applying 
pooling algorithms (which is the zero-th order local smoothing) to our setup, which are widely studied in the isotonic regression literature. 

\subsubsection{Matrix Estimation, Latent Variable Model, and Statistical Seriation}
\paragraph{Matrix Estimation}
Our problem of interest is closely related to, but goes beyond matrix estimation -- our objective is not only to recover the matrix, but estimate 
the distributions of the signal random variable associated with the matrix. In the last fifteen years, there have been a huge amount of advances 
in the matrix estimation, especially in spectral approaches and convex optimization based approaches. Since \cite{srebro2004generalization} 
suggested to use low-rank matrix approximation in this context, many statistically efficient estimators based on optimization have been suggested. 
They prove that $rn \log n$ samples out of $n^2$ entries suffice to impute the missing entries by matrix factorization, where $r$ is rank of the 
matrix to recover \cite{candes2009exact, candes2010power, rohde2011estimation, keshavan56matrix, koltchinskii2011nuclear, negahban2012restricted, jain2013low}.

However, many of these approaches require that the matrix is of low rank ($r \ll n$) to achieve a sensible sample complexity. Note that we consider 
a matrix of common monotonicity pattern and such a matrix can have high rank even though it has certain shape constraints.

\paragraph{Latent Variable Model}
Latent variable model is a more general model than the low-rank matrix model and it subsumes the low rank model as a special case -- let the latent features 
be $r$ dimensional vectors and the latent function be their inner product. Chatterjee proposed the universal singular value thresholding (USVT) estimator 
inspired by low-rank matrix approximation and he argued that the USVT estimator provides an accurate estimate for any Lipschitz function under the 
latent variable model \cite{Chatterjee15}. However, his analysis is based on step function approximation (stochastic block model approximation) and 
$\Omega\big( n^{2 - \frac{2}{r+2}} \big)$ observations out of $n^2$ are required to obtain a consistent estimate for an $n \times n$ matrix, where $r$ stands 
for the dimension of the latent spaces.The rate of the USVT estimator is further investigated in a more recent work by \citet{xu2017rates}.

In contrast, \cite{LLSS2016} suggested a similarity-based estimator for collaborative filtering and they proved that their estimator requires 
$\Omega\big( n^{\frac{3}{2} + \delta} \big)$ for any small $\delta >0$ out of $n^2$ for consistency of the estimator, as long as $r = o \left( \log n \right)$. 
As the name `blind regression' suggests, their estimator is effectively a kernel regression estimator defined on the latent feature space with 
a surrogate metric defined by behavioral pattern of the function values. They report that the overlap requirement between pairs of rows, namely $np^2 \gg 1$, 
determines the sample complexity of the estimator, which is a commonly observed phenomenon in neighbor-based approaches.

We may view the algorithm proposed in this paper as a `blind isotonic regression' estimator when viewing it as a method for matrix estimation.
The suggested algorithm can avoid this restrictive overlap requirement by assuming shared monotonicity property.

\paragraph{Statistical Seriation}
Seriation is the problem of finding a permutation to rearrange the matrix entries to satisfy certain shape constraints, e.g., monotonicity. In a recent work, 
a statistical model for seriation is proposed and the optimal rate for estimation is studied \cite{flammarion2019optimal}. The authors consider the setup 
where they observe $Z = A \Pi + N$ where $A \in \Reals^{m \times n}$ is assumed to belong to a class of matrices that satisfy certain shape constraints, 
$\Pi \in \Reals^{n \times n}$ is an unknown permutation matrix, and $N \in \Reals^{m \times n}$ denotes the noise. The goal is to estimate the product 
$A \Pi$. They show that the least square estimator achieves the optimal rate for statistical seriation in terms of squared Frobenius norm error, 
which scales as $\big( \frac{\log n}{n} \big)^{2/3}$ and also propose a computationally efficient two-step estimator that first estimates $\Pi$ in a similar 
procedure as ours and then estimate $A$ with the least squares. 

We note that their estimation objectives and model assumptions are similar to ours but slightly different. First, we do not assume full observation of $Z$ 
but allow for a partial observation. Second, we measure the error in max $\ell_2$ norm sense (or in matrix maximum norm sense), which is a more 
stringent error metric than the Frobenius norm. Lastly, we want to estimate the underlying distributions beyond estimating the values in the instantiated matrix. 
Due to the differences, we cannot directly utilize their results in our problem but their results suggest the possibility of achieving exponentially faster rate for 
deconvolution with repeated measurements by estimating the matrix first and interpreting it as the empirical distribution. Further discussion can be found 
in Section \ref{sec:discussion}.

\section{Problem Setup}\label{sec:setup}

In this section, we formally state our model and the problem of interest.

\subsection{Model: the Latent Variable Model}\label{sec:model}
Suppose that there is a matrix $A \in \Reals^{m \times n}$ we want to estimate. 
We assume the following generative model for $A$; there exist latent features $\frow{i}, \fcol{j} \in [0,1] \subset \Reals$ 
for each $i \in [m], j \in [n]$ and a latent function $g: [0,1] \times [0,1] \to \Reals$ such that 
\begin{equation}\label{eqn:latent_ftn}
	A(i,j) = g\big( \frow{i}, \fcol{j}\big).
\end{equation}
We assume the latent features are independent and identically distributed as per some (unknown) latent distribution $\cD_{\textrm{row}}, \cD_{\textrm{col}}$, i.e., $\frow{i} \sim \cD_{\textrm{row}}$ and $\fcol{j} \sim \cD_{\textrm{col}}$.

Note that there always exists such a latent model representation for exchangeable data and we may assume $\cD_{\textrm{row}}$, 
$\cD_{\textrm{col}}$ are the uniform distribution over $[0,1]$ with $g$ being some measurable function according to the celebrated 
Aldous-Hoover theorem \citet{Aldous81, Hoover82}. From now on, we let both $\cD_{\textrm{row}}$ and $\cD_{\textrm{col}}$ 
be the uniform distribution on $[0,1]$.

Our objective is to estimate\footnote{See Section \ref{sec:assumption_function} for the precise meaning of estimation of $g$.} 
the latent function $g$ from an incomplete and noisy measurement of $A$. Here we describe our model assumptions on 
the measurement model and the regularity of $g$.

\subsubsection{Measurement Model}
Let $\obs \subset [m] \times [n]$. We suppose the following measurement model:
\begin{equation}
	Z(i,j) = \begin{cases}
			A(i,j) + N(i,j)		& \textrm{if }(i,j) \in \obs,\\
			\textrm{unknown}	& \textrm{otherwise,}
		\end{cases}
\end{equation}
where $N \in \Reals^{m \times n}$ is a noise matrix. 
We impose the following assumptions on $N$ and $\obs$.
\paragraph{Assumptions on the Noise $N$}\label{sec:assumption_noise}
We assume the following properties hold for the noise distribution.
\begin{itemize}
	\item
	$N \equiv -N$ in distribution.	
	\item
	$N(i,j)$ are independent
	\item
	For each $i \in [m]$, there exists a random variable $N_i$ such that $N(i,j) \equiv N_i$ in distribution that satistifes
	\begin{itemize}
		\item
		(sub-gaussianity) there exists $\sigma_i > 0$ such that 
		$$ \bbE\big[ \exp (tN_i) \big] \leq \exp \Big( \frac{t^2 \sigma_i^2}{2}\Big), ~ \forall t \in \Reals.$$
		\item
		(super-smoothness) there exist $B_i \geq 1$, and $\beta_i, \gamma_i > 0$ such that
		\begin{equation}\label{eqn:model_supersmooth}
			\frac{1}{B_i}e^{ -\gamma_i |t|^{\beta_i} } \leq  \phi_{N_i} (t)  \leq B_i e^{ -\gamma_i |t|^{\beta_i} }, ~ \forall t \in \Reals,
		\end{equation}
		where $\phi_{N_i} (t)$ is the characteristic function of $N_i$.
	\end{itemize}
\end{itemize}
A centered Gaussian random matrix with i.i.d. entries is a typical example of such noise. For the simplicity of the exposition, 
we let $\sigma_i = \sigma, B_i = B, \gamma_i = \gamma, \beta_i = \beta$ for all $i \in [m]$.

\begin{remark}
Independence and sub-gaussianity are helpful in the analysis because they allow for the use of concentration inequalities. 
Symmetry and supersmoothness are commonly assumed in deconvolution literature for the success of plug-in CDF estmiator, 
which is obtained by integrating the deconvolution estimator of the density.
\end{remark}

\paragraph{Assumption on the $\obs$}\label{sec:assumption_mask}
Suppose that $M \in \{0, 1\}^{m \times n}$ is a random matrix with its entries drawn i.i.d. from Bernoulli distribution with 
parameter $p \in (0,1]$. Given an instance of $M$, we let 
\begin{equation}\label{eqn:obs_set}
	\obs = \{(i,j) \in [m] \times [n] \text{ s.t. } M(i,j) = 1\}.
\end{equation}
We refer to $M$ as the mask matrix.

\subsubsection{Regularity Assumptions on $g$}\label{sec:assumption_function}
To begin with, we remark that estimating $g$ from $Z$ without any structural assumptions is an ill-posed problem. Latentw variable 
representation of $A$ is not unique and there are multiple equivalent representations for $g$ up to measure-preserving transformations\footnote{For 
example, we can apply an invertible transform to the domain (the space of latent features) and take the push-forward of the latent function 
with respect to the transform, so that $A(i,j)$ remains the same under the new representation.}. 
We bypass this hurdle by redefining the objective as estimating $g\big(\frow{i}, \cdot~\big): [0,1] \to \Reals$ for $i \in[m]$ instead of 
estimating the bivariate latent function $g$ and imposing certain regularity assumptions on $g$ with respect to the second argument.


To be precise, we suppose that the latent function $g: [0,1]^2 \to \Reals$ satisfies the following two properties.
\begin{itemize}
	\item $g$ is bounded, i.e., $-\infty < \fmin \leq \fmax < \infty$ where 
	\begin{align*}
		\fmax	\triangleq \sup_{x,y \in [0,1]} g(x,y)  \qquad &\text{~and~}\qquad
		\fmin 	\triangleq \inf_{x,y \in [0,1]} g(x,y).
	\end{align*}

	\item $g$ is $(\lipmin, \lipmax)$ bi-Lipschitz with respect to the second argument. That is to say, there exist $\lipmin, \lipmax > 0$ 
	such that for all $x$ and for all $y_1 \neq y_2$,
	\begin{equation}\label{eqn:biLipschitz}
		0 < \lipmin \leq \frac{g(x,y_2) - g(x,y_1)}{y_2 - y_1} \leq \lipmax < \infty.	
	\end{equation}
\end{itemize}
A bi-Lipschitz mapping is injective (actually strictly monotone increasing), and is a bijection onto its image. Therefore, for each $x \in [0,1]$, 
we can define the inverse map of $g(x, \cdot)$ as $g^{-1}_{x} : \left[g(x,0), g(x,1) \right] \to [0,1]$. It is easy to verify that $g_x^{-1}$ is 
$\big(\frac{1}{\lipmax}, \frac{1}{\lipmin}\big)$ bi-Lipschitz. We may interpret $g^{-1}_{x}$ as the distribution function $F_x$ of a density $f_x$ 
that is supported on the interval $\left[g(x,0), g(x,1) \right]$ and $\frac{1}{\lipmax} \leq f_x(z) \leq  \frac{1}{\lipmin}$ for $z \in \left(g(x,0), g(x,1) \right)$.

Lastly, we remark here that the monotonicity of $g$ is assumed only with respect to the second argument and we do not impose such 
monotonicity assumptions with regard to the first argument.

\subsection{Problem Statement}\label{sec:statement}
\subsubsection{Deconvolution}
Let $F_i = g_{\frow{i}}^{-1}$ for all $i \in [m]$, which is the distribution function of the random variable associated with the $i$-th row of $A$. 
We want to estimate $F_i$ for all $i \in [m]$ from the data matrix $Z$. Suppose that $\varphi: Z \mapsto (\hat{F}_1, \ldots , \hat{F}_m )$ is 
an estimator of $F_1, \ldots, F_m$ based on $Z$. We define the risk of $\varphi$ using the squared $L_2$ loss maxized over $i \in [m]$, i.e.,
\begin{equation}\label{eqn:loss_phi}
	\RiskD(\varphi) = \bbE_{Z}~ \LossD\big( \varphi(Z); F_1, \ldots, F_m \big)	
		\qquad\text{where}\qquad \LossD(\hat{F}_1, \ldots, \hat{F}_m; F_1, \ldots, F_m) = \sup_{i \in [m]} \| \hat{F}_i - F \|_{L^2}^2.
\end{equation}
That is, we evaluate the performance of the estimator $\varphi$ in the $L_2$ sense for the worst $\hat{F}_i$ over $i \in [m]$. With the aid 
of above notion of risk, we pose the first problem of interest as follows.

\begin{problem}\label{problem:deconv}
	Can we build an efficient algorithm $\varphi$ to estimate $F_1, \ldots, F_m$ that achieves the optimal rate of $\RiskD(\varphi)$ as $m, n \to \infty$?
\end{problem}

\subsubsection{Matrix Estimation}
In some applications, one may want to estimate the matrix $A$ from its partial, and possibly noisy observation $Z$, rather than estimating 
$F_1, \ldots, F_m$. Let $\psi: Z \mapsto \hat{A}$ be an estimator of $A$ from $Z$. We define the risk of $\psi$ as follows\footnote{The loss 
function is the squared max norm of $\hat{A} - A$, or equivalently, the squared $L_{\infty, \infty}$ matrix norm of $\hat{A} - A$.}:
\begin{equation}\label{eqn:loss_psi}
	\RiskME(\psi) = \bbE_{Z}~ \LossME\big( \varphi(Z); A \big)
		\qquad\text{where}\qquad \LossME( \hat{A}; A ) = \sup_{(i, j) \in [m] \times [n] } | \hat{A}(i,j) - A(i,j) |^2.
\end{equation}
Now we pose the second problem of our interest as the following. 

\begin{problem}\label{problem:mat_est}
	Can we build an efficient algorithm $\psi$ to estimate $A$ from $Z$ such that $\RiskME(\psi) \to 0$ as $m,n \to \infty$? 
	What are the upper and lower bounds on $\RiskME(\psi)$?
\end{problem}

We provide a partial answer to Problem \ref{problem:deconv} in Corollary \ref{coro:answer.1} and discuss about the optimality (in some sense) 
of the achieved rate from deconvolution viewpoint in Corollary \ref{coro:lower_noisy}. We also provide a partial answer 
to Problem \ref{problem:mat_est} in establishing an upper bound in Corollary \ref{coro:answer.2}.

\section{Algorithm}\label{sec:algorithm}

In this section, we describe our algorithm to estimate $F_1, \ldots, F_m$ and reconstruct $A$ from $Z$. The generic procedure 
consists of three steps: (1) estimating the column feature $\fcol{j}$ for all $j \in [n]$; (2) estimating $F_1, \ldots, F_m$ using the 
`rankings' estimated in step 1; and (3) reconstructing the matrix $A$ by combining the aforementioned estimates together.
The details in the first two steps vary depending on the noise assumptions and are adapted for each of the three noise scenarios considered 
in this work: noiseless (Section \ref{sec:alg_noiseless}), noisy with known noise distribution (Section \ref{sec:alg_noisy_known}), and noisy 
with unknown noise distribution (Section\ref{sec:alg_noisy_unknown}).

\medskip
\noindent{\bf Notation.}
For $i \in [m]$, and for $j \in [n]$, we define
\begin{align}
	\cB_i &= \{ j' \in [n]: M(i,j') = 1 \},		\label{eqn:set_support_row}\\
	\cB^j &= \{ i' \in [m]: M(i',j) = 1 \}.	\label{eqn:set_support_column}
\end{align}
We let $\bbI$ denote the indicator function, i.e., given a boolean formula, namely, `condition', $\bbI\{\text{condition} \} = 1$ if and only 
if condition is true (and $0$ otherwise). Lastly, we define $\bbI_H: \Reals \to \left\{0, \frac{1}{2}, 1 \right\}$ as
\begin{equation}\label{eqn:Heaviside}
	\heavi{x} = \frac{1}{2} \big( \Ind{x > 0} + \Ind{ x\geq 0 } \big).
\end{equation}

\medskip
\noindent{\bf Handling exceptions.}
For completeness, we describe how our algorithm handles exceptions such as $\cB_i = \emptyset$ or $\cB^j = \emptyset$. 
For $j \in [n]$ with $\cB^j = \emptyset$, we let our algorithm output a trivial estimate $\hfcol{j} = \frac{1}{2} $. Likewise, for $i \in [m]$ 
with $\cB_i = \emptyset$, we let our algorithm return a trivial estimate\footnote{In case, $\fmin, \fmax$ are not known a priori, 
we instead use any given constants $\tfmin, \tfmax$ such that $ \tfmin \leq \fmin$ and $ \tfmax \geq \fmax$.} $\hat{g}^{(i)} (z) 
= (\fmax - \fmin) z + \fmin $ for $z \in [0,1]$.

\subsection{Scenario 1: Noiseless Setup}\label{sec:alg_noiseless}
As a warm-up, we describe our algorithm when there is no noise, i.e., when $N = 0$.
 
\begin{algorithm}
\caption{Algorithm in the noiseless setup}
\label{alg:noiseless}
{
\begin{enumerate}
	\item Estimation of $\fcol{j}$: 
		For all $j \in [n]$ and all $i \in \cB^j$, we define 
		\begin{equation}\label{eqn:quantile_rowwise}
			\hat{q}_i(j) = \frac{1}{| \cB_i |}\sum_{j' \in \cB_i}  \heavi{ Z(i, j) - Z(i, j') }. 	
		\end{equation}
		Then we define our estimate for $\fcol{j}$ to be
		\begin{equation}\label{eqn:quantile_est}
			\hfcol{j} = \hat{q}_{i^*}(j)
		\end{equation}
		where $i^* = i^*(j)$ is chosen from $\cB^j$ uniformly at random.
	\item Estimation of $F_i$:
		For $i \in [m]$, we define $\breve{F}_i: \Reals \to [0,1]$ as
		\begin{equation}\label{eqn:ECDF_noiseless}
			\breve{F}_i (z) = \frac{1}{| \cB_i |}\sum_{j' \in \cB_i } \Ind{Z(i,j') \leq z }.
		\end{equation}
	
	\item Estimation of $A$ by plug-in: 
		For each $i \in [m]$ and $j \in [n]$, $\hat{A}(i,j) = \breve{F}_i^{-1} \big( \hfcol{j} \big)$.
	\end{enumerate}
}
\end{algorithm}

We note that for any given $x \in [0,1]$, the latent function $g(x, \cdot): [0,1] \to \mathbb{R}$ $(\lipmin, \lipmax)$ is invertible 
due to our model assumptions. We interpret $F_i = g^{-1}(\frow{i}, \cdot): \mathbb{R} \to [0,1]$ as the distribution function 
of the random variable associated with the $i$-th row. With an estimate $\breve{F}_i$ of $F_i$ at hand, we define an estimate of 
$g(\frow{i}, \cdot)$ as the (pseudo-) inverse of $\breve{F}_i$\footnote{That is, we view $F_i$ as a CDF and $g(\frow{i}, \cdot)$ as 
the corresponding quantile function. See Definitions \ref{defn:CDF} and \ref{defn:Quantile} in Appendix \ref{appx:distribution} for details.}.

\subsection{Scenario 2: Noisy Measurement Setup with Known Noise Distribution}\label{sec:alg_noisy_known}
Now we consider a more realistic setup where we observe $Z$ with nontrivial additive noise $N$. 
First, notice that we cannot simply use $\hat{q}_i(j)$ defined in \eqref{eqn:quantile_rowwise} -- the empirical quantile along 
a given row $i$ -- as a proxy of $\fcol{j}$ unlike the noiseless setting. However, we can overcome the obstacle by ``averaging'' 
out the noise. To that end, we shall use empirical quantile estimation based on the ``averaged'' value. For each $j \in [n]$, we define 
\begin{equation}\label{eqn:Z_marg}
	Z_{\marg}(j) = \frac{ 1 }{ | \cB^j |} \sum_{i' \in \cB^j } Z(i',j)
\end{equation}
and
\begin{equation}\label{eqn:estimate_marg}
	\hat{q}_{\marg}(j) = \frac{1}{n}\sum_{j' = 1}^n ~\heavi{ Z_{\marg}(j) - Z_{\marg}(j') }.
\end{equation}

Also, when estimating $F_i$, we cannot simpy use the empirical CDF $\breve{F}_i$ any longer. Instead, we define kernel 
deconvolution estimator of $F_i$ by integrating the kernel deconvolution estimator of density $f_i$.
Since $Z(i,j) = A(i,j) + N(i,j)$ is the sum of two independent random variables $A(i,j) = g(\frow{i}, \fcol{j})$ and $N(i,j)$, the density 
of $Z$ is given as the convolution of the signal density and the noise density. We estimate the distribution of the signal random variable 
by traditional plug-in kernel deconvolution estimator, which reconstructs the signal density by shaving off the noise and then integrate 
the density.	

Let $\phi_{N_i}$ denote the characteristic function of the noise, which is the Fourier transform of the noisy density. 
Let $K$ be a symmetric Kernel and $\phi_K$ denote its Fourier transform. We assume 
\begin{itemize}
	\item 
		$\supp \phi_K \subset [-1,1]$, i.e., $\phi_K(t) = 0$ if $t \not\in[-1,1]$. 
	\item
		$\kmax=\max_{t \in [-1,1]} \left| \phi_K(t) \right| < \infty$
\end{itemize}
Using $K$ and the knowledge on the noise distribution, we define a function $L_i$ as 
$L_i \triangleq \cF^{-1} \left\{ \frac{ \phi_K(\, \cdot \,) }{\phi_{N_i}(\, \cdot \, h^{-1})} \right\}$, i.e., for $z \in \Reals$,
\begin{equation}\label{eqn:kernel_known_main}
	L_i(z) = \frac{1}{2\pi} \int \exp(- \img\, t z ) \frac{\phi_K(t)}{\phi_{N_i}\left(\frac{t}{h}\right)} dt.
\end{equation}
For each $i \in [m]$, we define the kernel deconvolution estimator of the density using $L$ as
\begin{equation}\label{eqn:known_density}
	\tilde{f}_i(z) = \frac{1}{h |\cB_i|} \sum_{j\in \cB_i} L_i \left( \frac{z- Z(i,j)}{h} \right)
\end{equation}
where $h$ denotes the kernel bandwidth parameter. Specifically, we choose $h = \left(4\gamma \right)^{\frac{1}{\beta}}
\left( \log |\cB_i| \right)^{-\frac{1}{\beta}}$ where $\beta$ and $\gamma$ are smoothness parameters for the noise $N_i$: cf. \eqref{eqn:model_supersmooth}.
Lastly, we obtain $\tilde{F}_i$ by integrating $\tilde{f}_i$.

\begin{algorithm}
\caption{Algorithm in the noisy setup when the noise distribution is known}
\label{alg:noisy_known}
{
\begin{enumerate}
	\item Estimation of $\fcol{j}$: 
		For all $j \in [n]$, we let $\hfcol{j} = \hat{q}_{\marg}(j)$, cf. \eqref{eqn:estimate_marg}.

	\item Estimation of $F_i$:
		For $i \in [m]$, we define $\tilde{F}_i: \Reals \to [0,1]$ as
		\begin{equation}\label{eqn:ECDF_known_noise}
			\tilde{F}_i(z) = \begin{cases}
				\int_{\fmin}^{z } \tilde{f}_i(w) dw,	& \text{if } z < \fmax,\\
							1,			& \text{if } z \geq \fmax.
			\end{cases}
		\end{equation}
		where $\tilde{f}_i$ is defined as in \eqref{eqn:known_density}.
	
	\item Estimation of $A$ by plug-in: 
		For each $i \in [m]$ and $j \in [n]$, $\hat{A}(i,j) = \tilde{F}_i^{-1} \big( \hfcol{j} \big)$.
	\end{enumerate}
}
\end{algorithm}

\subsection{Scenario 3: Noisy Measurement Setup with Unknown Noise Distribution}\label{sec:alg_noisy_unknown}
When the noise distribution is not known a priori, the CDF estimate defined in \eqref{eqn:ECDF_known_noise} is no longer valid 
because the deconvolution kernel $L_i$ requires the knowledge of $\phi_{N_i}$; see \eqref{eqn:kernel_known_main}. 
To overcome the challenge, we first estimate the noise characteristic function and then define a modified deconvolution estimator 
with the estimate. We first discuss in Section \ref{sec:deconv_mod} how to modify the deconvolution estimator, assuming the availability 
of accurate noise characteristic function estimation. Then we argue in Section \ref{sec:noise_estimation} that such an accurate estimation 
of the noise characteristic function $\phi_N(t)$ is possible by providing an explicit form of the estimator $\hat{\phi}_{N, i}(t)$ in \eqref{eqn:chN_est} 
and a concrete construction algorithm, cf. Algorithm \ref{alg:setT_main}.

\subsubsection{Modified Deconvolution Kernel Estimator}\label{sec:deconv_mod}
Fix $i \in [m]$. Suppose that we are given $\hat{\phi}_{N,i}$ such that $\hat{\phi}_{N,i}(t) \approx \phi_{N_i}(t)$ for all $t \in \big[-\frac{1}{h}, \frac{1}{h} \big]$. 
We assume $\hat{\phi}_{N, i}(t)$ is real and $\hat{\phi}_{N, i}(t) \geq 0$ for all $t \in \Reals$.

With $\hat{\phi}_{N, i}$ at hand, we define a modified deconvolution kernel $\hat{L}_i$ as
\begin{equation}\label{eqn:kernel_estimated}
	\deconvker(z) = \frac{1}{2\pi} \int e^{-\img tz} \frac{\phi_K(t)}{\hat{\phi}_{N,i}\left(\frac{t}{h}\right) + \rho}dt.	
\end{equation}
In this paper, we specifically choose the ridge parameter $\rho = \frac{1}{B} |\cB_i|^{-\frac{9}{20}}$ (we may choose 
$\rho = \frac{1}{B} |\cB_i|^{-\frac{1}{2} + \delta}$ for any $0 < \delta < \frac{1}{4}$) for the convenience of our analysis. 
Then we define 
\begin{equation}\label{eqn:unknown_density}
	\hat{f}_i(z) = \frac{1}{h |\cB_i|} \sum_{j\in \cB_i} \deconvker \left( \frac{z- Z(i,j)}{h} \right)
\end{equation}
with the same choice of the bandwidth parameter $h = \left(4\gamma\right)^{\frac{1}{\beta}} \left( \log \left| \cB_i \right| \right)^{-\frac{1}{\beta}}$ 
as in Section \ref{sec:alg_noisy_known}. The rest of the procedure remains the same.

\begin{algorithm}
\caption{Algorithm in the noisy setup when the noise distribution is unknown}
\label{alg:noisy_unknown}
{
\begin{enumerate}
	\item Estimation of $\fcol{j}$: 
		For all $j \in [n]$, we let we let $\hfcol{j} = \hat{q}_{\marg}(j)$, cf. \eqref{eqn:estimate_marg}.
		
	\item Estimation of $F_i$: For $i \in [m]$, 
		\begin{itemize}
		\item
		we estimate $\phi_N(t)$ with $\hat{\phi}_{N, i}(t)$ as described in \eqref{eqn:chN_est}, and then 
		 
		\item
		define $\hat{F}_i: \Reals \to [0,1]$ as
		\begin{equation}\label{eqn:ECDF_unknown_noise}
			\tilde{F}_i(z) = \begin{cases}
				\int_{\fmin}^{z } \hat{f}_i(w) dw,	& \text{if } z < \fmax,\\
							1,			& \text{if } z \geq \fmax.
			\end{cases}
		\end{equation}
		where $\hat{f}_i$ is defined as in \eqref{eqn:unknown_density}.
		\end{itemize}
	
	\item Estimation of $A$ by plug-in: 
		For each $i \in [m]$ and $j \in [n]$, $\hat{A}(i,j) = \hat{F}_i^{-1} \big( \hfcol{j} \big)$.
	\end{enumerate}
}
\end{algorithm}

\subsubsection{Estimation of the Noise Distribution}\label{sec:noise_estimation}
To begin with, suppose that we can repeatedly observe the same instance 
$X$ of target random variable up to independent additive noise, i.e., $Z^{(j)} = X + N^{(j)}$ with $N^{(j)}$ independent. 
Although we don't know the value of $X$, we can see that the difference in the observed data entries is equal to the difference 
between two independent noise instances: $Z^{(1)} - Z^{(2)} = \big( X + N^{(1)} \big) - \big( X + N^{(2)} \big) = N^{(1)} - N^{(2)} $. 
Assuming symmetry in the noise distribution, $N^{(1)} - N^{(2)} \equiv N^{(1)} + N^{(2)} $. Therefore, $\phi_{N^{(1)} - N^{(2)}}(t) = \phi_N(t)^2$.
From symmetry of $N$, we know that $\phi_N(t)$ is real-valued. Moreover, it is positive because $N$ is assumed to be supersmooth. 
Therefore, we can estimate $\phi_N(t)$ by taking square root of the (absolute value of the) estimate $\hat{\phi}_{N^{(1)} - N^{(2)}}(t)$ as
\[
	\hat{\phi}_N(t) = \hat{\phi}_{N^{(1)} - N^{(2)}}(t)^{\frac{1}{2}} 
		= \bigg| \frac{1}{n} \sum_{i=1}^n \cos \Big[ t \big( N^{(1)} - N^{(2)} \big) \Big] \bigg|^{\frac{1}{2}}.
\]
However, the repeated measurement assumption is not feasible because we have {\em at most} one measurement for a given index 
$(i, j)$. Despite this challenge, we can still hope to obtain {\em nearly} repeated samples from observations in a given row, 
if we choose columns $j_1, j_2 \in [n]$ that have {\em very} similar features $\fcol{j_1} \approx \fcol{j_2}$ so that $A(i, j_1) - A(i,j_2) \approx 0$ and
\begin{align*}
	Z(i,j_1) - Z(i,j_2)		
		&= \left[ A(i, j_1) - A(i,j_2) \right]  + \left[ N(i,j_1) - N(i,j_2) \right]
		\approx N(i,j_1) - N(i,j_2). 	
\end{align*}

For the ease of exposition, we assume $N_i \equiv N$ in distribution for all $i \in [m]$. We estimate $\phi_N$ as follows.
\begin{enumerate}
	\item 
	Construct set $\cT$ as described in Algorithm \ref{alg:setT_main}.

	\item 
	For each $i \in [m]$, define 
	\begin{equation}\label{eqn:set_Ti}
		\cT_i := \Big\{ (i',j_1, j_2) \in \cT: i' \neq i\Big\}.
	\end{equation}
	and define
	\begin{equation}\label{eqn:chN_est}
		\hat{\phi}_{N, i}(t) =	
			\Bigg| \frac{1}{\left| \cT_i \right|} \sum_{ \left(i', j_1, j_2 \right) \in \cT_i} 
				\cos \Big[ t \big( Z(i', j_1) - Z(i', j_2)  \big) \Big] \Bigg|^{\frac{1}{2}}.	
	\end{equation}
\end{enumerate}
Intuitively, $\cT$ is the set of index triples to imitate the repeated measurements. The refinement of $\cT$ to $\cT_i$ for each row $i$ is done 
only for the convenience in our analysis and might be unnecessary; one may be able to define $\hat{\phi}_N$ with the entire set $\cT$ and use 
it for all $i \in [m]$.

\begin{algorithm}
	\caption{Construction of the set $\cT$}
	\label{alg:setT_main}
	\SetAlgoLined
	\SetKwInOut{Input}{input}
	\SetKwInOut{Output}{output}
    
    	\Input{Data matrix $Z$ of size $(m, n)$}
	\Output{The set of index triples $\cT \subset [m] \times [n] \times [n]$}
	$J \gets \big\{ j \in [n]: |\cB^j| \geq \frac{mp}{2}\big\}$\;
	$I \gets \big\{ i \in [m]: |\cB_i \cap J| \geq \frac{|J|p}{2}\big\}$\;
	$\cT \gets \emptyset$ \;
	Sort $j \in [n]$ in the increasing order of $\hat{q}_{\marg}(j)$, i.e., find a permutation $\pi$ such that $\hat{q}_{\marg}\left( j \right) \leq \hat{q}_{\marg}\left( j' \right)$ if $\pi(j) <\pi( j')$\;
	\For{$i \in I$}{
 		Renumber $j \in \cB_i \cap J$ with $j' \in \left[\left| \cB_i \cap J \right| \right]$ in the increasing order of $\hat{q}_{\marg}\left( j \right)$\;
		(let $\sigma_i: \cB_i \cap J \subseteq [n] \to \left[\left| \cB_i \cap J \right| \right]$; this map can be induced from $\pi$)\\
		$j' \gets 0$\;
		\While{$j' \leq \left| \cB_i \cap J \right| - 1 $}{
		\eIf{$\hat{q}_{\marg}\big( \sigma_i^{-1} (j' + 1)\big) - \hat{q}_{\marg}\big( \sigma_i^{-1} (j')\big) \leq \frac{1}{\sqrt{\left| \cB_i \cap J \right|}}$}{
			$\cT \gets \cT \cup \big\{ (i, \sigma_i^{-1}\left(j'\right), \sigma_i^{-1}(j'+1)) \big\}$\;
			$j' \gets j' + 2$\;
		}{
			$j' \gets j' + 1$\;
		}
		}
 	}
\end{algorithm}

\section{Main Results on Noise Scenario 3}	\label{sec:main_results}

\subsection{Definitions of Key Quantities}
First, we let
\begin{align*}
	\constaa &= \frac{1}{ \lipmin } \big( \fmax - \fmin + 2\sigma \big),\\
	\constbb &= \constbb(\lipmin) > 0\\
	\constcc &= \frac{B \kmax (\fmax - \fmin)}{\pi (4\gamma)^{\frac{1}{\beta}}}
\end{align*}

Then we also define two quantities:
\begin{align}
	\valphiAA &= 64 \sqrt{\frac{\log(mn)}{mnp}} \bigg[ 1 + (4\gamma)^{-\frac{1}{\beta}} \lipmax 
		\bigg( \frac{32 \sqrt{ \pi} \constaa}{\sqrt{ mp } } + \frac{2\sqrt{2}}{\sqrt{ n p}} 
		+ 8\sqrt{\frac{\log n}{n}} \bigg) \big( \log n \big)^{\frac{1}{\beta}}\bigg],	\label{eqn:valphi_A}\\
	\valphiBB &= \frac{1}{mn} \Bigg\{ (4\gamma)^{-\frac{2}{\beta}} \Bigg[  6 \lipmax^2 \bigg( \frac{1024 \pi \constaa^2}{mp} + \frac{8}{np} 
		+ \frac{64 \log n}{n} \bigg) + 128 \sigma^2 \log(mn) \Bigg] \big( \log n \big)^{\frac{2}{\beta}} 
		+ 2 ( 4\gamma)^{-\frac{1}{\beta}} \sigma B \big( \log n \big)^{\frac{1}{\beta}}  \Bigg\}.	\label{eqn:valphi_B}
\end{align}

With aid of $\valphiAA$ and $\valphiBB$, we define a conditioning event
\begin{equation}\label{eqn:Ef}
	\Ephi := \bigg\{ \sup_{t \in [-\frac{1}{h}, \frac{1}{h} ]}  \big| \hatNest (t) - \phi_N(t) \big|^2  \leq   \valphiAA + \valphiBB  \bigg\}.	
\end{equation}
Recall that we have chosen $h = (4\gamma)^{\frac{1}{\beta}}(\log n )^{-\frac{1}{\beta}}$.

\subsection{Theorem Statements}

Here we shall establish that $\hat{F}_i$ converges uniformly to $F_i$ in the large sample limit. Specifically, we obtain 
an exponentially decaying probabilistic tail bound for the uniform convergence, conditioned on the availability of a good 
estimator of noise characteristic function (implied by the description of event $\Ephi$ in \eqref{eqn:Ef}).

\begin{theorem}\label{thm:cdf_noisy_unknown}
For $i \in [m]$, let $\tilde{F}_i$ be defined as in \eqref{eqn:ECDF_known_noise} and $\hat{F}_i$ be defined as in 
\eqref{eqn:ECDF_unknown_noise} with $\hat{\phi}_{N}(t) = \hat{\phi}_{N, i}(t)$ as described in Section \ref{sec:noise_estimation}, 
cf. \eqref{eqn:chN_est}. 
Suppose that the kernel bandwidth $h = (4\gamma)^{\frac{1}{\beta}}(\log |\cB_i|)^{-\frac{1}{\beta}}$ and the ridge parameter 
$\rho = \frac{1}{B} |\cB_i|^{-\frac{9}{20}}$. 
If $|\cB_i| \geq 1024$ and $mp$ and $n$ are sufficiently large so that 
$\valphiAA + \valphiBB \leq \frac{1}{B} |\cB_i|^{-\frac{9}{20}}$, then for any $t \geq 0$,
\begin{align*}
	&\bbP\bigg( \sup_{z \in [\fmin, \fmax] } \big| \hat{F}_i (z) - F_i(z) \big| > 
		t + \big( \constbb + \constcc \big) \big( \log n_i \big)^{-\frac{1}{\beta}} 
			+ 4\constcc  \frac{\big( \log n_i \big)^{\frac{1}{\beta}}}{ n_i^{\frac{1}{5}} }	~\bigg|~ \Ephi \cap \big\{ |\cB_i| = n_i \big\} \bigg)\\
		&\qquad\leq 2 n_i^{\frac{9}{20}} \big( \log n_i \big)^{\frac{2}{\beta}}
				\exp\left( - \frac{ n_i^{\frac{1}{10}} }{2 \constcc^2 \left( \log n_i \right)^{\frac{2}{\beta}} } t^2  \right).
\end{align*}
\end{theorem}
By letting $t$ of order $\big( \log |\cB_i| \big)^{-\frac{1}{\beta} }$, we can conclude from Theorem \ref{thm:cdf_noisy_unknown} that 
$\sup_{z \in [\fmin, \fmax] } \big| \hat{F}_i (z) - F_i(z) \big|$ decays to $0$ as $np \to \infty$ at the rate of at least $\big( \log |\cB_i| \big)^{-\frac{1}{\beta}}$ 
with high probability (conditioned on $\Ephi$). The proof of Theorem \ref{thm:cdf_noisy_unknown} can be found in Appendix \ref{sec:proof_noisy_unknown_CDF}.

\begin{remark}\label{rem:valphi}
We observe that 
\begin{align*}
	\valphiAA &= \cO \bigg( \sqrt{\frac{\log(mn)}{mnp}} \bigg),\qquad\text{and}\qquad
	\valphiBB = \cO \bigg( \frac{ \log(mn) (\log n)^{\frac{2}{\beta}}}{mn} \bigg).
\end{align*}
Since $|\cB_i| \leq n$, the condition $\valphiAA + \valphiBB \leq \frac{1}{B} |\cB_i|^{-\frac{9}{20}}$ is easily satisfied 
when $mp$ and $n$ are sufficiently large.
\end{remark}

Next, we argue that such an accurate noise estimation is possible with high probability with the proof of Theorem \ref{thm:ensure_condition} 
postponed to Appendix \ref{sec:proof_ensure}.
\begin{theorem}\label{thm:ensure_condition}
Let $\Ephi$ denote the event as defined in \eqref{eqn:Ef} where $\valphiAA$, $\valphiBB$ are as described in \eqref{eqn:valphi_A}, \eqref{eqn:valphi_B}. 
Then
\[	
	\Prob{\Ephi^c} 
	\leq   \frac{3}{n^7} + \frac{6}{m^7n^7} + n \exp \left( - \frac{mp}{8} \right) + \exp \left( - \frac{m}{16} \right) + 2 \exp \left( - \frac{n}{16} \right).
\]
\end{theorem}

\subsection{Implications}

\subsubsection{On Deconvolution}

Combining Theorem \ref{thm:cdf_noisy_unknown} and Theorem \ref{thm:ensure_condition} leads to Corollary \ref{coro:answer.1}, 
which provides a partial answer to Problem \ref{problem:deconv}. We remark that the corollary implies that $\RiskD$ is approximately 
$2 \big( \fmax - \fmin \big) \big( \constbb + 2\constcc \big)^2 \big( \log \nmin \big)^{-\frac{2}{\beta}} $ in the asymptotic regime where $m, np \to \infty$.

\begin{corollary}[Partial Answer to Problem \ref{problem:deconv}]\label{coro:answer.1}
	Let $\varphi: Z \mapsto (\hat{F}_1, \ldots , \hat{F}_m )$ denote an estimator that outputs $\hat{F}_i$ as described 
	in \eqref{eqn:ECDF_unknown_noise}. If $mp$ and $n$ are sufficiently large so that the condition in 
	Theorem \ref{thm:cdf_noisy_unknown} is satisfied, then
	\begin{align*}
		\RiskD(\varphi)	
			&\leq\big( \fmax - \fmin \big) \Bigg[ 2 \big( \constbb + 2\constcc \big)^2 \Big[ \log \Big( \frac{np}{2} \Big) \Big]^{-\frac{2}{\beta}} 
			+ 32 \constcc^2\frac{\big[ \log ( 2np ) \big]^{\frac{2}{\beta}}}{ \big( \frac{np}{2} \big)^{\frac{2}{5}} }		
			+ 2m (2np)^{\frac{9}{20}} \big[ \log (2np) \big]^{\frac{2}{\beta}} \exp \bigg( - \frac{\big( \frac{np}{2} \big)^{\frac{1}{10}}}{2} \bigg)\\
			&\qquad 
				+ \frac{3}{n^7} + \frac{6}{m^7n^7} + 2m \exp \left( - \frac{np}{8} \right) + n \exp \left( - \frac{mp}{8} \right) + \exp \left( - \frac{m}{16} \right) + 2 \exp \left( - \frac{n}{16} \right) \Bigg].
	\end{align*}
\end{corollary}
We remark here that $\RiskD(\varphi) \lesssim 2\big( \fmax - \fmin \big) \big( \constbb + 2\constcc \big)^2 \Big[ \log \Big( \frac{np}{2} \Big) \Big]^{-\frac{2}{\beta}}$ 
as this leading term dominates the others as $mp, np \to \infty$.

\begin{proof}
Let $\Erow := \cap_{i=1}^m \{ \frac{np}{2} \leq |\cB_i| \leq 2np \}$. We observe that $|\cB_i| = \sum_{j=1}^n \Ind{M_{ij} = 1}$ is the sum of $n$ independent 
Bernoulli random variables for all $i \in [m]$. We have $\Prob{ |\cB_i| < \frac{np}{2} } \leq \exp \big( -\frac{np}{8} \big)$ and 
$\Prob{ |\cB_i| > 2np } \leq \exp \big( -\frac{np}{3} \big)$ for each $i \in [m]$ by the binomial Chernoff bound. Applying the union bound, 
\begin{equation}\label{eqn:upper_row}
	\Prob{\Erow^c} \leq \sum_{i=1}^m \bigg[ \Prob{ |\cB_i| < \frac{np}{2} } + \Prob{ |\cB_i| > 2np } \bigg] \leq 2m \exp \Big( -\frac{np}{8} \Big).
\end{equation}

Now we recall the definition of $\RiskD(\varphi)$ from \eqref{eqn:loss_phi}. We can see that for any $\delta > 0$,
\begin{align*}
	\RiskD(\varphi) &= \bbE_{Z} \Big[~ \sup_{i \in [m]} \| \hat{F}_i - F \|_{L^2[\fmin, \fmax]}^2 ~\Big] \\
		&\leq \big( \fmax - \fmin \big) \Bigg( \delta^2 + \bbP\bigg( \sup_{i \in [m]} \sup_{z \in \Reals} \big| \hat{F}_i(z) - F_i(z) \big| > \delta \bigg) \Bigg)\\
		&\leq \big( \fmax - \fmin \big) \Bigg( \delta^2 + \bbP\bigg( \sup_{i \in [m]} \sup_{z \in \Reals} \big| \hat{F}_i(z) - F_i(z) \big| > \delta ~\Big|~ \Ephi \cap \Erow \bigg) + \Prob{\Ephi^c \cup \Erow^c}  \Bigg)\\
		&\leq \big( \fmax - \fmin \big) \Bigg( \delta^2 + \Prob{\Ephi^c} + \Prob{\Erow^c} +  \sum_{i \in [m]} \bbP\bigg(  \sup_{z \in \Reals} \big| \hat{F}_i(z) - F_i(z) \big| > \delta ~\Big|~ \Ephi \cap \Erow \bigg) \Bigg)
\end{align*}
Let $\nmin = \frac{np}{2}$ and $\nmax = 2np$. With the choice of $t = \constcc \big( \log \nmin \big)^{-\frac{1}{\beta}}$ and 
$\delta = \big( \constbb + 2\constcc \big) \big( \log \nmin \big)^{-\frac{1}{\beta}} + 4 \constcc \frac{\big( \log \nmax \big)^{\frac{1}{\beta}}}{ \nmin^{\frac{1}{5}} } $, 
for all $i \in [m]$,
\begin{align*}
	\Prob{  \sup_{z \in \Reals} \big| \hat{F}_i(z) - F_i(z) \big| > \delta ~\Big|~ \Ephi \cap \Erow }
		&\leq 
			2 \nmax^{\frac{9}{20}} \big( \log \nmax \big)^{\frac{2}{\beta}} \exp \left( - \frac{\nmin^{\frac{1}{10}}}{2} \right).
\end{align*}
We conclude the proof by noticing that $\delta^2 \leq 2\big( \constbb + 2\constcc \big)^2 \big( \log \nmin \big)^{-\frac{2}{\beta}} 
+ 32 \constcc^2\frac{\big( \log \nmax \big)^{\frac{2}{\beta}}}{ \nmin^{\frac{2}{5}} }$.
\end{proof}

\subsubsection{On Matrix Estimation}
We remark that we actually establish the reliability of the estimated column feature, $\hat{q}_{\marg}(j) \approx \fcol{j}$, in the course of 
proving Theorem \ref{thm:cdf_noisy_unknown}. This results is summarized as the following proposition and its proof can be found in 
Appendix \ref{sec:proof_col_noisy}..

\begin{proposition}\label{prop:quantile_noisy}
For any $j \in [n]$, let $\hat{q}_{\marg}(j)$ be defined as in \eqref{eqn:estimate_marg}. Then for any $t > 0$,
\begin{align*}
	\Prob{ \big| \hat{q}_{\marg}(j) - \fcol{j} \big| > t + \frac{8 \sqrt{2 \pi} \constaa}{\sqrt{ m_* } } ~\bigg|~ \Big\{ \min_{j' \in [n]} | \cB^{j'} | = m_* \Big\} }
		&\leq 3\exp \left( -\frac{nt^2}{2} \right)
\end{align*}
\end{proposition}

The above proposition is used as a lemma in the proof of Theorem \ref{thm:cdf_noisy_unknown} in order to argue that the estimated 
noise characteristic function, $\hat{\phi}_{N, i}(t)$, is uniformly close to the true noise characteristic function $\phi_N(t)$ over 
$t \in [-\frac{1}{h}, \frac{1}{h}]$. However, there is a further implication of Proposition \ref{prop:quantile_noisy} when it is combined with 
Theorem \ref{thm:cdf_noisy_unknown}, which provides an upper bound on the error of estimating the matrix $A$ in the max row $\ell_2$ 
norm sense. This result is summarized in Corollary \ref{coro:answer.2}, which also provides an answer to our Problem \ref{problem:mat_est} 
stated in Section \ref{sec:setup}.

\begin{corollary}[Answer to Problem \ref{problem:mat_est}]\label{coro:answer.2}
	Let $\psi$ denote the steps 1-3 of Algorithm. If $mp$ and $n$ are sufficiently large so that the condition in 
	Theorem \ref{thm:cdf_noisy_unknown} is satisfied, then
	\begin{align*}
		\RiskME(\psi)
			&\leq \constdd^2 
				+ 2 \Big( 2 m ( 2np )^{\frac{9}{20}} \big[ \log (2np) \big]^{\frac{2}{\beta}} + 3n \Big) \bigg( \sqrt{\frac{\pi}{2}}\constdd + \constee \bigg) \constee\\
			&\quad	+ 2 \big( \fmax - \fmin \big)^2	\bigg[  \frac{3}{n^7} + \frac{6}{m^7n^7} + m \exp \left( - \frac{np}{8} \right) + 2n \exp \left( - \frac{mp}{8} \right) 
				+ \exp \left( - \frac{m}{16} \right) + 2 \exp \left( - \frac{n}{16} \right) \bigg].
	\end{align*}
	where
	\begin{align*}
		\constdd	&=	\lipmax \bigg\{ \big( \constbb + \constcc \big) \Big[ \log \big( \frac{np}{2} \big) \Big]^{-\frac{1}{\beta}} 
				+ 4\constcc  \frac{\big[ \log ( 2np) \big]^{\frac{1}{\beta}}}{ \big( \frac{np}{2} \big)^{\frac{1}{5}} } 
				+  \frac{8 \sqrt{2 \pi} \constaa}{\sqrt{ \frac{mp}{2} } } \bigg\}\\
		\constee	&=	\lipmax \Big[ \frac{\constcc [ \log (2np) ]^{\frac{1}{\beta}}}{( \frac{np}{2} )^{\frac{1}{20}}} + \frac{1}{\sqrt{n}} \Big].
	\end{align*}
\end{corollary}
We remark here that $\constdd^2$ is the leading term in the upper bound in Corollary \ref{coro:answer.2} as it diminishes to $0$ at a logarithmic rate as $np \to \infty$ 
whereas the other terms decay at least polynomially fast. That is to say, $\RiskME(\psi) \lesssim \lipmax^2 (\constbb + \constcc )^2 
\big[ \log \big( \frac{np}{2} \big) \big]^{-\frac{2}{\beta}}$ as $mp, np \to \infty$.

The proof of Corollary \ref{coro:answer.2} can be found in Section \ref{sec:proof_ME}.

\section{Further Exposition of the Results on Noise Scenarios 1 and 2}	\label{sec:result_full}

We provide results on the other (easier) noise scenarios, arguing upper and lower bounds on the CDF estimation.

\subsection{On Scenario 1: Noiseless Setup}

\subsubsection{Upper Bounds on the Estimation Error}\label{sec:noiseless_upper}
In the noiseless setup, we can establish probabilistic tail bounds on the estimation error of $\hat{q}(j)$ for each $j \in [n]$ and $\breve{F}_i(z)$ 
for each $i \in [m]$ as presented in Proposition \ref{prop:quantile_noiseless} and Proposition \ref{prop:cdf_noiseless}, respectively.

\begin{proposition}\label{prop:quantile_noiseless}
For any $j \in [n]$ and for any $t \geq 0$,
\[	
	\Prob{ \left| \hat{q}(j) - \theta_{col}^{(j)} \right| \geq t ~\Big|~ \Big\{ \min_{i \in \cB^j} | \cB_i | = n_* \Big\}} 
		\leq 2 \exp\left( -2 n_* t^2 \right).
\]
\end{proposition}
\begin{proof}
Recall from Eq. \eqref{eqn:quantile_rowwise} that when conditioned on $\frow{i}$, the quantile of $j$ estimated 
from row $i$ is a function of $|\cB_i| = \sum_{j' = 1}^n M(i,j')$ many independent random variables, 
$H\big( Z(i, j) - Z(i, j') \big)$:
\[\hat{q}_i(j) = \frac{\sum_{j'=1}^n M(i,j') H\big( Z(i, j) - Z(i, j') \big) }{\sum_{j' = 1}^n M(i,j')}.\]
Since $H\big( Z(i, j_1) - Z(i, j_2) \big)$ takes value in $\{ 0, \frac{1}{2}, 1\}$, it satisfies the bounded difference condition. To be more specific, let's consider a perturbation on the column feature associated with one index. For any $j_0 \in [n]$, if $j_0 \in \cB_i$ (i.e., if $M(i, j_0) = 1$), then
\[
	\left|  \left.\hat{q}_i(j)\right|_{\fcol{j_0}=a} - \left.\hat{q}_i(j)\right|_{\fcol{j_0}=b} \right|	\leq \frac{1}{\left| \cB_i \right|},
\]
for any value $a, b \in [0,1]$, while if $j_0 \not\in \cB_i$ (i.e., if $M(i, j_0) = 0$), then obviously
\[
	\left| \left.\hat{q}_i(j)\right|_{\fcol{j_0} = a} - \left.\hat{q}_i(j)\right|_{\fcol{j_0} = b} \right| = 0. 
\]
		
 Since  $\Exp {\hat{q}_i(j)} = \fcol{j}$, we can achieve the following probabilistic tail bound by an application of McDiarmid's inequality 
\[	
	\Prob{ \left| \hat{q}_i(j) -  \fcol{j} \right| \geq t  ~\Big|~ |\cB_i| = n_i } 
		\leq 2 \exp\big( -2 n_i t^2 \big).
\]
According to \eqref{eqn:quantile_est}, we let $\hfcol{j} = \hat{q}_{i^*}(j)$ by choosing $i^* = i^*(j)$ uniformly at random from $\cB^j$. 
We obtain the desired inequality because $\min_{i \in \cB^j} | \cB_i | = n_* $ is assumed.

\end{proof}

\begin{proposition}\label{prop:cdf_noiseless}
	For any $i \in [m]$, let $\breve{F}_i$ be defined as in \eqref{eqn:ECDF_noiseless}. Then for any $t \geq 0$,
	\[	
		\Prob{\sup_{z \in \Reals} \big| \breve{F}_i(z) - F_i(z) \big| > t  ~\Big|~ |\cB_i| = n_i} \leq 2 \exp\big(-2 n_i t^2\big).	
	\]
\end{proposition}
\begin{proof}
The proof is a direct application of Dvoretzky-Kiefer-Wolfowitz inequality; see Lemma \ref{lem:DKW}. 
\end{proof}

Since $\breve{F}_i$ and $F_i$ are distribution functions, $\big| \breve{F}_i (z) - F_i(z) \big| \in [0,1]$ for all $z \in \Reals$. 
Also, we know that $\big| \breve{F}_i (z) - F_i(z) \big| = 0 $ for all $ z \not\in[-\fmax, \fmax]$. Therefore, for each $i \in [m]$,
\[
	\big\| \breve{F}_i - F_i \big\|_{L^2}^2 \leq \big( \fmax - \fmin \big) \big\| \breve{F}_i - F_i \big\|_{L^{\infty}}^2.
\]
This observation yields that for any $\delta > 0$,
\begin{align*}
	&\bbE\Big[ \sup_{i \in [m]} \big\| \breve{F}_i - F_i \big\|_{L^2}^2 \Big]\\
		&\qquad\leq \big( \fmax - \fmin \big) \bigg[ \delta^2 + \bbP\Big(\sup_{i \in [m]}\sup_{z \in \Reals} \big| \breve{F}_i(z) - F_i(z) \big| > \delta \Big) \bigg]\\
		&\qquad\leq \big( \fmax - \fmin \big) \Bigg[ \delta^2 + \bbP\bigg( \sup_{i \in [m]}\sup_{z \in \Reals} \big| \breve{F}_i(z) - F_i(z) \big| > \delta ~\bigg|~
			\Big\{ \min_{i \in [m]} |\cB_i| \geq \frac{np}{2} \Big\} \bigg) \Bigg] + \Prob{ \Big\{ \min_{i \in [m]} |\cB_i| < \frac{np}{2} \Big\} }\\
		&\qquad\leq \big( \fmax - \fmin \big) \Bigg[ \delta^2 + \sum_{i \in [m]} \bbP\bigg( \sup_{z \in \Reals} \big| \breve{F}_i(z) - F_i(z) \big| > \delta ~\bigg|~
			\Big\{ \min_{i \in [m]} |\cB_i| \geq \frac{np}{2} \Big\} \bigg) \Bigg] + \Prob{ \Big\{ \min_{i \in [m]} |\cB_i| < \frac{np}{2} \Big\} }\\
		&\qquad\leq \big( \fmax - \fmin \big) \bigg( \delta^2 + 2m \exp\big(-np \delta^2\big) + m \exp \Big( - \frac{np}{8} \Big) \bigg).
\end{align*}
By letting $\delta = \sqrt{ \frac{\log (mnp)}{np}}$, we can see that 
\begin{equation}\label{eqn:upper_bound.noiseless}
	\bbE\Big[ \sup_{i \in [m]} \big\| \breve{F}_i - F_i \big\|_{L^2}^2 \Big]  
		\leq \big( \fmax - \fmin \big) \bigg( \frac{ \log(mnp) + 2 }{np} + m  \exp \Big( - \frac{np}{8} \Big)  \bigg).
\end{equation}
We can conclude that $\bbE\big[ \sup_{i \in [m]} \| \breve{F}_i - F_i \|_{L^2}^2 \big] \lesssim \big( \fmax - \fmin \big)  \frac{ \log(mnp) }{np}$ as $np \to \infty$, 
assuming $m \leq \exp\big( \frac{np}{16} \big)$. We believe this upper bound on $m$ is an artifact of our analysis -- especially, resulting from 
naively taking the union bound over $i \in [m]$ -- and can be removed.

\subsubsection{Lower Bound on the Estimation Error}\label{sec:noiseless_lower}

Next, we argue that the rate obtained in \eqref{eqn:upper_bound.noiseless} is nearly optimal up to a logarithmic factor, 
based on the results from function approximation theory. Without loss of generality, we may assume $i=1$ by focusing only on 
estimating (the slice of) the latent function associated with the first row. Since there is no noise, our algorithm $\varphi$ can evaluate 
$g(\frow{1}, y)$ without error at points $y \in \{\fcol{j} : j \in [n], ~M(1, j) = 1\}$.

Now, we show that for any slice of true latent function $g_1:= g(\frow{1}, \cdot) : [0,1] \to \Reals$ and for any set of sampling points 
$y_1, \ldots, y_{{n_1}} \in [0,1]$, there exists an adversarial function $g_1^{\dagger}: [0,1] \to \Reals$ such that $g_1(y) = g_1^{\dagger}(y)$ 
for all $y \in \{y_1, \ldots, y_{n_1}\}$, yet $F_1 = \big(g_1\big)^{-1}$ and $F_1^{\dagger} = \big( g_1^{\dagger} \big)^{-1}$ are significantly 
different in the $L^2$ sense. This claim follows from a classical result in function approximation theory.

\begin{lemma}[a simplified version of Lemma 4.4 from \citet{Kudryavtsev1995}]\label{lem:Kudryatsev}
	There exists a universal constant $c$ such that for every ${n_1} \in \Nats$, and for any $y_1, \ldots, y_{{n_1}} \in [0,1]$, there exists a 
	$\delta$-Lipschitz function $h \in L^1[0,1] \cap C^{\infty}[0,1]$ for which 
	\begin{enumerate}
		\item	$h(y_i) = 0$, for all $i=1, \ldots, {n_1}$, and
		\item	$\left\| h \right\|_{L^2[0,1]} \geq c \frac{\delta}{\sqrt{{n_1}}}$.
	\end{enumerate}
\end{lemma}
Note that we may replace $[0,1]$ with any bounded interval $[\fmin, \fmax]$ with a conforming change in the constant $c$.  Suppose that 
$F_1 = \big( g_1 \big)^{-1}$ is $\big( \frac{3}{4\lipmax} + \frac{1}{4\lipmin}, \frac{1}{4\lipmax} + \frac{3}{4\lipmin} \big)$-biLipschitz and let 
$\delta = \frac{1}{4}\big( \frac{1}{\lipmin} - \frac{1}{\lipmax}\big)$. By Lemma \ref{lem:Kudryatsev}, there exists a $\delta$-Lipschitz (and $C^{\infty}$) 
function $h$ such that $h(z) = 0$ for all $z \in \big\{ g_1( \fcol{j} ): ~ j \in \cB_1 \big\}$ and $\left\| h \right\|_{L^2[\fmin, \fmax]} 
\geq c \frac{\delta}{\sqrt{|\cB_1|}}$. Observe that both $F_1$ and $F_1^{\dagger} = F_1 + h$ are $\big( \frac{1}{\lipmax}, \frac{1}{\lipmin} \big)$-biLipschitz, 
and hence, both $g_1$ and $g_1^{\dagger} = \big( F_1^{\dagger} \big)^{-1}$ are valid latent functions in our model.

Notice that there is no way for the algorithm (estimator) $\varphi$ to distinguish $F_1^{\dagger}$ from $F_1$ based on the data, 
$\big\{ Z(1,j): ~ j \in \cB_1 \big\} = \big\{ g_1\big( \fcol{j} \big): ~ j \in \cB_1 \big\}$. Therefore, $\varphi$ would return the same output 
$\hat{F_1}$ even when the true latent function is $g_1$ or it were replaced with $g_1^{\dagger}$ and 
\[
	\big\| F_1 - F_1^{\dagger} \big\|_{L^2[\fmin, \fmax]}^2 
		= \| h \|_{L^2[\fmin, \fmax]}^2 \geq \frac{c^2}{16} \big( \frac{1}{\lipmin} - \frac{1}{\lipmax}\big) \frac{1}{|\cB_1|}
\]
sets a lower bound on the estimation error of $\varphi$ because we may assume $\hat{F_1} = F_1$. By the law of total probability,
\begin{align}
	\bbE \big\| F_1 - F_1^{\dagger} \big\|_{L^2}^2
		&\geq \bbE \Big[ \big\| F_1 - F_1^{\dagger} \big\|_{L^2}^2 ~\big|~ |\cB_1| \geq \frac{np}{2} \Big] ~\Prob{|\cB_1| \geq \frac{np}{2}}	\nonumber\\
		&\geq \frac{c^2}{8} \big( \frac{1}{\lipmin} - \frac{1}{\lipmax}\big) \frac{1}{np} \bigg[ 1 - m \exp \Big( - \frac{np}{8} \Big) \bigg].
		\label{eqn:lower_bound.noiseless}
\end{align}

We summarize \eqref{eqn:upper_bound.noiseless} and \eqref{eqn:lower_bound.noiseless} as the following corollary.
\begin{corollary}\label{coro:noiseless}
	Let $\varphi: Z \mapsto \big(\breve{F}_1, \ldots, \breve{F}_m\big)$ denote an algorithm that estimates $F_1, \ldots, F_m$ where $\breve{F}_i$ is the 
	ECDF as described in \eqref{eqn:ECDF_noiseless}. Then 
	\begin{equation*}
		\RiskD(\varphi) 
			\leq 
				\big( \fmax - \fmin \big) \bigg( \frac{ \log(mnp) + 2 }{np} + m  \exp \Big( - \frac{np}{8} \Big)  \bigg).
	\end{equation*}	
	Now suppose that $\varphi$ is any algorithm that estimates $F_1, \ldots, F_m$ such that $\varphi$ estimates $F_i$ based only on $\{ Z(i',j) : ~i' = i \}$ 
	for each $i \in [m]$. Then there exists some constant $c > 0$, which depends only on $\fmin, \fmax$, such that
	\begin{equation*}
		\RiskD(\varphi) \geq  \frac{c^2}{8} \Big( \frac{1}{\lipmin} - \frac{1}{\lipmax}\Big) \frac{1}{np} \bigg[ 1 - m \exp \Big( - \frac{np}{8} \Big) \bigg].
	\end{equation*}
\end{corollary}


\subsection{On Scenario 2: Noisy Measurement Setup with Known Noise Distribution}

\subsubsection{Upper Bounds on the Estimation Error}\label{sec:noisy_known_upper}
In the noisy measurement setup, we can establish a probabilistic tail bound on the estimation error of $\hat{q}_{\marg}(j)$ for each $j \in [n]$ as 
the upper bound for the noiseless setup that can be found in Proposition \ref{prop:quantile_noiseless}. In fact,  we already presented our probabilistic 
tail upper bound for $\big| \hat{q}_{\marg}(j) - \fcol{j} \big| $ in Proposition \ref{prop:quantile_noisy}.

Here we present a proposition that sets up a tail bound on $ \big| \tilde{F}_i (z) - F_i(z) \big|$ for the noisy measurement setup with known 
noise distribution. Note that the setup is harder than the noiseless setup, but no harder than the noisy measurement setup with unknown noise 
distribution. We refer the reader to Proposition \ref{prop:cdf_noiseless} for the noiseless counterpart and Theorem \ref{thm:cdf_noisy_unknown} 
for the one for the unknown noise setup, respectively.

\begin{proposition}\label{prop:cdf_noisy_known}
	For $i \in [m]$, let $\tilde{F}_i$ be defined as in \eqref{eqn:ECDF_known_noise} with 
	$h = \left(4\gamma \right)^{\frac{1}{\beta}}\left( \log |\cB_i| \right)^{-\frac{1}{\beta}}$. Then for any $t > 0$, 
	\begin{align*}
		\Prob{\sup_{z \in [\fmin, \fmax]} \big|~ \tilde{F}_i (z) - F_i(z)~ \big| > t + \big(\constbb + \constcc \big) \left( \log n_i \right)^{-\frac{1}{\beta}} 
		 	~\Big|~ \big\{ |\cB_i| = n_i \big\}}
			&\leq 2 n_i^{\frac{1}{4}} \left( \log n_i \right)^{\frac{2}{\beta}}
				\exp\left( - \frac{ n_i^{\frac{1}{2}}  }{2 \constcc^2 \left( \log n_i \right)^{\frac{2}{\beta}} } t^2 \right).
	\end{align*}
\end{proposition}
The proof of Proposition \ref{prop:cdf_noisy_known} can be found in Appendix \ref{sec:proof_noisy_known_CDF}.

We derive an upper bound on $\bbE \big\| \tilde{F}_i - F_i \big\|_{L^2}^2$ as we have done in Section \ref{sec:noiseless_upper}.
Observe that $\big| \tilde{F}_i (z) - F_i(z) \big| \in [0,1]$ for all $z \in \Reals$ and that $\big| \tilde{F}_i (z) - F_i(z) \big| = 0 $ for all 
$z \not\in[-\fmax, \fmax]$ by definition of $\tilde{F}_i$. Therefore, for each $i \in [m]$,
\[
	\big\| \tilde{F}_i - F_i \big\|_{L^2}^2 \leq \big( \fmax - \fmin \big) \big\| \tilde{F}_i - F_i \big\|_{L^{\infty}}^2.
\]
With $\delta_0 := \big(\constbb + \constcc \big) \big[ \log (\frac{np}{2}) \big]^{-\frac{1}{\beta}}$, this observation yields that for any $\delta > 0$,
\begin{align*}
	&\bbE\Big[ \sup_{i \in [m]} \big\| \tilde{F}_i - F_i \big\|_{L^2}^2 \Big]\\
		&\qquad\leq \big( \fmax - \fmin \big) \bigg[  \big( \delta + \delta_0 \big)^2 + \bbP\Big(\sup_{i \in [m]}\sup_{z \in \Reals} \big| \tilde{F}_i(z) - F_i(z) \big| > \delta + \delta_0 \Big) \bigg]\\
		&\qquad\leq \big( \fmax - \fmin \big) \Bigg[ \big( \delta + \delta_0 \big)^2 + \bbP\bigg( \sup_{i \in [m]}\sup_{z \in \Reals} \big| \tilde{F}_i(z) - F_i(z) \big| > \delta + \delta_0 ~\bigg|~
			\Big\{ \frac{np}{2} \leq |\cB_i| \leq 2np, ~\forall i \in [m] \Big\} \bigg) \\
			&\qquad\qquad\qquad\qquad\qquad\quad + \Prob{ \Big\{ \frac{np}{2} \leq |\cB_i| \leq 2np, ~\forall i \in [m] \Big\}^c }\Bigg]\\
		&\qquad\leq \big( \fmax - \fmin \big) \Bigg[ \big( \delta + \delta_0 \big)^2 + \sum_{i \in [m]} \bbP\bigg( \sup_{z \in \Reals} \big| \tilde{F}_i(z) - F_i(z) \big| > \delta + \delta_0 ~\bigg|~
			\Big\{ \frac{np}{2} \leq |\cB_i| \leq 2np, ~\forall i \in [m] \Big\} \bigg)\\
			&\qquad\qquad\qquad\qquad\qquad\quad  + 2m \exp \Big( - \frac{np}{8} \Big)  \Bigg]\\
		&\qquad\leq \big( \fmax - \fmin \big) \Bigg[ \big( \delta + \delta_0 \big)^2 + 
			2 m (2np)^{\frac{1}{4}} \big[ \log (2np) \big]^{\frac{2}{\beta}}
				\exp\bigg( - \frac{ \big( \frac{np}{2}\big)^{\frac{1}{2}}  }{2 \constcc^2 \big[ \log (2np) \big]^{\frac{2}{\beta}} } \delta^2 \bigg)
			+ 2m \exp \Big( - \frac{np}{8} \Big) \Bigg].
\end{align*}
By letting $\delta = \frac{\constcc [ \log(2np) ]^{\frac{1}{\beta}} }{(np)^{\frac{1}{8}}}$, we can see that 
\begin{align}
	\bbE\Big[ \sup_{i \in [m]} \big\| \tilde{F}_i - F_i \big\|_{L^2}^2 \Big]  
		&\leq \big( \fmax - \fmin \big) \Bigg[ 
			2  \big(\constbb + \constcc \big)^2 \Big[ \log \Big(\frac{np}{2}\Big) \Big]^{-\frac{2}{\beta}} 
			+ 2 \frac{ \constcc^2 \big[ \log (2np) \big]^{\frac{2}{\beta}}}{(np)^{\frac{1}{4}}}
		\nonumber\\
		&\qquad\qquad\qquad\qquad\quad
			+ 2 m (2np)^{\frac{1}{4}} \big[ \log (2np) \big]^{\frac{2}{\beta}}
				\exp \Big( - \frac{ (np)^{\frac{1}{4}} }{2\sqrt{2}} \Big)
		 	+ 2m  \exp \Big( - \frac{np}{8} \Big)  \Bigg].	\label{eqn:upper_bound.noisy}
\end{align}
We can conclude that $\bbE\big[ \sup_{i \in [m]} \| \tilde{F}_i - F_i \|_{L^2}^2 \big] \lesssim 2( \fmax - \fmin ) (\constbb +\constcc)^2 
\big[ \log \frac{np}{2} \big]^{-\frac{2}{\beta}}$ as $np \to \infty$, assuming $m \leq \exp\big( \frac{ (np)^{\frac{1}{4}} }{4} \big)$. 
We believe this upper bound on $m$ is an artifact of our analysis -- especially, resulting from naively taking the union bound over $i \in [m]$ 
-- and can be removed.

\subsubsection{Lower Bound on the Estimation Error}\label{sec:noisy_known_lower}

Next, we argue that the rate obtained in \eqref{eqn:upper_bound.noisy} is nearly optimal up to a logarithmic factor, 
based on the hardness results from deconvolution literature. Without loss of generality, we may assume $i=1$ by focusing only on 
estimating (the slice of) the latent function associated with the first row. 

First, we recall that each slice of latent function, $g(\frow{i}, \cdot)~i\in[m]$ is interpreted as the inverse of a cumulative distribution function $F_i$ in this work. 
Moreover, $F_i$ admits the density $f_i$ such that $\frac{1}{\lipmax} \leq f_i(z) \leq  \frac{1}{\lipmin}$ for $z \in \big(g(\frow{i},0), g(\frow{i},1) \big) \subset 
[\fmin, \fmax]$. See Section \ref{sec:assumption_function} for more details about the bi-Lipschitzness model assumption.

Next, we define a class of probability densities parametrized by three parameters $d, C$, and $0 \leq \alpha < 1$, following \cite{Fan1991}:
\begin{equation}\label{eqn:fan.class}
	\cC_{d, \alpha, C}: = \left\{ f(x): \left| f^{(d)}(x) - f^{(d)}\left( x + \delta \right)\right| \leq C \delta^{\alpha} \right\},
\end{equation}
where $f^{(d)}$ denotes the $d$-th derivative of $f$. Now we introduce the following hardness result excerpted from \citet{Fan1991}. 
\begin{lemma}[a simplified version of Theorem 4 from \cite{Fan1991}]\label{lem:deconv_hard}
	Let $f \in \cC_{d, \alpha, C}$ and $T(f) = f^{(\lambda)}(x)$ for some $x \in \supp f$. 
	Suppose that $z_1, \ldots, z_n$ are samples drawn from $f$ under the noisy measurement model with supersmooth additive noise.  
	Then there is a universal constant $c > 0$ such that for any estimator $\hat{T}$ of $T(f)$, 
	\begin{equation}\label{eqn:lower_deconvolution}
		\sup_{f \in \cC_{d,\alpha, C}} \bbE \big( \hat{T} - T(f) \big)^2 > c \left( \log n \right)^{-\frac{2(d+\alpha - \lambda )}{\beta}}.
	\end{equation}
\end{lemma}

From the above observations, we can verify that for any valid latent function $g$ in our model, the derived density for all $i \in [m]$ satisfies 
$f_i \in \cC_{0,0, \frac{1}{\lipmin}}$ because $\frac{1}{\lipmax} \leq f_i(z) \leq  \frac{1}{\lipmin}$. As discussed in \cite{Fan1991}, (i)
one can estimate a CDF in the supersmooth case by `plugging-in' (integrating the estimated density), which corresponds to the case $\lambda=-1$ 
and (ii) no estimator can estimate the CDF faster than the rate in \eqref{eqn:lower_deconvolution} with $\lambda=-1$. We refer interested readers 
to see Eq. (2.7) and Theorem 6 of \cite{Fan1991} for the original discussion.

Let $T(f_1) = F_1$ and $\hat{T} = \tilde{F}_1$. We may assume\footnote{That is, we are considering the minimax bound, which provides the 
minimum squared $L^2$ error with respect to the maximally hard latent function instance, for a given estimator.} our lantent function achieves 
the lower bound in \eqref{eqn:lower_deconvolution}. Then
\begin{align}
	\bbE \big\| \tilde{F}_1 - F_1 \big\|_{L^1[\fmin, \fmax]}^2
		&\geq \bbE \Big[ \big\| \tilde{F}_1 - F_1 \big\|_{L^1[\fmin, \fmax]}^2 ~\Big|~ |\cB_1| \geq \frac{np}{2} \Big]	\Prob{|\cB_1| \geq \frac{np}{2} }
			\nonumber\\
		&= \bbE \bigg[ \int_{\fmin}^{\fmax} \big( \tilde{F}_1(z) - F_1(z) \big)^2 dz~\Big|~ |\cB_1| \geq \frac{np}{2} \bigg]	\Prob{|\cB_1| \geq \frac{np}{2} }
			\nonumber\\
		&= \Bigg( \int_{\fmin}^{\fmax} \bbE \Big[ \big( \tilde{F}_1(z) - F_1(z) \big)^2  ~\Big|~ |\cB_1| \geq \frac{np}{2} \Big]  dz \Bigg)	\Prob{|\cB_1| \geq \frac{np}{2} }
			\nonumber\\
		&> c \big( \fmax - \fmin \big) \Big[ \log \Big(\frac{np}{2}\Big) \Big]^{- \frac{2}{\beta}}\bigg[ 1 - \exp\Big( - \frac{np}{8} \Big) \bigg].
			\label{eqn:lower_bound.noisy}
\end{align}

We summarize \eqref{eqn:upper_bound.noisy} and \eqref{eqn:lower_bound.noisy} as the following corollary.
\begin{corollary}\label{coro:lower_noisy}
	Assume the noisy measurement setup with supersmooth additive noise. Let $\varphi: Z \mapsto \big(\tilde{F}_1, \ldots, \tilde{F}_m\big)$ 
	denote an algorithm that estimates $F_1, \ldots, F_m$ where $\tilde{F}_i$ is the kernel deconvolution estimator as described in 
	\eqref{eqn:ECDF_known_noise}. Then 
	\begin{align*}
		\RiskD(\varphi) 
			&\leq \big( \fmax - \fmin \big) \Bigg[ 
			2  \big(\constbb + \constcc \big)^2 \Big[ \log \Big(\frac{np}{2}\Big) \Big]^{-\frac{2}{\beta}} 
			+ 2 \frac{ \constcc^2 \big[ \log (2np) \big]^{\frac{2}{\beta}}}{(np)^{\frac{1}{4}}}\\
			&\qquad\qquad\qquad\qquad\quad
			+ 2 m (2np)^{\frac{1}{4}} \big[ \log (2np) \big]^{\frac{2}{\beta}}
				\exp \Big( - \frac{ (np)^{\frac{1}{4}} }{2\sqrt{2}} \Big)
		 	+ 2m  \exp \Big( - \frac{np}{8} \Big)  \Bigg].
	\end{align*}
	Now suppose that $\varphi$ is any algorithm that estimates $F_i$ based only on $\{ Z(i',j) : ~i' = i \}$ for each $i \in [m]$. Then 
	there exists $c > 0$ such that for any $\varphi$,
	\begin{equation*}
		\RiskD(\varphi) 
			\geq   c \big( \fmax - \fmin \big) \Big[ \log \Big(\frac{np}{2}\Big) \Big]^{- \frac{2}{\beta}}\bigg[ 1 - \exp\Big( - \frac{np}{8} \Big) \bigg].
	\end{equation*}
\end{corollary}

\section{Discussion}\label{sec:discussion}
\subsection{Summary of the Results}
In this work, we propose a matrix-based framework to tackle the hard problem of deconvolution with unknown noise distribution. 
Our framework subsumes the setup of deconvolution with repeated measurements as a special case, which has been suggested 
to reduce the hard deconvolution problem to the usual deconvolution problem with known noise. 

We propose a simple three-step algorithm (Algorithm \ref{alg:noisy_unknown}) and provide a non-asymptotic error analysis. 
Our algorithm first estimates the column features by (noisy) sorting and then estimates the noise density using the ranked column features 
(Algorithm \ref{alg:setT_main}), thereby retrieving the signal CDF that is equivalent to the inverse of the latent function in our model.

In the course of answering to our first main question about the possibility of reliably estimating $m$ distribution in the maximum $L^2$ 
norm sense (Question \ref{problem:deconv}), we prove that our algorithm estimates the noise density very well with high probability 
(Theorem \ref{thm:ensure_condition}) and estimates the signal CDFs with vanishing $L^{\infty}$ error with high probability 
(Theorem \ref{thm:cdf_noisy_unknown}). Consequently, we provide an upper bound on $\RiskD$ of our proposed algorithm, 
which effectively scales as $\big[ \log(np) \big]^{-\frac{2}{\beta}}$ when $mp, np \to \infty$ in Corollary \ref{coro:answer.1}. 
This upper bound matches the minimax lower bound of single CDF deconvolution with known noise distribution -- Corollary \ref{coro:lower_noisy} 
contains the lower bound.

\subsection{Interpretation of the Results}
First, our results reconfirm that with the aid of repeated measurements, deconvolution with unknown noise is no harder than deconvolution 
with known noise. Indeed, the stringent requirement of repeated measurements can be relaxed as our framework allows for simultaneous 
deconvolution of multiple CDFs as long as they have common monotonicity pattern with respect to a certain latent feature (not necessarily observable).

However, we do not think our results imply that deconvolution with unknown noise distribution is as easy as deconvolution with known 
noise distribution. Rather, they should be interpreted as deconvolution with repeated measurements is a substantially easier problem than 
deconvolution with unknown noise distribution. We further elaborate this point by considering the problem from matrix estimation perspective.

\subsection{Connection to Statistical Seriation}
Recall that we use the matrix structure to represent the measurements. In our model, we assumed only a $p \in (0,1]$ fraction out of total $mn$ 
entires of the matrix is available. Recall that we asked in Question \ref{problem:mat_est} whether we can efficiently estimate the total $mn$ numbers 
in the matrix using $mnp$ noisy data points in the matrix maximum norm sense. We answer to this question by separately estimating the CDF 
(inverse of the latent function) and the ranking (latent column feature) and our matrix estimation error is dominated by the error in CDF estimation. 
The resulting upper bound scales at the rate of $\big[ \log(np) \big]^{-\frac{2}{\beta}}$ when $mp, np \to \infty$, cf. Corollary \ref{coro:answer.2}.

In a recent work, the authors of \cite{flammarion2019optimal} consider a closely related problem, called the statistical seriation. In their model, 
they observe a matrix $Y \in \Reals^{m \times n}$ such that $Z = A^* \Pi^* + N$ where $\Pi^*$ is an $n \times n$ permutation matrix, $A^*$ is 
the parameter matrix that has monotone nondecreasing rows, and $N$ is a sub-gaussian noise matrix. They discuss the error rate of the 
least square estimator for estimating $A^* \Pi^*$ in the normalized squared Frobenius norm sense (cf. Corollary 3.4 in \cite{flammarion2019optimal}):
\begin{equation}\label{eqn:optimal_seriation}
	\frac{1}{mn} \big\| \hat{A} \hat{\Pi} - A^* \Pi^* \big\|_F^2 
		\lesssim \bigg( \frac{ (\fmax - \fmin) \sigma^2 \log n}{n} \bigg)^{\frac{2}{3}} + \sigma^2 \frac{ \log n}{\min\{ m, n \}}
\end{equation}
and argue that this rate is minimax optimal up to a log factor\footnote{To be fair, their optimality results extend beyond monotone matrices up to 
unimodal matrices. However, there is no known computationally efficient estimator for the general unimodal case so far, to the best of our knowledge.}.

Despite the optimality in the error rate, the least square estimator is not computationally tractable and hence, the authors of \cite{flammarion2019optimal} 
propose a computationally efficient alternative estimator for the monotonic case. The efficient algorithm sorts the columns to estimate $\Pi^*$ by 
scoring them in a similar manner as we did, and then estimate $A^*$ by solving a least square problem. They show this estimator achieves 
the same error rate (cf. Theorem 4.1 in \cite{flammarion2019optimal}).

We conjecture that deconvolution with repeated measurements can attain a polynomial error rate instead of the current logarithmic rate due to 
the connection with the statistical seriation problem. Suppose that we can strengthen the result of \cite{flammarion2019optimal}; that is, suppose that 
it is possible to solve the statistical seriation problem (1) with a similar error rate as in \eqref{eqn:optimal_seriation} in the max norm sense, (2) based on 
a partially observed $Z$. Then after solving the seriation problem, we have $\hat{A} \hat{\Pi}$ at our disposal. The $n$ number of entries in the 
$i$-th row of $\hat{A} \hat{\Pi}$ form a set of `denoised' samples with a residual error upper bounded by the max norm error bound. Now most of the 
original sub-gaussian noise in each sample is peeled off and there remains only a small error that decays to $0$ at a polynomial rate of $n$. Therefore, 
the empirical CDF constructed from the $n$ points in the $i$-th row of $\hat{A} \hat{\Pi}$ well approximates the `pure' ideal empirical CDF with no noise at all. 
The ideal empirical CDF is uniformly close to the true CDF in accordance with Proposition \ref{prop:cdf_noiseless} (or see Dvoretzky-Kiefer-Wolfowitz inequality; 
Lemma \ref{lem:DKW}) and therefore, the empirical CDF based on $\hat{A} \hat{\Pi}$ will be a good uniform approximation of $F_i$. It could be an interesting 
direction of future research to rigorously investigate the validity of this argument.

\bibliographystyle{IEEEtran}
\bibliography{seriation_deconvolution_bibliography}
\newpage

\appendix
\onecolumn
\section{Prelude to the Proof of Theorem \ref{thm:cdf_noisy_unknown}: Proof of Proposition \ref{prop:cdf_noisy_known}}\label{sec:proof_noisy_known_CDF}
In this section, we prove Proposition \ref{prop:cdf_noisy_known} to show that $\tilde{F}_i$ is close to $F_i$ in the $L^{\infty}$ sense. 
En route to the proof of Proposition \ref{prop:cdf_noisy_known}, we establish two helper lemmas. Specifically, Lemma \ref{lem:mean_difference_tilde} presented in 
Section \ref{sec:known_bias} asserts that the bias of the estimator $\tilde{F}_i$ is small and Lemma \ref{lem:sup_tilde} in Section \ref{sec:known_variance} 
provides a uniform control over the variance of $\tilde{F}_i$. With aid of these two helper lemmas, we prove Proposition \ref{prop:cdf_noisy_known} in 
Section \ref{sec:proof_cdf_known}.


\subsection{Support Lemma to Control the Bias of $\tilde{F}_i$}\label{sec:known_bias}
\begin{lemma}\label{lem:mean_difference_tilde} 
	For $i \in [m]$, let $\tilde{F}_i$ be defined as in \eqref{eqn:ECDF_known_noise}. Then there exists a constant $\constbb = \constbb(\lipmin) > 0$ 
	such that
	\[	\sup_{z \in \Reals}\left| ~ \Exp{\tilde{F}_i (z)} - F_i(z) \right| \leq \constbb \left( \log \left| \cB_i \right| \right)^{-\frac{1}{\beta}},\qquad\forall i \in [m].	\]
\end{lemma}
Note that the expectation in the lemma statement is taken with respect to the randomness in data generation process as described in Section \ref{sec:model}.
	
\begin{proof}
Recall that $F_i$ is the inverse function of a slice $g(\frow{i}, \cdot)$ of the latent function $g$ at the fixed row feature $\frow{i}$ in our model. 
Since $F_i$ is $(\frac{1}{\lipmax}, \frac{1}{\lipmin})$-biLipschitz by the model assumption, it admits probability density $f_i$ such that 
$\frac{1}{\lipmax} \leq f_i (z) \leq \frac{1}{\lipmin}$ for all $z \in \supp f_i$ (and $f_i(z) = 0$ outside the support). Therefore, for all $i \in [m]$, 
$f_i$, the density corresponding to $F_i$ belongs to Fan's density class \cite{Fan1991}
\[	\cC_{m, \alpha, B} = \left\{ f(x): \left| f^{(m)}(x) - f^{(m)}\left( x + \delta \right)\right| \leq B \delta^{\alpha} \right\},		 \]
with $m = 0, \alpha = 0$, and $B = \frac{1}{\lipmin}$. Here, $f^{(m)}$ denotes the $m$-th derivative of $f$.

Therefore, we can conclude that for any $i \in [m]$,
\begin{align*}
	\sup_{z \in \Reals}\Big|~ \Exp{\tilde{F}_i (z)} - F_i(z) \Big| 
		&\stackrel{(a)}{\leq}	\sup_{z \in \Reals} \Exp{\left( \tilde{F}_i (z) - F_i (z) \right)^2}^{\frac{1}{2}}\\
		&\leq \sup_{f \in \cC_{0,0, \frac{1}{\lipmin}}} \sup_{z \in \Reals} \Exp{\left( \tilde{F}_{\left|\cB_i\right|} (z) - F(z) \right)^2}^{\frac{1}{2}}\\
		&\stackrel{(b)}{=} \cO\left( \left( \log \left| \cB_i \right| \right)^{-\frac{1}{\beta}} \right),
\end{align*}
where $\tilde{F}_{\left|\cB_i\right|}$ denotes an estimate of $F$ obtained from $\left|\cB_i\right|$ number of samples. Here, (a) follows from the observation that 
\begin{align*}
	\Big|~ \Exp{\tilde{F}_i (z)} - F_i(z) \Big|	
		&=  \Big|~ \Exp{\tilde{F}_i (z) - F_i (z)} \Big|
		\leq \Exp{\left( \tilde{F}_i (z) - F_i (z) \right)^2}^{\frac{1}{2}}	
\end{align*}
and (b) is the result of Theorem \ref{thm:Fan2} (originally Theorem 3 of \cite{Fan1991}).

Actually the upper bound is uniformly valid over all possible realizations of $\frow{i} \in [0,1]$ because Fan's original result holds uniformly over the whole class 
$\cC_{0,0, \frac{1}{\lipmin}}$. We also observe that the constant hidden in the big O notation is dependent only on the class $\cC_{0,0, \frac{1}{\lipmin}}$, 
hence, only on the model parameter $\lipmin$. Therefore, we can explicitly introduce a constant $\constbb = \constbb(\lipmin)$.
\end{proof}

\subsection{Support Lemmas to Control the Variance of $\tilde{F}_i$}\label{sec:known_variance}
First, we introduce Lemma \ref{lem:concentration_tilde} to control the variance of $\tilde{F}_i$ at a single point and then refine it to Lemma \ref{lem:sup_tilde} 
by the usual $\eps$-net argument to obtain a uniform control over the entire support of $f_i$
	
\begin{lemma}\label{lem:concentration_tilde}
	For $i \in [m]$, let $\tilde{F}_i$ be defined as in \eqref{eqn:ECDF_known_noise} with 
	$h = \left(4\gamma \right)^{\frac{1}{\beta}}\left( \log |\cB_i| \right)^{-\frac{1}{\beta}}$. Then for any $t > 0$, 
	\begin{align*}
		\mathbb{P} \bigg( \left| \tilde{F}_i(z) - \Exp{\tilde{F}_i(z)} \right| \geq t ~ \bigg| ~ | \cB_i | = \numi \bigg)	
			&\leq 2\exp\left( - \frac{ \numi^{\frac{1}{2}} }{2\constcc^2 \left( \log \numi \right)^{\frac{2}{\beta}} } t^2	 \right).
	\end{align*}
\end{lemma}
	
\begin{proof} [Proof of Lemma \ref{lem:concentration_tilde}]
First, we observe that when conditioned on $\frow{i}$, the kernel smoothed ECDF $\tilde{F}_i$ evaluated at $z$ is a function of $\left| \cB_i \right|$ 
independent random variables $\{Z(i, j)\}_{j \in \cB_i}$. That is, when $z$ is fixed, $\tilde{F}_i(z): \Reals^{\left| \cB_i \right|} \to \Reals$ such that
\begin{align*}
	\tilde{F}_i(z)\left[ Z(i,j_1), \ldots, Z(i,j_{\left| \cB_i \right|}) \right] 
		& = \int_{\fmin}^{z \wedge \fmax} \frac{1}{h \left| \cB_i \right|} \sum_{j\in \cB_i} L \left( \frac{w- Z(i,j)}{h} \right) dw,	
\end{align*}
where $L(z) = \frac{1}{2\pi} \int e^{-\img tz} \frac{\phi_K(t)}{\phi_N\left(\frac{t}{h}\right)}dt$ and $h$ is the bandwidth parameter for kernel $K$. 

Next, we show that $\tilde{F}_i(z)$ satisfies the bounded difference condition (see Eq. \eqref{eqn:bounded_difference}). 
Let $\zeta^{n_i} = (\zeta_1, \ldots, \zeta_{n_i})$ and $\zeta^{n_i}_j = (\zeta_1, \ldots, \zeta_j', \ldots, \zeta_{n_i})$ be two 
$n_i$-tuples of real numbers, which differ only at the $j$-th position. Then
\begin{align}
	\Big| \tilde{F}_i(z)[\zeta^{n_i}] - \tilde{F}_i(z)[\zeta^{n_i}_j] \Big|
		&= \Bigg| \frac{1}{h n_i} \int_{\fmin}^{z \wedge \fmax}  L \left( \frac{w- \zeta_j}{h} \right) - L \left( \frac{w- \zeta'_j}{h} \right) dw \Bigg|		\nonumber\\
		&= \Bigg| \frac{1}{h n_i} \int_{\fmin}^{z \wedge \fmax}  \frac{1}{2\pi} \int \Big(e^{-\img t \frac{w- \zeta_j}{h} } - e^{-\img t \frac{w- \zeta'_j}{h} }\Big) 
			\frac{\phi_K(t)}{\phi_N\left(\frac{t}{h}\right)}dt dw \Bigg|		\nonumber\\
		&\leq \frac{1}{2\pi h n_i} \int_{\fmin}^{z \wedge \fmax}  \int \Big|e^{-\img t \frac{w- \zeta_j}{h} } - e^{-\img t \frac{w- \zeta'_j}{h} }\Big| 
			\left|\frac{\phi_K(t)}{\phi_N\left(\frac{t}{h}\right)} \right| dt dw.	\label{eqn:difference_intermediate}
\end{align}
We make three observations to further simplify \eqref{eqn:difference_intermediate}:
\begin{itemize}
\item
Since $\big| e^{-\img tz} \big| = 1$ for any real numbers $t$ and $z$, we have 
\begin{align*}
	\Big|e^{-\img t \frac{w- \zeta_j}{h} } - e^{-\img t \frac{w- \zeta'_j}{h} }\Big| 
		&\leq \Big|e^{-\img t \frac{w- \zeta_j}{h} }\Big| + \Big| e^{-\img t \frac{w- \zeta'_j}{h} }\Big|  = 2.
\end{align*}
\item
Also, we have $\left| \phi_N\left(\frac{t}{h}\right) \right| \geq B^{-1} \exp\left( -\gamma \left|\frac{t}{h}\right|^{\beta}\right)$, from the 
supersmoothness assumption on the noise, cf. \eqref{eqn:model_supersmooth}. 
\item	
Recall that we choose\footnote{In fact, this choice is made following Fan \cite{Fan1991}; see Theorems \ref{thm:Fan1}, \ref{thm:Fan2}.} 
$h = \left(4\gamma \right)^{\frac{1}{\beta}}\left( \log n_i \right)^{-\frac{1}{\beta}}$ in the algorithm description in Section \ref{sec:alg_noisy_known}. 
\end{itemize}

Combining these observations with \eqref{eqn:difference_intermediate}, we have
\begin{align*}
	\Big| \tilde{F}_i(z)[\zeta^{n_i}] - \tilde{F}_i(z)[\zeta^{n_i}_j] \Big|
		&\leq \frac{\left( \log n_i \right)^{\frac{1}{\beta}}}{2\pi \left( 4\gamma \right)^{\frac{1}{\beta}} n_i} \int_{\fmin}^{z \wedge \fmax}  
			\int_{-1}^{1} 2 B \kmax \exp\left( \frac{1}{4} \left|t \right|^{\beta} \log n_i\right) dt dw\\
		&\leq \frac{ B \kmax \left( \log n_i \right)^{\frac{1}{\beta}}}{\pi \left( 4\gamma \right)^{\frac{1}{\beta}} n_i} \int_{\fmin}^{z \wedge \fmax} 
			\left( 1 - (-1) \right) \max_{t \in [-1,1]}\exp\left( \frac{1}{4} \left|t \right|^{\beta} \log n_i\right) dw\\
		&= \frac{ B \kmax \left( \log n_i \right)^{\frac{1}{\beta}}}{\pi \left( 4\gamma \right)^{\frac{1}{\beta}} n_i}  
			\left( \left(z  \wedge \fmax \right) - \fmin \right) 2 n_i^{\frac{1}{4}}\\
		&\leq  \frac{ 2B \kmax (\fmax - \fmin) \left( \log n_i \right)^{\frac{1}{\beta}}}{\pi \left( 4\gamma \right)^{\frac{1}{\beta}} n_i^{\frac{3}{4}}}\\
		&=  \frac{ 2\constdd \left( \log n_i \right)^{\frac{1}{\beta}}}{n_i^{\frac{3}{4}}},
\end{align*}
for any $z \in [\fmin, \fmax]$. In other words, the bounded difference condition is established for any fixed $z \in [\fmin, \fmax]$.

Applying McDiarmid's inequality (Lemma \ref{lem:McDiarmid}), we can conclude that for any $t > 0$,
\begin{align*}
	\mathbb{P} \bigg( \left| \tilde{F}_i(z)[\zeta^{n_i}] - \bbE_{\zeta^{n_i}}{\tilde{F}_i(z)[\zeta^{n_i}]} \right| \geq t \bigg)
		&\leq 2\exp\left( - \frac{ n_i^{\frac{1}{2}} }{2\constcc^2 \left( \log n_i \right)^{\frac{2}{\beta}}} t^2	 \right).		
\end{align*}
\end{proof}

We want to uniformly control the variance over all $z \in [\fmin, \fmax]$. Applying the $\varepsilon$-net argument, we obtain the following lemma 
as a corollary of Lemma \ref{lem:concentration_tilde}. For succinct representation of the result, we define a function $\textrm{Res}: [n] \to \Reals$ as
\begin{equation}\label{eqn:resolution}
	\res{k} = \constcc  k^{\frac{1}{4}} \left( \log k \right)^{\frac{1}{\beta}}.
\end{equation}

\begin{lemma}\label{lem:sup_tilde}
	For $i \in [m]$, let $\tilde{F}_i$ be defined as in \eqref{eqn:ECDF_known_noise} with 
	$h = \left(4\gamma \right)^{\frac{1}{\beta}}\left( \log |\cB_i| \right)^{-\frac{1}{\beta}}$. Then for any positive integer 
	$\Nnet$ and for any $t > 0$,
	\begin{align*}
		\Prob{ \sup_{z \in [\fmin, \fmax]} \left|~ \tilde{F}_i(z) - \Exp{\tilde{F}_i(z)} ~\right| \geq t  + \frac{\res{\numi}}{\Nnet} ~ \bigg| ~ |\cB_i | = \numi }
		&\leq 	 2 \Nnet \exp\left( - \frac{ \numi^{\frac{1}{2}}  }{2 \constcc^2 \left( \log \numi \right)^{\frac{2}{\beta}} } t^2 \right).
\end{align*}
\end{lemma}

\begin{proof} [Proof of Lemma \ref{lem:sup_tilde}]
First, we discretize the interval interval $[\fmin, \fmax]$ by constructing an $\varepsilon$-net. For any integer $\Nnet \geq 1$, define the set
\[	\cT_{\Nnet} := \left\{ \fmin + \frac{2k - 1}{2 \Nnet} ( \fmax - \fmin ),~\forall k \in [\Nnet] \right\}.	\]
Then for any $\Nnet > 0$, $\cT_{\Nnet} \subset [\fmin, \fmax]$ and it forms a $\frac{( \fmax - \fmin )}{2\Nnet}$-net with $\left| \cT_{\Nnet} \right| = \Nnet$, 
i.e., for any $z \in \left[\fmin, \fmax\right]$, there exists $k \in [\Nnet]$ such that $\left| z - \frac{2k - 1}{2\Nnet} ( \fmax - \fmin ) \right| 
\leq \frac{\left( \fmax - \fmin \right)}{2\Nnet}$.

Next, we observe that
\begin{align*}
	\left\| \tilde{f}_i \right\|_{\infty} 	
		&= \bigg\| \frac{1}{h \left| \cB_i \right|} \sum_{j \in \cB_i} L \left( \frac{z - Z(i,j)}{h} \right) \bigg\|_{\infty} \\
		&\leq \frac{1}{h} \left\| L \right\|_{\infty}\\
		&= \frac{1}{2\pi h} \left\| \int_{-\infty}^{\infty} e^{-\img tz} \frac{\phi_K(t)}{\phi_N\left( \frac{t}{h} \right)} dt  \right\|_{\infty}
		\leq \frac{1}{2\pi h} \int_{-\infty}^{\infty} \left| e^{-\img tz} \frac{\phi_K(t)}{\phi_N\left( \frac{t}{h} \right)} \right| dt\\
		&\leq \frac{1}{2\pi h} \int_{-1}^1 \frac{\kmax}{B^{-1} \exp \left(-\gamma \left| \frac{t}{h} \right|^{\beta} \right)} dt
		\leq \frac{B \kmax \left( \log \left| \cB_i \right| \right)^{\frac{1}{\beta}}}{2\pi \left( 4\gamma \right)^{\frac{1}{\beta}} } 
			\int_{-1}^1 \exp\left( \frac{1}{4} \left| t \right|^{\beta} \log \left| \cB_i \right| \right) dt	\\
		&\leq	 \frac{ B \kmax \left( \log \left| \cB_i \right| \right)^{\frac{1}{\beta}}}{2\pi \left( 4\gamma \right)^{\frac{1}{\beta}} } \int_{-1}^1  \left| \cB_i \right|^{\frac{1}{4}} dt\\
		&=	\frac{ B \kmax }{\pi \left( 4\gamma \right)^{\frac{1}{\beta}} }  \left| \cB_i \right|^{\frac{1}{4}} \left( \log \left| \cB_i \right| \right)^{\frac{1}{\beta}} \\
		&= \frac{ \res{|\cB_i|} }{\fmax - \fmin}.
\end{align*}

When conditioned on $|\cB_i| = \numi$, this upper bound is universal for all realization of $N, M$. Therefore, when conditioned on $|\cB_i| = \numi$,
$\big\| \bbE\big[ \tilde{f}_i \big] \big\|_{\infty} \leq \frac{ \res{\numi} }{\fmax - \fmin}$, too. By triangle inequality, 
$\big\| \tilde{f}_i - \bbE[ \tilde{f}_i ] \big\|_{\infty} \leq \frac{ 2 \res{\numi} }{\fmax - \fmin}$ 
and it follows from the definition of $\tilde{F}_i$ (see \eqref{eqn:ECDF_known_noise}) that 
\begin{align*}
	\sup_{z \in [\fmin, \fmax]} \bigg| \tilde{F}_i(z) - \Exp{\tilde{F}_i(z)} \bigg| 
		&\leq \sup_{z \in \cT_{\Nnet}} \bigg| \tilde{F}_i(z) - \Exp{\tilde{F}_i(z)} \bigg| +  \frac{ 2 \res{\numi} }{\fmax - \fmin} \frac{ \fmax - \fmin }{ 2 \Nnet}\\
		&= \sup_{z \in \cT_{\Nnet}} \bigg| \tilde{F}_i(z) - \Exp{\tilde{F}_i(z)} \bigg| +  \frac{ \res{\numi} }{ \Nnet}.
\end{align*}
Therefore, if $\left| \tilde{F}_i(z) - \Exp{\tilde{F}_i(z)} \right| \leq \varepsilon$ for all $z \in \cT_n$, the supremum over the whole domain is also bounded above 
by $\eps$, up to an additional discretization error term, $ \frac{ \res{\numi} }{ \Nnet}$. That is to say,  
\[	
	\sup_{z \in \cT_n } \left|~ \tilde{F}_i(z) - \Exp{\tilde{F}_i(z)} ~\right| \leq \varepsilon
	\qquad\implies\qquad
	\sup_{z \in [\fmin, \fmax]} \left|~ \tilde{F}_i(z) - \Exp{\tilde{F}_i(z)} ~\right| \leq \varepsilon + \frac{\res{\numi} }{\Nnet}.	
\]
Applying the union bound on the contraposition of the previous statement yields the conclusion: for any $t > 0$,
\begin{align*}
	\Prob{ \sup_{z \in [\fmin, \fmax]} \left|~ \tilde{F}_i(z) - \Exp{\tilde{F}_i(z)} ~\right| \geq t  + \frac{\res{\numi}}{\Nnet} ~ \bigg| ~ |\cB_i | = \numi }
		&\leq \Prob{ \sup_{z \in \cT_{\Nnet}} \left|~ \tilde{F}_i(z) - \Exp{\tilde{F}_i(z)} ~\right| \geq t }\\
		&\leq \sum_{z \in \cT_{\Nnet}} \Prob{ \left|~ \tilde{F}_i(z) - \Exp{\tilde{F}_i(z)} ~\right| \geq t }\\
		&\leq  2 \Nnet \exp\left( - \frac{ \numi^{\frac{1}{2}} }{2 \constcc^2 \left( \log \numi \right)^{\frac{2}{\beta}} } t^2  \right).
\end{align*}		
\end{proof}

\subsection{Completing the Proof of Proposition \ref{prop:cdf_noisy_known}}\label{sec:proof_cdf_known}

\begin{proof}[Proof of Proposition \ref{prop:cdf_noisy_known}]
We put Lemma \ref{lem:mean_difference_tilde} and Lemma \ref{lem:sup_tilde} together by applying the union bound. Notice that 
\[
	\sup_{z \in [\fmin, \fmax]} \big| \tilde{F}_i (z) - F_i(z) \big|	\leq 
		\sup_{z \in [\fmin, \fmax]} \big| \tilde{F}_i (z) - \bbE[ \tilde{F}_i (z) ] \big| 
		+ \sup_{z \in [\fmin, \fmax]} \big| \bbE[ \tilde{F}_i (z) -  F_i (z) ] \big|
\]
by triangle inequality. Therefore, for any $\delta_1, \delta_2 > 0$, 
\begin{align*}
	&\Prob{ \sup_{z \in [\fmin, \fmax]} \big| \tilde{F}_i (z) - F_i(z) \big| > \delta_1 + \delta_2 ~\Big|~ |\cB_i| = \numi }\\
		&\qquad\leq \Prob{ \sup_{z \in [\fmin, \fmax]} \big| \tilde{F}_i (z) - \bbE[ \tilde{F}_i (z) ] \big|  > \delta_1 ~\Big|~ |\cB_i| = \numi }
		+ \Prob{\sup_{z \in [\fmin, \fmax]} \big| \bbE[ \tilde{F}_i (z) -  F_i (z) ] \big|  > \delta_2 ~\Big|~ |\cB_i| = \numi }.
\end{align*}

Specifically, we choose $\delta_1 = t  + \frac{\res{\numi}}{\Nnet}$ with $\Nnet = \numi^{\frac{1}{4}} \left( \log \numi \right)^{\frac{2}{\beta}}$
and $\delta_2 = \constbb \left( \log \numi \right)^{- \frac{1}{\beta}}$. Note that $\frac{\res{\numi}}{\Nnet} 
= \constcc \left( \log \numi \right)^{-\frac{1}{\beta}}$. Therefore, for any $t > 0$,
\begin{align*}
	&\Prob{\sup_{z \in [\fmin, \fmax]} \big|~ \tilde{F}_i (z) - F_i(z)~ \big| > t + \big(\constbb + \constcc \big) \left( \log \numi \right)^{-\frac{1}{\beta}} ~\bigg| ~ | \cB_i | = \numi }\\
		&\qquad \leq \Prob{ \sup_{z \in [\fmin, \fmax]} \big| \tilde{F}_i (z) - \bbE[ \tilde{F}_i (z) ] \big|  > t  + \constcc \left( \log \numi \right)^{-\frac{1}{\beta}} ~\Big|~ |\cB_i| = \numi }\\
		&\qquad \quad+  \Prob{\sup_{z \in [\fmin, \fmax]} \big| \bbE[ \tilde{F}_i (z) -  F_i (z) ] \big|  >  \constbb \left( \log \numi \right)^{- \frac{1}{\beta}} ~\Big|~ |\cB_i| = \numi }\\
		&\qquad\leq 2 \numi^{\frac{1}{4}} \left( \log \numi \right)^{\frac{2}{\beta}}
			\exp\left( - \frac{ \numi^{\frac{1}{2}} }{2 \constcc^2 \left( \log \numi \right)^{\frac{2}{\beta}} } t^2 \right).
\end{align*}
\end{proof}

\section{Proof of Theorem \ref{thm:cdf_noisy_unknown} }\label{sec:proof_noisy_unknown_CDF}
In this section, we prove Theorem \ref{thm:cdf_noisy_unknown} in a similar fashion as in Section \ref{sec:proof_noisy_known_CDF}. 
For the purpose, we separately control the bias and the variance of $\hat{F}_i$ with Lemmas \ref{lem:bias_hat_conditioned} and \ref{lem:sup_hat}, 
respectively. 

In \ref{sec:bias_unknown}, we present and prove Lemmas \ref{lem:bias_hat_conditioned}. The goal of Lemmas \ref{lem:bias_hat_conditioned} is 
in establishing a uniform upper bound on $\Exp{\hat{F}_i(z) - \tilde{F}_i(z) }$ conditioned on that a reliable estimate of $\phi_N$ is available, 
which is ensured to be the case with high probability by Theorem \ref{thm:ensure_condition}. Then in Section \ref{sec:variance_unknown}, 
we prove Lemma \ref{lem:sup_hat} by the same logic with which we prove Lemma \ref{lem:sup_tilde}, i.e., by refining the concentration inequality 
presented in Lemma \ref{lem:concentration_hat} with the $\eps$-net argument. Lastly, we conclude the section with a proof of Theorem \ref{thm:cdf_noisy_unknown} presented in Section \ref{sec:proof_unknown}.

\subsection{Support Lemmas to Control the Bias of $\hat{F}_i$}\label{sec:bias_unknown}
In this section, we argue that $\bbE \hat{F}_i(z)$ is uniformly close to $\bbE\tilde{F}_i(z)$ over $z \in \Reals$. 
	
\begin{lemma}\label{lem:hat_tilde}
	For $i \in [m]$, let $\tilde{F}_i$ be defined as in \eqref{eqn:ECDF_known_noise} and $\hat{F}_i$ be defined as in 
	\eqref{eqn:ECDF_unknown_noise} with the kernel bandwidth $h$ and the ridge parameter $\rho$. Then
	\begin{equation}\label{eqn:term.upper}
		\bigg| \sup_{z \in \Reals} \Exp{\hat{F}_i(z) - \tilde{F}_i(z) }\bigg|
			\leq	\frac{\kmax(\fmax - \fmin)}{ \pi h } \max_{t \in [-1, 1]} 
			\Bigg| ~\bbE\Bigg[ \frac{ \phi_N(\frac{t}{h}) -\big[ \hat{\phi}_{N,i}(\frac{t}{h}) + \rho \big]}{ \hat{\phi}_{N,i}(\frac{t}{h}) + \rho } \Bigg]~ \Bigg|.
	\end{equation}
\end{lemma}
	
\begin{proof}
First, we observe from the definition of $\tilde{F}_i$ (see \eqref{eqn:ECDF_known_noise}) and $\hat{F}_i$ (see \eqref{eqn:ECDF_unknown_noise}) that given $i \in [m]$,
\begin{align}
	\hat{F}_i(z) - \tilde{F}_i(z)
		&= \int_{\fmin}^{z \wedge \fmax} \hat{f}_i(w) - \tilde{f}_i(w) dw	\nonumber\\
		&= \int_{\fmin}^{z \wedge \fmax} \frac{1}{ h |\cB_i|} \sum_{j \in \cB_i}
			\bigg[ \hat{L} \left( \frac{w - Z(i,j)}{h} \right) - L \left( \frac{w - Z(i,j)}{h} \right) \bigg] dw	\nonumber\\
		&= \frac{1}{2\pi h |\cB_i|} 
		\Bigg( \int_{\fmin}^{z \wedge \fmax} \sum_{j \in \cB_i} \int_{-\infty}^{\infty} e^{-\img t \frac{w - Z(i,j)}{h}} \left[ \frac{\phi_K(t)}{\hat{\phi}_{N,i}(\frac{t}{h}) + \rho} - \frac{\phi_K(t)}{\phi_N(\frac{t}{h}) } \right] dt~dw \Bigg).		\label{eqn:term.integral}
\end{align}

Now we let $\Flatent = \{ \frow{1}, \ldots, \frow{m}, \fcol{1}, \ldots, \fcol{n} \}$ denote the latent variables and consider the expectation of 
\eqref{eqn:term.integral}. Note that we can exchange the order of integrals and the expectation in \eqref{eqn:term.integral} because the support of 
$\phi_K$ is contained in $[-1,1]$ and the integrand is a bounded continuous function:
\begin{align}
	&\Exp{ \hat{F}_i(z) - \tilde{F}_i(z)  }\nonumber\\
		&= \frac{1}{2\pi h |\cB_i|} 	\int_{\fmin}^{z \wedge \fmax} \bbE\Bigg[ \sum_{j \in \cB_i} \int_{-\infty}^{\infty} 
			e^{-\img t \frac{w - Z(i,j)}{h}} \phi_K(t) \frac{ \phi_N(\frac{t}{h}) -\big[ \hat{\phi}_{N,i}(\frac{t}{h}) + \rho \big]}{\phi_N(\frac{t}{h}) \big[ \hat{\phi}_{N,i}(\frac{t}{h}) + \rho \big]}dt \Bigg] dw	\nonumber\\
		&= \frac{1}{2\pi h |\cB_i|} 	\int_{\fmin}^{z \wedge \fmax} \sum_{j \in \cB_i} \bbE\Bigg[ \int_{-\infty}^{\infty} e^{-\img t \frac{w - Z(i,j)}{h}} \phi_K(t) 
			\frac{ \phi_N(\frac{t}{h}) -\big[ \hat{\phi}_{N,i}(\frac{t}{h}) + \rho \big]}{\phi_N(\frac{t}{h}) \big[ \hat{\phi}_{N,i}(\frac{t}{h}) + \rho \big]}dt  \Bigg] dw	\nonumber\\
		&=\frac{1}{2\pi h |\cB_i|} 	 \int_{\fmin}^{z \wedge \fmax} \sum_{j \in \cB_i} \int_{-\infty}^{\infty} \bbE \Bigg[ e^{-\img t \frac{w - Z(i,j)}{h}} \phi_K(t) 
			\frac{ \phi_N(\frac{t}{h}) -\big[ \hat{\phi}_{N,i}(\frac{t}{h}) + \rho \big]}{\phi_N(\frac{t}{h}) \big[ \hat{\phi}_{N,i}(\frac{t}{h}) + \rho \big]} \Bigg] dt~dw	\nonumber\\
		&=\frac{1}{2\pi h |\cB_i|} 	 \int_{\fmin}^{z \wedge \fmax} \sum_{j \in \cB_i} \int_{-\infty}^{\infty} \bbE_{\Theta} \Bigg[ \bbE \Bigg[ e^{-\img t \frac{w - Z(i,j)}{h}} \phi_K(t) 
			\frac{ \phi_N(\frac{t}{h}) -\big[ \hat{\phi}_{N,i}(\frac{t}{h}) + \rho \big]}{\phi_N(\frac{t}{h}) \big[ \hat{\phi}_{N,i}(\frac{t}{h}) + \rho \big]} ~\bigg|~ \Flatent \Bigg] \Bigg] dt~dw	\nonumber\\
		&\stackrel{(a)}{=} \frac{1}{2\pi h |\cB_i|} 	\int_{\fmin}^{z \wedge \fmax} \sum_{j \in \cB_i} \int_{-\infty}^{\infty} e^{-\img t \frac{w}{h}}\phi_K(t)  
			\bbE_{\Theta} \Bigg[ \bbE\big[ e^{\img \frac{t}{h}Z(i,j)} ~\big|~ \Flatent \big]	
			\bbE\bigg[ \frac{ \phi_N(\frac{t}{h}) -\big[ \hat{\phi}_{N,i}(\frac{t}{h}) + \rho \big]}{\phi_N(\frac{t}{h}) \big[ \hat{\phi}_{N,i}(\frac{t}{h}) + \rho \big]} ~\bigg|~ \Flatent \bigg] \Bigg]dt~dw.	\label{eqn:term.integral2}
\end{align}
Here, (a) follows from the conditional independence\footnote{Observe that $\hatNest(t)$ is a function of 
$\{ N(i',j_1) - N(i',j_2)\}_{(i',j_1, j_2) \in \cT_i}$. By the construction of the set $\cT_i$ described in \eqref{eqn:set_Ti}, $\cT_i$ is conditionally 
independent of $Z(i',j)$ with $i' = i$ when conditioned on $\Flatent$. Therefore, $\hatNest(t)$ is conditionally independent of 
$Z(i,j)$ for any $j \in [n]$, too.} between $Z(i,j)$ and $\hatNest(t)$ when conditioned on $\Theta$.

Now we observe that $Z(i,j) = A(i,j) + N(i,j) = g(\frow{i}, \fcol{j}) + N(i,j)$ when conditioned on $\Theta$. Therefore,
\begin{align}
	\Exp{ e^{\img \frac{t}{h}Z(i,j)} ~\big| ~ \Theta} &= \Exp{ e^{\img \frac{t}{h}A(i,j)}e^{\img \frac{t}{h}N(i,j)} ~\big| ~ \Theta}
		= e^{\img \frac{t}{h}g(\frow{i}, \fcol{j})}	\Exp{e^{\img \frac{t}{h}N(i,j)} ~\big|~\Theta}
		= e^{\img \frac{t}{h}g(\frow{i}, \fcol{j})}	\Exp{e^{\img \frac{t}{h}N(i,j) } }		\nonumber\\
		&= e^{\img \frac{t}{h}g(\frow{i}, \fcol{j})}\phi_N\bigg( \frac{t}{h} \bigg).			\label{eqn:mgf_condition}
\end{align}
By \eqref{eqn:term.integral2} and \eqref{eqn:mgf_condition},
\begin{align}
	\Exp{\hat{F}_i(z) - \tilde{F}_i(z) }
		&= \frac{1}{2\pi h |\cB_i|} 	\int_{\fmin}^{z \wedge \fmax} \sum_{j \in \cB_i} \int_{-\infty}^{\infty} \bbE\Big[ e^{\img t \frac{g(\frow{i}, \fcol{j}) - w}{h}} \Big] \phi_K(t)  \phi_N\Big(\frac{t}{h}\Big)
			\bbE\Bigg[ \frac{ \phi_N(\frac{t}{h}) -\big[ \hat{\phi}_{N,i}(\frac{t}{h}) + \rho \big]}{\phi_N(\frac{t}{h}) \big[ \hat{\phi}_{N,i}(\frac{t}{h}) + \rho \big]} \Bigg] dt~dw	\nonumber\\
		&= \frac{1}{2\pi h |\cB_i|} 	\int_{\fmin}^{z \wedge \fmax} \sum_{j \in \cB_i} \int_{-\infty}^{\infty} \bbE \Big[ e^{\img t \frac{g(\frow{i}, \fcol{j}) - w}{h}} \Big] \phi_K(t) 
			\bbE\Bigg[ \frac{ \phi_N(\frac{t}{h}) -\big[ \hat{\phi}_{N,i}(\frac{t}{h}) + \rho \big]}{ \hat{\phi}_{N,i}(\frac{t}{h}) + \rho } \Bigg] dt~dw.	\label{eqn:bias_hat_tilde_intermediate}
\end{align}
Next, we take the supremum of $\Exp{\hat{F}_i(z) - \tilde{F}_i(z) }$ over $z \in \Reals$ to obtain
\begin{align*}
	&\bigg| \sup_{z \in \Reals} \Exp{\hat{F}_i(z) - \tilde{F}_i(z) } \bigg|	\\
	&\leq	 \frac{1}{2\pi h}	\Bigg| \sup_{z \in \Reals}   \int_{\fmin}^{z \wedge \fmax}  \frac{1}{ |\cB_i|} \sum_{j \in \cB_i}  \int_{-\infty}^{\infty} \bbE\Big[e^{\img t \frac{g(\frow{i}, \fcol{j}) - w}{h}} \Big] \phi_K(t)
		\bbE\Bigg[ \frac{ \phi_N(\frac{t}{h}) -\big[ \hat{\phi}_{N,i}(\frac{t}{h}) + \rho \big]}{ \hat{\phi}_{N,i}(\frac{t}{h}) + \rho } \Bigg] dt~dw \Bigg|		\\
	&\leq	\frac{\fmax - \fmin}{ 2 \pi h } \max_{j \in \cB_i} \int_{-\infty}^{\infty} \Bigg| \bbE \Big[ e^{\img t \frac{g(\frow{i}, \fcol{j}) - w}{h}} \Big] \phi_K(t) 
		\bbE \Bigg[ \frac{ \phi_N(\frac{t}{h}) -\big[ \hat{\phi}_{N,i}(\frac{t}{h}) + \rho \big]}{ \hat{\phi}_{N,i}(\frac{t}{h}) + \rho } \Bigg] \Bigg| dt		\\
	&\leq	\frac{\fmax - \fmin}{ 2 \pi h } \int_{-\infty}^{\infty} \max_{j \in \cB_i}  \bigg| \bbE\Big[e^{\img t \frac{g(\frow{i}, \fcol{j}) - w}{h}}\Big] \bigg| \big| \phi_K(t) \big| 
		\Bigg| \bbE\Bigg[ \frac{ \phi_N(\frac{t}{h}) -\big[ \hat{\phi}_{N,i}(\frac{t}{h}) + \rho \big]}{ \hat{\phi}_{N,i}(\frac{t}{h}) + \rho } \Bigg] \Bigg| dt		\\
	&\leq	\frac{\fmax - \fmin}{ 2 \pi h } \int_{-\infty}^{\infty} \max_{j \in \cB_i}  \bbE\bigg[  \Big|e^{\img t \frac{g(\frow{i}, \fcol{j}) - w}{h}}\Big| \bigg] \big| \phi_K(t) \big| 
		\Bigg| \bbE\Bigg[ \frac{ \phi_N(\frac{t}{h}) -\big[ \hat{\phi}_{N,i}(\frac{t}{h}) + \rho \big]}{ \hat{\phi}_{N,i}(\frac{t}{h}) + \rho } \Bigg] \Bigg| dt		\\
	&\stackrel{(a)}{\leq}	\frac{\fmax - \fmin}{ 2 \pi h } \int_{-\infty}^{\infty} \big| \phi_K(t) \big| 
		\Bigg| \bbE\Bigg[ \frac{ \phi_N(\frac{t}{h}) -\big[ \hat{\phi}_{N,i}(\frac{t}{h}) + \rho \big]}{ \hat{\phi}_{N,i}(\frac{t}{h}) + \rho } \Bigg] \Bigg| dt		\\
	&\stackrel{(b)}{\leq}	\frac{\fmax - \fmin}{ 2 \pi h } \int_{-1}^{1} \kmax \Bigg| \bbE\Bigg[ \frac{ \phi_N(\frac{t}{h}) -\big[ \hat{\phi}_{N,i}(\frac{t}{h}) + \rho \big]}{ \hat{\phi}_{N,i}(\frac{t}{h}) + \rho } \Bigg] \Bigg| dt			\\
	&\leq	\frac{\kmax(\fmax - \fmin)}{ \pi h } \max_{t \in [-1, 1]} \Bigg| \bbE\Bigg[ \frac{ \phi_N(\frac{t}{h}) -\big[ \hat{\phi}_{N,i}(\frac{t}{h}) + \rho \big]}{ \hat{\phi}_{N,i}(\frac{t}{h}) + \rho } \Bigg] \Bigg|.	
\end{align*}
Here, (a) follows from that $ \Big| e^{\img t \frac{g(\frow{i}, \fcol{j}) - w}{h}} \Big| \leq 1$; and (b) follows from the assumption that 
$\phi_K(t) = 0$ for $t \not\in [-1, 1]$.
\end{proof}

\begin{lemma}\label{lem:bias_hat_conditioned}
	Given $i \in [m]$, let $\tilde{F}_i$ be defined as in \eqref{eqn:ECDF_known_noise} and $\hat{F}_i$ be defined as in 
	\eqref{eqn:ECDF_unknown_noise} with the kernel bandwidth $h = (4\gamma)^{\frac{1}{\beta}}(\log |\cB_i|)^{-\frac{1}{\beta}}$ 
	and the ridge parameter $\rho = \frac{1}{B} |\cB_i|^{-\frac{9}{20}}$. If $|\cB_i| \geq 1024$ and $mp$ and $n$ are 
	sufficiently large so that $\valphiAA + \valphiBB \leq \frac{1}{B} |\cB_i|^{-\frac{9}{20}}$, then
	\begin{align*}
		\bbP\Bigg( \bigg| \sup_{z \in \Reals} \Exp{\hat{F}_i(z) - \tilde{F}_i(z) }\bigg|
			> 4\constcc \frac{(\log |\cB_i|)^{\frac{1}{\beta}}}{|\cB_i|^{\frac{1}{5}}} 
		~\bigg|~ \Ephi\Bigg)	 = 0.
	\end{align*}
\end{lemma}

\begin{proof}
In this proof, we establish an upper bound on the term on the right-hand side of \eqref{eqn:term.upper} in Lemma \ref{lem:hat_tilde}, 
conditioned on the event $\Ephi$. As the first step, we note that
\[	
	\frac{ \phi_N(\frac{t}{h}) -\big[ \hat{\phi}_{N,i}(\frac{t}{h}) + \rho \big]}{ \hat{\phi}_{N,i}(\frac{t}{h}) + \rho } 
	= \frac{ \phi_N(\frac{t}{h}) - \hat{\phi}_{N,i}(\frac{t}{h}) -  \rho }{ \phi_N(\frac{t}{h}) - \big[ \phi_N(\frac{t}{h}) - \hat{\phi}_{N,i}(\frac{t}{h}) - \rho \big] }.
\]
Then we observe from the supersmoothness assumption on the noise (cf. \eqref{eqn:model_supersmooth}) that
\begin{equation}\label{eqn:lower_phiN}	
	\phi_N \left(\frac{t}{h} \right)  \geq \frac{1}{B}\exp\bigg(- \gamma \Big|\frac{t}{h}\Big|^{\beta} \bigg) 
		= \frac{1}{B}\exp \bigg(- \frac{1}{4} t^{\beta} \log |\cB_i| \bigg)
	 	= \frac{1}{B} |\cB_i|^{-\frac{1}{4}t^{\beta}} \geq \frac{1}{B} |\cB_i|^{-\frac{1}{4}},		\qquad\text{for all }t \in [-1, 1].
\end{equation}

Recall from the definition of $\Ephi$ in \eqref{eqn:Ef} that 
\begin{equation}\label{eqn:upper_phi_hphi}
	\max_{t \in [-1, 1]}\bigg| \phi_N \Big(\frac{t}{h} \Big) - \hat{\phi}_{N,i} \Big(\frac{t}{h} \Big) \bigg| \leq \valphi,
\end{equation} 
when conditioned on $\Ephi$. Recall that we have chosen the ridge parameter $\rho = \frac{1}{B} |\cB_i|^{-\frac{9}{20}}$ and 
we assumed that $\valphiAA + \valphiBB \leq  \frac{1}{B} |\cB_i|^{-\frac{9}{20}}$. It follows that when $|\cB_i| \geq 2^{10}$,
\begin{align*}
	\bigg| \phi_N \Big(\frac{t}{h} \Big) \bigg| 
		&\geq \frac{1}{B} |\cB_i|^{-\frac{1}{4}}		\geq \frac{4}{B} | \cB_i|^{-\frac{9}{20}}	\geq 2 \big( \valphiAA + \valphiBB + \rho \big)\\
		&\geq 2 \bigg( \bigg| \phi_N \Big(\frac{t}{h}\Big) - \hat{\phi}_{N,i}\Big(\frac{t}{h}\Big)\bigg| + \rho \bigg)
		\geq 2 \bigg|~ \phi_N \Big(\frac{t}{h}\Big) - \hat{\phi}_{N,i}\Big(\frac{t}{h}\Big) - \rho~ \bigg|,	\qquad\text{ for all } t \in [-1, 1].
\end{align*}
Therefore,
\begin{align}
	\max_{t \in [-1, 1]} \left| \frac{ \phi_N(\frac{t}{h}) - \hat{\phi}_{N,i}(\frac{t}{h}) - \rho}{ \hat{\phi}_{N,i}(\frac{t}{h}) + \rho } \right|
		&\leq \max_{t \in [-1, 1]} \left| \frac{ \phi_N(\frac{t}{h}) - \hat{\phi}_{N,i}(\frac{t}{h}) - \rho}{ \frac{1}{2}\phi_{N}(\frac{t}{h}) } \right|
				\nonumber\\
		&\leq 2 \max_{t \in [-1, 1]} \Bigg( \bigg|\phi_N \Big(\frac{t}{h}\Big) \bigg|^{-1} 
			\bigg|\phi_N \Big(\frac{t}{h}\Big) - \hat{\phi}_{N,i}\Big(\frac{t}{h}\Big) - \rho \bigg| \Bigg)	
				\nonumber\\
		&\leq  2  \max_{t \in [-1, 1]} \Bigg( \bigg|\phi_N \Big(\frac{t}{h}\Big) \bigg|^{-1} \Bigg)
			\Bigg( \max_{t \in [-1, 1]}  \bigg|\phi_N \Big(\frac{t}{h}\Big) - \hat{\phi}_{N,i}\Big(\frac{t}{h}\Big) \bigg| + \rho \Bigg) 
				\nonumber\\
		&\stackrel{(a)}{\leq} 2B |\cB_i|^{\frac{1}{4}} \big( \valphi + \rho \big)		\nonumber\\
		&\stackrel{(b)}{\leq} 4B |\cB_i|^{-\frac{1}{5}}	\label{eqn:term_upper}
\end{align}
when conditioned on $\Ephi$; (a) follows from \eqref{eqn:lower_phiN}, \eqref{eqn:upper_phi_hphi}, and (b) follows from 
the assumption that $\valphi \leq \frac{1}{B}|\cB_i|^{-\frac{9}{20}} = \rho$. We complete the proof by inserting 
\eqref{eqn:term_upper} to \eqref{eqn:term.upper} in Lemma \ref{lem:hat_tilde}.
\end{proof}

\subsection{Support Lemmas to Control the Variance of $\hat{F}_i$}\label{sec:variance_unknown}
	
\begin{lemma}\label{lem:concentration_hat}
	For $i \in [m]$, let $\hat{F}_i$ be defined as in \eqref{eqn:ECDF_unknown_noise} with the kernel bandwidth $h = (4\gamma)^{\frac{1}{\beta}}(\log |\cB_i|)^{-\frac{1}{\beta}}$ 
	and the ridge parameter $\rho = \frac{1}{B} |\cB_i|^{-\frac{9}{20}}$. Then for any $t > 0$, 
	\begin{align*}
		\mathbb{P} \bigg( \left| \hat{F}_i(z) - \Exp{\hat{F}_i(z)} \right| \geq t ~ \bigg| ~ | \cB_i | = \numi \bigg)	
			&\leq 2\exp\left( - \frac{ n_i^{\frac{1}{10}} }{2\constcc^2 \left( \log n_i \right)^{\frac{2}{\beta}}} t^2	 \right).
	\end{align*}
\end{lemma}
	
\begin{proof} [Proof of Lemma \ref{lem:concentration_hat}]
We follow the same logic as in the proof of Lemma \ref{lem:concentration_tilde}. Recall that when conditioned on $\frow{i}$, 
the kernel smoothed ECDF $\hat{F}_i$ evaluated at $z$ is a function of $\left| \cB_i \right|$ independent random variables 
$\{Z(i, j)\}_{j \in \cB_i}$, i.e., when $z$ is fixed, $\hat{F}_i(z): \Reals^{\left| \cB_i \right|} \to \Reals$ such that
\begin{align*}
	\hat{F}_i(z)\left[ Z(i,j_1), \ldots, Z(i,j_{\left| \cB_i \right|}) \right] 
		& = \int_{\fmin}^{z \wedge \fmax} \frac{1}{h \left| \cB_i \right|} \sum_{j\in \cB_i} \hat{L} \left( \frac{w- Z(i,j)}{h} \right) dw,	
\end{align*}
where $\hat{L}(z) = \frac{1}{2\pi} \int e^{-\img tz} \frac{\phi_K(t)}{\hat{\phi}_N\left(\frac{t}{h}\right)+\rho}dt$ and $h$ is the bandwidth parameter for kernel $K$. 

First, we show that $\hat{F}_i(z)$ satisfies the bounded difference condition (see Eq. \eqref{eqn:bounded_difference}). 
Let $\zeta^{n_i} = (\zeta_1, \ldots, \zeta_{n_i})$ and $\zeta^{n_i}_j = (\zeta_1, \ldots, \zeta_j', \ldots, \zeta_{n_i})$ be 
two $n_i$-tuples of real numbers, which differ only at the $j$-th position. Then
\begin{align}
	\Big| \hat{F}_i(z)[\zeta^{n_i}] - \hat{F}_i(z)[\zeta^{n_i}_j] \Big|
		&= \Bigg| \frac{1}{h n_i} \int_{\fmin}^{z \wedge \fmax}  \hat{L} \left( \frac{w- \zeta_j}{h} \right) - \hat{L} \left( \frac{w- \zeta'_j}{h} \right) dw \Bigg|		\nonumber\\
		&= \Bigg| \frac{1}{h n_i} \int_{\fmin}^{z \wedge \fmax}  \frac{1}{2\pi} \int \Big(e^{-\img t \frac{w- \zeta_j}{h} } - e^{-\img t \frac{w- \zeta'_j}{h} }\Big) 
			\frac{\phi_K(t)}{\hat{\phi}_N\left(\frac{t}{h}\right) + \rho}dt dw \Bigg|		\nonumber\\
		&\leq \frac{1}{2\pi h n_i} \int_{\fmin}^{z \wedge \fmax}  \int \Big|e^{-\img t \frac{w- \zeta_j}{h} } - e^{-\img t \frac{w- \zeta'_j}{h} }\Big| 
			\left|\frac{\phi_K(t)}{\hat{\phi}_N\left(\frac{t}{h}\right) + \rho} \right| dt dw.	\label{eqn:difference_intermediate.unknown}
\end{align}

We make three observations to further simplify \eqref{eqn:difference_intermediate.unknown}:
\begin{itemize}
\item
Since $\big| e^{-\img tz} \big| = 1$ for any real numbers $t$ and $z$, we have 
\begin{align*}
	\Big|e^{-\img t \frac{w- \zeta_j}{h} } - e^{-\img t \frac{w- \zeta'_j}{h} }\Big| 
		&\leq \Big|e^{-\img t \frac{w- \zeta_j}{h} }\Big| + \Big| e^{-\img t \frac{w- \zeta'_j}{h} }\Big|  = 2.
\end{align*}
\item
Also, we observe that $\hat{\phi}_N \left( \frac{t}{h} \right) \geq 0$ for all $t$ by definition, and hence, 
$\hat{\phi}_N \left( \frac{t}{h} \right) + \rho \geq \rho = \frac{1}{B} n_i^{-\frac{9}{20}}$.
\item	
Recall that we choose $h = \left(4\gamma \right)^{\frac{1}{\beta}}\left( \log n_i \right)^{-\frac{1}{\beta}}$ in the algorithm description 
in Section \ref{sec:alg_noisy_unknown}
\end{itemize}

Plugging these expresions into \eqref{eqn:difference_intermediate.unknown} leads to
\begin{align*}
	\Big| \hat{F}_i(z)[\zeta^{n_i}] - \hat{F}_i(z)[\zeta^{n_i}_j] \Big|
		&\leq \frac{\left( \log n_i \right)^{\frac{1}{\beta}}}{2\pi \left( 4\gamma \right)^{\frac{1}{\beta}} n_i} \int_{\fmin}^{z \wedge \fmax}  
			\int_{-1}^{1} 2 B \kmax n_i^{\frac{9}{20}} dt dw\\
		&\leq \frac{ B \kmax \left( \log n_i \right)^{\frac{1}{\beta}}}{\pi \left( 4\gamma \right)^{\frac{1}{\beta}} n_i^{\frac{11}{20}}} \int_{\fmin}^{z \wedge \fmax} 
			\left( 1 - (-1) \right)  dw\\
		&\leq \frac{ 2 B \kmax (\fmax - \fmin ) \left( \log n_i \right)^{\frac{1}{\beta}}}{\pi \left( 4\gamma \right)^{\frac{1}{\beta}} n_i^{\frac{11}{20}}}  \\
		&=   \frac{ 2 \constcc \left( \log n_i \right)^{\frac{1}{\beta}}}{ n_i^{\frac{11}{20}}} ,	
			\qquad \text{for any }z \in [\fmin, \fmax].
\end{align*}

Applying McDiarmid's inequality (Lemma \ref{lem:McDiarmid}), we can conclude that,
\begin{align*}
	\mathbb{P} \bigg( \left| \hat{F}_i(z)[\zeta^{n_i}] - \bbE_{\zeta^{n_i}}{\hat{F}_i(z)[\zeta^{n_i}]} \right| \geq t \bigg)
		&\leq 2\exp\left( - \frac{ n_i^{\frac{1}{10}} }{2\constcc^2 \left( \log n_i \right)^{\frac{2}{\beta}}} t^2	 \right).		
\end{align*}
\end{proof}

We want to uniformly control the variance over all $z \in [\fmin, \fmax]$. Applying the $\varepsilon$-net argument, we obtain the following lemma 
as a corollary of Lemma \ref{lem:concentration_hat}. We define $\widehat{\textrm{Res}}: [n] \to \Reals$ in a similar manner as we define $\textrm{Res}$ in \eqref{eqn:resolution} (note that the only difference is in the power of $k$; $\frac{1}{4}$ vs $\frac{9}{20}$):
\begin{equation}\label{eqn:resolution_hat}
	\hres{k} =  \constcc k^{\frac{9}{20}} \left( \log k \right)^{\frac{1}{\beta}}.
\end{equation}

\begin{lemma}\label{lem:sup_hat}
	For $i \in [m]$, let $\hat{F}_i$ be defined as in \eqref{eqn:ECDF_unknown_noise} with the kernel bandwidth 
	$h = (4\gamma)^{\frac{1}{\beta}}(\log |\cB_i|)^{-\frac{1}{\beta}}$ and the ridge parameter 
	$\rho = \frac{1}{B} |\cB_i|^{-\frac{9}{20}}$. Then for any positive integer $\Nnet$ and for any $t > 0$,
	\begin{align*}
		\bbP\bigg( \sup_{z \in [\fmin, \fmax]} \left| \hat{F}_i(z) - \Exp{\hat{F}_i(z)} \right| \geq t 
			+  \frac{\hres{\numi}}{\Nnet} ~ \bigg| ~ | \cB_i | = \numi \bigg)
			&\leq 	2 \Nnet \exp\left( - \frac{ \numi^{\frac{1}{10}} }{2 \constcc^2 \left( \log \numi \right)^{\frac{2}{\beta}} }  t^2 \right).
	\end{align*}
\end{lemma}

\begin{proof} [Proof of Lemma \ref{lem:sup_hat}]
The following proof has the same structure as in the proof of Lemma \ref{lem:sup_hat}]. For any given positive integer 
$\Nnet \geq 1$, define the set
\[	\cT_{\Nnet} := \left\{ \fmin + \frac{2k - 1}{2 \Nnet}\left( \fmax - \fmin \right),~\forall k \in [\Nnet] \right\}.	\]
Then for any $\Nnet > 0$, $\cT_{\Nnet} \subset [\fmin, \fmax]$ and it forms a $\frac{\big( \fmax - \fmin \big)}{2\Nnet}$-net with $\left| \cT_{\Nnet} \right| = \Nnet$, 
i.e., for any $z \in \left[\fmin, \fmax\right]$, there exists $k \in [\Nnet]$ such that $\left| z - \frac{2k - 1}{2\Nnet}\left( \fmax - \fmin \right) \right| 
\leq \frac{\left( \fmax - \fmin \right)}{2\Nnet}$.

Next, we observe that
\begin{align*}
	\left\| \hat{f}_i \right\|_{\infty} 	
		&= \bigg\|\frac{1}{h \left| \cB_i \right|} \sum_{j \in \cB_i} \hat{L} \left( \frac{z - Z(i,j)}{h} \right) \bigg\|_{\infty}\\
		&\leq \frac{1}{h} \big\| \hat{L} \big\|_{\infty}
		= \frac{1}{2\pi h} \bigg\| \int_{-\infty}^{\infty} e^{-\img tz} \frac{\phi_K(t)}{\hat{\phi}_{N,i}\left( \frac{t}{h} \right) + \rho } dt  \bigg\|_{\infty}\\
		&\leq \frac{1}{2\pi h} \int_{-\infty}^{\infty} \big| e^{-\img tz}\big| \bigg| \frac{\phi_K(t)}{\hat{\phi}_{N,i}\left( \frac{t}{h} \right) + \rho } \bigg| dt\\
		&\stackrel{(a)}{\leq} \frac{1}{2\pi h} \int_{-\infty}^{\infty} \bigg| \frac{\phi_K(t)}{\rho } \bigg| dt		\\
		&\stackrel{(b)}{\leq} \frac{1}{2\pi h} \int_{-1}^1 B \kmax \left| \cB_i \right|^{\frac{9}{20}} dt\\%
		&\stackrel{(c)}{\leq} \frac{ \left( \log \left| \cB_i \right| \right)^{\frac{1}{\beta}}}{2\pi \left( 4\gamma \right)^{\frac{1}{\beta}} } \int_{-1}^1 B \kmax \left| \cB_i \right|^{\frac{9}{20}}  dt	\\
		&\leq \frac{B \kmax }{\pi \left( 4\gamma \right)^{\frac{1}{\beta}} }  \left| \cB_i \right|^{\frac{9}{20}} \left( \log \left| \cB_i \right| \right)^{\frac{1}{\beta}} \\
		&= \frac{ \hres{ |\cB_i| } }{ \fmax - \fmin }.
\end{align*}
Here, (a) follows from $\hat{\phi}_{N,i}\left( \frac{t}{h} \right) \geq 0$; (b) is the result of $\rho = \frac{1}{B} |\cB_i|^{-\frac{9}{20}}$; and 
(c) follows from the choice $h = \left( 4\gamma\right)^{\frac{1}{\beta}} \left( \log \left| \cB_i \right| \right)^{-\frac{1}{\beta}}$.

When conditioned on $|\cB_i| = \numi$, this upper bound is universal for all realization of $N, M$. Therefore, when conditioned 
on $|\cB_i| = \numi$, $\left\| \bbE\big[ \hat{f}_i \big] \right\|_{\infty} \leq \frac{ \hres{\numi} }{\fmax - \fmin}$, too. By triangle inequality, 
$\left\| \hat{f}_i - \Exp{ \hat{f}_i } \right\|_{\infty} \leq \frac{ 2 \hres{\numi} }{\fmax - \fmin}$ and it follows from the definition of 
$\hat{F}_i$ (see \eqref{eqn:ECDF_unknown_noise}) that 
\[	\sup_{z \in [\fmin, \fmax]} \bigg| \hat{F}_i(z) - \Exp{\hat{F}_i(z)} \bigg| 
	\leq \sup_{z \in \cT_{\Nnet}} \bigg| \hat{F}_i(z) - \Exp{\hat{F}_i(z)} \bigg| +  \frac{ 2 \hres{\numi} }{\fmax - \fmin }\frac{ \fmax - \fmin }{2 \Nnet}.	\]

Therefore, if $\left| \tilde{F}_i(z) - \Exp{\tilde{F}_i(z)} \right| \leq \varepsilon$ for all $z \in \cT_n$, the supremum over the whole domain 
is also bounded above by $\eps$, up to an additional discretization error term, $ \frac{ \hres{\numi} }{ \Nnet}$. That is to say,  
\[	
	\sup_{z \in \cT_n } \left|~ \hat{F}_i(z) - \Exp{\hat{F}_i(z)} ~\right| \leq \varepsilon
	\qquad\implies\qquad
	\sup_{z \in [\fmin, \fmax]} \left|~ \hat{F}_i(z) - \Exp{\hat{F}_i(z)} ~\right| \leq \varepsilon + \frac{\hres{\numi} }{\Nnet}.	
\]
Applying the union bound on the contraposition of the previous statement yields the conclusion: for any $t > 0$,
\begin{align*}
	\Prob{ \sup_{z \in [\fmin, \fmax]} \left|~ \hat{F}_i(z) - \Exp{\hat{F}_i(z)} ~\right| \geq t  + \frac{\hres{\numi}}{\Nnet} ~ \bigg| ~ |\cB_i | = \numi }
		&\leq \Prob{ \sup_{z \in \cT_{\Nnet}} \left|~ \hat{F}_i(z) - \Exp{\hat{F}_i(z)} ~\right| \geq t }\\
		&\leq \sum_{z \in \cT_{\Nnet}} \Prob{ \left|~ \hat{F}_i(z) - \Exp{\hat{F}_i(z)} ~\right| \geq t }\\
		&\leq   2 \Nnet \exp\left( - \frac{ \numi^{\frac{1}{10}} }{2 \constcc^2 \left( \log \numi \right)^{\frac{2}{\beta}} }  t^2 \right).
\end{align*}		
	
\end{proof}

\subsection{Completing the Proof of Theorem \ref{thm:cdf_noisy_unknown}}\label{sec:proof_unknown}

\begin{proof}[Proof of Theorem \ref{thm:cdf_noisy_unknown}]
First of all, we observe that $\hat{F}_i(z) = F_i(z) = 0$ for all $z \leq \fmin$ and $\hat{F}_i(z) = F_i(z) = 1$ for all $z \geq \fmax$. 
Therefore,
\[	
	\sup_{z \in \Reals} \big| \hat{F}_i (z) - F_i(z) \big| = \sup_{z \in [\fmin, \fmax] } \big| \hat{F}_i (z) - F_i(z) \big|.	
\]
By the usual trick of subtracting and adding the same term (and then applying triangle inequality), we have
\begin{align}
	&\sup_{z \in [\fmin, \fmax] } \big| \hat{F}_i (z) - F_i(z) \big|	\label{eqn:triangle}\\
		&\qquad\leq \sup_{z \in [\fmin, \fmax] } \big| \hat{F}_i (z) - \bbE[ \hat{F}_i (z) ] \big|
			+ \sup_{z \in [\fmin, \fmax] } \big| \bbE\big[ \hat{F}_i (z) - \tilde{F}_i(z) \big] \big|
			+ \sup_{z \in [\fmin, \fmax] } \big| \bbE\big[ \tilde{F}_i(z) \big] - F_i(z) \big|.	\nonumber
\end{align}
Applying the union bound, the following inequality follows from \eqref{eqn:triangle}. For any $t \geq 0$ and any $s_1, s_2, s_3 \geq 0$,
\begin{align}
	\bbP\bigg( \sup_{z \in [\fmin, \fmax] } \big| \hat{F}_i (z) - F_i(z) \big| > t + s_1 + s_2 + s_3 	~\Big|~ \Ephi	\bigg)
		&\leq 	\bbP\bigg( \sup_{z \in [\fmin, \fmax] } \big| \hat{F}_i (z) - \bbE\big[ \hat{F}_i (z) \big] \big| > t + s_1 	~\Big|~ \Ephi	 \bigg)		\label{eqn:hat_term.1}\\
		&\quad + 	\bbP\bigg( \sup_{z \in [\fmin, \fmax] } \big| \bbE\big[ \hat{F}_i (z) - \tilde{F}_i(z) \big] \big| > s_2 	~\Big|~ \Ephi	 \bigg)		\label{eqn:hat_term.2}\\
		&\quad + 	\bbP\bigg( \sup_{z \in [\fmin, \fmax] } \big| \bbE\big[ \tilde{F}_i(z) \big] - F_i(z) \big| > s_3 	~\Big|~ \Ephi	 \bigg).			\label{eqn:hat_term.3}
\end{align}
In the rest of the proof, we establish upper bounds on the three terms in \eqref{eqn:hat_term.1}, \eqref{eqn:hat_term.2}, and \eqref{eqn:hat_term.3} separately.

Specifically, given $i \in [m]$, we let 
\begin{align*}
	s_1 &= \constcc  \big( \log |\cB_i| \big)^{-\frac{1}{\beta}},\\
	s_2 &= 4\constcc  \frac{\big( \log |\cB_i| \big)^{\frac{1}{\beta}}}{ |\cB_i|^{\frac{1}{5}} },\\
	s_3 &= \constbb  \big( \log |\cB_i| \big)^{-\frac{1}{\beta}}.
\end{align*}

\begin{itemize}
\item 
	To find an upper bound on \eqref{eqn:hat_term.1}, we let $\Nnet = |\cB_i|^{\frac{9}{20}} \big( \log |\cB_i| \big)^{\frac{2}{\beta}}$ 
	and observe that $s_1 =  \frac{ \hres{|\cB_i|}}{\Nnet}$; see \eqref{eqn:resolution_hat} for the definition of $\hres{k}$. 
	Then it follows from Lemma \ref{lem:sup_hat} that for any $t > 0$,
	\begin{equation}\label{eqn:hat_term.1a}
		\bbP\bigg( \sup_{z \in [\fmin, \fmax]} \left| \hat{F}_i(z) - \Exp{\hat{F}_i(z)} \right| \geq t + s_1 	~\Big|~ \Ephi	 \bigg)
			\leq 	2 |\cB_i|^{\frac{9}{20}} \big( \log |\cB_i| \big)^{\frac{2}{\beta}}
				\exp\left( - \frac{ |\cB_i|^{\frac{1}{10}} }{2 \constcc^2 \left( \log |\cB_i| \right)^{\frac{2}{\beta}} } t^2  \right).
	\end{equation}
	Note that the concentration argument in the proof of Lemma \ref{lem:sup_hat} is valid regardless of conditioning on $\Ephi$ 
	and therefore, we obtain the same probabilistic tail bound whether we condition on $\Ephi$ or not.

\item 
	Next, it follows from Lemma \ref{lem:bias_hat_conditioned} that
	\begin{equation}\label{eqn:hat_term.2a}
		\bbP\bigg( \sup_{z \in [\fmin, \fmax] } \Big|  \bbE\big[ \hat{F}_i(z) - \tilde{F}_i(z) \big] \Big| > s_2 ~\Big|~ \Ephi \bigg) = 0,
	\end{equation}
	which establishes an upper bound on \eqref{eqn:hat_term.2}.

\item 
	Lastly, it follows from Lemma \ref{lem:mean_difference_tilde} that
	\begin{align}
		\bbP\bigg( \sup_{z \in [\fmin, \fmax] }  \big| \bbE\big[ \tilde{F}_i(z) \big] - F_i(z) \big| > s_3  ~\Big|~ \Ephi \bigg) 
			&= \bbP\bigg( \sup_{z \in [\fmin, \fmax] }  \big| \bbE\big[ \tilde{F}_i(z) \big] - F_i(z) \big| > s_3  \bigg)	\nonumber\\
			&= 0	\label{eqn:hat_term.3a}
	\end{align}
	where the first equality is the result of the independence between $\sup_{z \in [\fmin, \fmax] }  \big| \bbE\big[ \tilde{F}_i(z) \big] - F_i(z) \big|$ and $\Ephi$.
\end{itemize}
We conclude the proof by inserting \eqref{eqn:hat_term.1a}, \eqref{eqn:hat_term.2a}, \eqref{eqn:hat_term.3a} to 
\eqref{eqn:hat_term.1}, \eqref{eqn:hat_term.2}, \eqref{eqn:hat_term.3}.

\end{proof}

\section{Proof of Theorem \ref{thm:ensure_condition} }\label{sec:proof_ensure}
In this section, we prove Theorem \ref{thm:ensure_condition} to ensure that $\hatNest (t)$ is a good estimator of $\phi_N(t)$ 
for all $t \in [-\frac{1}{h}, \frac{1}{h}]$, i.e., $\big| \hatNest (t) - \phi_N(t) \big|$ is uniformly small over the interval $[-\frac{1}{h}, \frac{1}{h}]$, 
with high probability. Our goal is in establishing an upper bound on the `failure' probability, $\Prob{\Ephi}$.

As the first step to the proof of Theorem \ref{thm:ensure_condition}, we define some ancillary events for conditioning in Section \ref{sec:definitions}. 
Then we present support lemmas (Lemmas \ref{lem:event_cols} - \ref{lem:event_Ephi}) to ensure those events are (conditionally) high-probability events 
in Section \ref{sec:support_lemmas}, with their proofs being postponed to Section \ref{sec:deferred_proof_support_lemmas}. 
Combining the support lemmas, we complete our proof of Theorem \ref{thm:ensure_condition} in Section \ref{sec:complete_proof_ensure}. 
The proof is based on the law of total probability.

\subsection{Definition of Ancillary Events}\label{sec:definitions}
We define some events to be used in our analysis. Recall that $(m,n)$ is the problem size, $p, \sigma, \lipmin, \lipmax$ 
are model parameters, and $J, \cT, \cT_i$ are the sets defined in Section \ref{sec:alg_noisy_unknown} to estimate $\phi_N$, cf. 
Algorithm \ref{alg:setT_main} and \eqref{eqn:set_Ti}.
we define the following six events\footnote{Note that we define $\EcTi$ and $\Ephi$ for each $i \in [m]$, while all the other events are defined 
without dependence on $i \in [m]$.}:
\begin{align}
	&\Ecols := \bigg\{  \min_{j' \in [n]} | \cB^{j'} |  \geq \frac{mp}{2} \bigg\},		\label{eqn:Ecols}\\
	&\EJ  := \bigg\{ |J| \geq \frac{1}{4} n \bigg\},			\label{eqn:Ea}\\
	&\EcTi := \bigg\{ |\cT_i| \geq \frac{1}{256} mnp \bigg\},	\label{eqn:Eb}\\
	&\EdA := \bigg\{ \max_{ (i',j_1, j_2) \in \cT} \big| A(i',j_1) - A(i',j_2) \big| 
		\leq 	  \lipmax \bigg( \frac{32 \sqrt{ \pi} \constaa}{\sqrt{ mp } } + \frac{2\sqrt{2}}{\sqrt{ n p}} + 8\sqrt{\frac{\log n}{n}} \bigg) \bigg\},	\label{eqn:Ed}\\
	&\EdN := \bigg\{ \max_{ (i',j_1, j_2) \in \cT} \big| N(i',j_1) - N(i',j_2) \big| \leq 	 8\sigma \sqrt{\log ( mn ) }  \bigg\}.			\label{eqn:Ee}
\end{align}

\subsection{Technical Lemmas to Support the Proof of Theorem \ref{thm:ensure_condition}}\label{sec:support_lemmas}

\begin{lemma}\label{lem:event_cols}
Let $\Ecols$ denote the event as defined in \eqref{eqn:Ecols}. Then
\begin{equation*}
	\Prob{\Ecols^c} 	\leq	n \exp \left( - \frac{mp}{8} \right).	
\end{equation*}
\end{lemma}
The proof of Lemma \ref{lem:event_cols} is deferred to Section \ref{sec:proof_E1}.

\begin{lemma}\label{lem:event_J}
Let $\EJ$ denote the event as defined in \eqref{eqn:Ea}. If $p \geq \frac{8 \log 2}{m}$, then
\begin{equation*}
	\Prob{\EJ^c}  \leq \exp\left( - \frac{n}{16} \right).
\end{equation*}
\end{lemma}
The proof of Lemma \ref{lem:event_J} is deferred to Section \ref{sec:proof_E2}.

\begin{lemma}\label{lem:event_Ti}
Let $\EJ, \EcTi$ denote the events as defined in \eqref{eqn:Ea}, \eqref{eqn:Eb}. If $m \geq 8$ and $np \geq 8(1+ \sqrt{3})^2$, then for any $i \in [m]$,
\begin{equation*}
	\Prob{\EcTi^c ~\big|~ \EJ}
		\leq \exp\left( - \frac{m}{16} \right).	
\end{equation*}
\end{lemma}
The proof of Lemma \ref{lem:event_Ti} is deferred to Section \ref{sec:proof_E3}.

\begin{lemma}\label{lem:event_dA}
Let $\Ecols, \EJ, \EdA$ denote the events as defined in \eqref{eqn:Ecols}, \eqref{eqn:Ea}, \eqref{eqn:Ed}. Then
\begin{equation*}
	\Prob{ \EdA^c ~\big|~ \Ecols \cap \EJ  }	\leq \frac{3}{ n^{7}}.
\end{equation*}
\end{lemma}
The proof of Lemma \ref{lem:event_dA} is deferred to Section \ref{sec:proof_E4}.

\begin{lemma}\label{lem:event_dN}
Let $\EdN$ denote the events as defined in \eqref{eqn:Ee}. Then 
\begin{equation*}
	\Prob{\EdN^c  }	 \leq   \frac{2}{m^7n^7}.
\end{equation*}
\end{lemma}
The proof of Lemma \ref{lem:event_dN} is deferred to Section \ref{sec:proof_E5}.


\begin{lemma}\label{lem:event_Ephi}
Let $\EcTi, \EdA, \EdN, \Ephi$ denote the events as defined in \eqref{eqn:Eb}, \eqref{eqn:Ed}, \eqref{eqn:Ee}, \eqref{eqn:Ef}. Then for all $i \in [m]$,
\begin{equation*}
	\Prob{ \Ephi^c ~\big|~ \EcTi \cap \EdA \cap \EdN}	\leq \frac{4}{m^7n^7}.
\end{equation*}
\end{lemma}
The proof of Lemma \ref{lem:event_Ephi} is deferred to Section \ref{sec:proof_E6}.

\subsection{Completing the Proof of Theorem \ref{thm:ensure_condition}}\label{sec:complete_proof_ensure}

\begin{proof}[Proof of Theorem \ref{thm:ensure_condition}]

Now it remains to find an upper bound on $ \bbP \big( \Ephi^c \big)$. We observe that
\begin{align}
	\Prob{\Ephi^c} 	&\leq	\Prob{\Ephi^c ~\big|~ \EcTi \cap \EdA \cap \EdN} + \Prob{\EcTi^c \cup \EdA^c \cup \EdN^c}		\nonumber\\
		&= \Prob{\Ephi^c ~\big|~ \EcTi \cap \EdA \cap \EdN} + \Prob{\EcTi^c} + \Prob{ \EdA^c} + \Prob{\EdN^c}.	\label{eqn:prob_E6}
\end{align}
First of all, by Lemma \ref{lem:event_Ephi},
\begin{equation}\label{eqn:prob_E6.1}
	\Prob{\Ephi^c ~\big|~ \EcTi \cap \EdA \cap \EdN} \leq \frac{4}{m^7n^7}.
\end{equation}
Second, we note the following inequality holds by Lemma \ref{lem:event_J} and Lemma \ref{lem:event_Ti}:
\begin{align}
	\Prob{\EcTi^c}	&\leq \Prob{\EcTi^c ~\big|~  \EJ} + \Prob{ \EJ^c}	\nonumber\\
		&\leq \exp\left( - \frac{m}{16} \right) + \exp\left( - \frac{n}{16} \right).	\label{eqn:prob_E6.2}
\end{align}
Third,
\begin{align}
	\Prob{\EdA^c}	&\leq \Prob{\EdA^c ~\big|~  \Ecols \cap \EJ} + \Prob{ \Ecols^c \cup \EJ^c}	\nonumber\\
		&\leq \Prob{\EdA^c ~\big|~  \Ecols \cap \EJ} + \Prob{ \Ecols^c } + \Prob{ \EJ^c }	\nonumber\\
		&\leq n \exp \left( - \frac{mp}{8} \right) + \exp\left( - \frac{n}{16} \right) + \frac{3}{n^7}	\label{eqn:prob_E6.3}
\end{align}
by Lemma \ref{lem:event_cols}, Lemma \ref{lem:event_J} and Lemma \ref{lem:event_dA}.
Lastly, we know from Lemma \ref{lem:event_dN} that
\begin{equation}\label{eqn:prob_E6.4}
	\Prob{\EdN^c  }	 \leq   \frac{2}{m^7n^7}
\end{equation}

Inserting \eqref{eqn:prob_E6.1}, \eqref{eqn:prob_E6.2}, \eqref{eqn:prob_E6.3} and \eqref{eqn:prob_E6.4} to \eqref{eqn:prob_E6}, we conclude
\[	
	\Prob{\Ephi^c} 
		\leq \frac{3}{n^7} + \frac{6}{m^7n^7} + n \exp \left( - \frac{mp}{8} \right) + \exp \left( - \frac{m}{16} \right) + 2 \exp \left( - \frac{n}{16} \right).
\]
\end{proof}

\section{Supplement 1 to the Proof of Theorem \ref{thm:ensure_condition}: \\Deferred Proof of the Support Lemmas from Section \ref{sec:proof_ensure} }\label{sec:deferred_proof_support_lemmas}

\subsection{Proof of Lemma \ref{lem:event_cols}}\label{sec:proof_E1}

\begin{proof}[Proof of Lemma \ref{lem:event_cols}]
Observe that $|\cB^j| = \sum_{i \in [m]} M(i,j)$ is the sum of $n$ independent Bernoulli random variables with parameter $p$. 
It follows from the binomial Chernoff bound that
\[	\Prob	{\big| \cB^{j} \big| < \frac{mp}{2}	} \leq \exp\left( - \frac{mp}{8} \right).	\] 
By definition of $\Ecols$, we obtain the following inequality by applying the union bound:
\begin{align*}
	\Prob{ \Ecols^c }
		&= \Prob{ \exists j. \in [n] \quad\text{such that}\quad | \cB^{j} | < \frac{mp}{2}  }\\
		&\leq \sum_{j \in [n]} \Prob{ | \cB^{j} | < \frac{mp}{2}  }\\
		&\leq n \exp\left( - \frac{mp}{8} \right).
\end{align*}
\end{proof}

\subsection{Proof of Lemma \ref{lem:event_J}}\label{sec:proof_E2}

\subsubsection{Helper Lemma for the Proof of Lemma \ref{lem:event_J}}
\begin{lemma}\label{lem:J}
Let $J = \big\{ j \in [n]: |\cB^j| \geq \frac{mp}{2}\big\}$. Then
\begin{align*}
	\mathbb{P}\bigg(|J| < \frac{n}{2} \left[ 1 - \exp\left( - \frac{mp}{8} \right) \right] \bigg) 
		&\leq \exp\bigg( - \frac{n}{8} \left[ 1 - \exp\left( - \frac{mp}{8} \right) \right] \bigg).
\end{align*}
\end{lemma}

\begin{proof}
Observe that the cardinality of the set $J$ can be written as the sum of indicator variables as
\begin{equation}\label{eqn:J}
	|J| = \sum_{j \in [n]} \Ind{\big| \cB^{j} \big| \geq \frac{mp}{2}}.	
\end{equation}
Note that $|\cB^j| = \sum_{i \in [m]} M(i,j)$ is the sum of $n$ independent Bernoulli random variables with parameter $p$. 
It follows from the binomial Chernoff bound that
\[	\Prob	{\big| \cB^{j} \big| \geq \frac{mp}{2}	} \geq 1 - \exp\left( - \frac{mp}{8} \right).	\] 
Therefore, we can view the $n$ indicator variables in \eqref{eqn:J} as independent Bernoulli random variables, each of which 
takes value $1$ with probability $p'$ such that $p' \geq 1 - \exp\left( - \frac{mp}{8} \right)$.
Therefore,
\begin{align*}
	\Prob{|J| < \frac{n}{2} \left[ 1 - \exp\left( - \frac{mp}{8} \right) \right] } 
		&\leq \Prob{ |J| < \frac{np'}{2}}
		 \stackrel{(a)}{\leq} \exp\left( - \frac{np'}{8} \right) 
		 \leq \exp\left( - \frac{n}{8} \left[ 1 - \exp\left( - \frac{mp}{8} \right) \right] \right).
\end{align*}
by applying the Binomial Chernoff bound again at (a).

\end{proof}

\subsubsection{Completing the Proof of Lemma \ref{lem:event_J}}
\begin{proof}[Proof of Lemma \ref{lem:event_J}]
When $p \geq \frac{8 \log 2}{m}$, we can observe that $mp \geq 8 \log 2$, and hence, $ \exp\left( - \frac{mp}{8} \right) \leq \frac{1}{2}$. 
Then by Lemma \ref{lem:J},
\begin{align*}
	\Prob{\EJ^c} &= \mathbb{P}\bigg(|J| < \frac{n}{4} \bigg)
		 \leq \mathbb{P}\bigg(|J| < \frac{n}{2} \left[ 1 - \exp\left( - \frac{mp}{8} \right) \right] \bigg)
		\leq \exp\bigg( - \frac{n}{8} \left[ 1 - \exp\left( - \frac{mp}{8} \right) \right] \bigg)
		\leq \exp\bigg( - \frac{n}{16} \bigg).
\end{align*}
\end{proof}

\subsection{Proof of Lemma \ref{lem:event_Ti}}\label{sec:proof_E3}

\subsubsection{Helper Lemma for the Proof of Lemma \ref{lem:event_Ti}}

\begin{lemma}\label{lem:I}
Let $J = \big\{ j \in [n]: |\cB^j| \geq \frac{mp}{2}\big\}$ and $I = \big\{ i \in [m]: |\cB_i \cap J| \geq \frac{|J|p}{2}\big\}$. Then
	\begin{align*}
		\mathbb{P}\bigg(|I| < \frac{m}{2} \left[ 1 - \exp\left( - \frac{n_J p}{8} \right) \right] ~\Big|~ |J| = n_J \bigg) 
			&\leq \exp\bigg( - \frac{m}{8} \left[ 1 - \exp\left( - \frac{n_J p}{8} \right) \right] \bigg).
	\end{align*}
\end{lemma}

\begin{proof}
In the same vein as int he proof of Lemma \ref{lem:J}, we observe that
\begin{equation}\label{eqn:I}
	|I| = \sum_{i \in [m]} \Ind{\left| \cB_{i} \cap J \right| \geq \frac{|J|p}{2}}.
\end{equation}
Now $|\cB_i \cap J| = \sum_{j \in J} M(i,j)$ is distributed as the binomial distribution with parameters $(m, p')$ with $p' \geq p$. 
We can see that $p' \geq p$ because $p' = \Prob{M(i,j) = 1 ~\big| ~ j \in J } \geq \Prob{M(i,j) = 1 ~\big| ~ j \not\in J}$ and 
$\Prob{M(i,j) = 1 } = p$. These $m$ indicator variables are independent Bernoulli variables, each of which takes value $1$ with probability greater than
\[	\Prob	{\left| \cB_{i} \cap J \right| \geq \frac{ n_J p}{2} ~\Big|~ |J| = n_J } 
		\geq \Prob	{\left| \cB_{i} \cap J \right| \geq \frac{ n_J p'}{2}	}
		\geq 1 - \exp\left( - \frac{ n_J p'}{8} \right)
		\geq 1 - \exp\left( - \frac{ n_J p}{8} \right).	\]
Therefore, when $|J| = n_J$, we can see that the $n$ indicator variables in \eqref{eqn:I} are independent Bernoulli random variables 
with parameter $p''$ such that $p'' \geq 1 - \exp\left( - \frac{ n_J p}{8} \right)$. That is to say, $|I|$ is distributed as the binomial distribution 
with parameter $(m, p'')$ when conditioned on $|J| = n_J$. Letting $W$ denote a binomial random variable with parameter $(m, p'')$, 
we observe that $\bbE W = m p'' \geq m \big( 1 - \exp\left( - \frac{ n_J p}{8} \right) \big)$ and therefore,
\begin{align*}
	\mathbb{P}\Bigg( |I| < \frac{m}{2} \left[ 1 - \exp\left( - \frac{ n_J p}{8} \right) \right] ~\Big|~ |J| = n_J \Bigg) 
		&\leq \mathbb{P}\Big( W < \frac{1}{2} \bbE W \Big)
		\stackrel{(a)}{ \leq} \exp\left( - \frac{mp''}{8} \right) 
		 \leq \exp\Bigg( - \frac{m}{8} \left[ 1 - \exp\left( - \frac{n_J p}{8} \right) \right] \Bigg).
\end{align*}
The inequality (a) follows from the Binomial Chernoff bound.
\end{proof}

\begin{lemma}\label{lem:T_large} 
Given $i \in [m]$, let $\cT_i$ denote the set as defined in \eqref{eqn:set_Ti} that is constructed by Algorithm \ref{alg:setT_main}.
Then
\[
	\Prob{ | \cT_i | <  \Bigg(\frac{m}{2} \left[ 1 - \exp\left( - \frac{n_J p}{8} \right) \right] - 1 \Bigg)	
		\Bigg\lceil \frac{ n_J p}{4} - \frac{1}{2} \bigg( \bigg\lfloor \sqrt{\frac{n_J p}{2}}\bigg\rfloor + 1 \bigg)\Bigg\rceil 	~\Bigg|~ |J| \geq n_J }
		\leq \exp\bigg( - \frac{m}{8} \left[ 1 - \exp\left( - \frac{n_J p}{8} \right) \right] \bigg).
\]
\end{lemma}
	
\begin{proof}
Recall the definitions of $J = \big\{ j \in [n]: |\cB^j| \geq \frac{mp}{2}\big\}$ and $I = \big\{ i \in [m]: |\cB_i \cap J| \geq \frac{|J|p}{2}\big\}$. 
For $i' \in I$, let $\sigma_{i'}: \cB_{i'} \cap J \to \left[ |\cB_{i'} \cap J | \right]$ denote a map that sorts the column index $j \in \cB_{i'} \cap J \subseteq [n]$ 
in the increasing order of $\hat{q}_{\marg}\left( j \right)$ such that $\hat{q}_{\marg}\left( j_1 \right) \leq \hat{q}_{\marg}\left( j_2 \right)$ if $\sigma_{i'}(j_1) 
< \sigma_{i'}(j_2)$. Note that $\sigma_{i'}$ is a bijection and is invertible; we let $\sigma_{i'}^{-1}: [|\cB_{i'} \cap J |] \to \cB_{i'} \cap J \subseteq [n]$ 
denote the inverse map of $\sigma_{i'}$.

Now, we define a set 
\begin{equation}\label{eqn:bad_condition}
	\cS_{i'} := \left\{ k \in \left[ \left| \cB_{i'} \cap J \right| -1 \right] ~\Bigg|~ 
		\Big| \hat{q}_{\marg}\left( \sigma_{i'}^{-1}(k+1) \right) - \hat{q}_{\marg}\left( \sigma_{i'}^{-1}(k) \right) \Big| 
		> \frac{1}{\sqrt{\left|\cB_{i'} \cap J \right| }} \right\}.
\end{equation}
It is easy to verify that $| \cS_{i'} | \leq \left\lfloor \sqrt{\left| \cB_{i'} \cap J \right|} \right\rfloor$ because $\hat{q}_{\marg}\left( \sigma_{i'}^{-1}(k) \right)$ is 
increasing with respect to $k$ and $\hat{q}_{\marg}\left( \sigma_{i'}^{-1}(k) \right) \in [0,1]$ for all $k \in \left[ \left| \cB_{i'} \cap J \right| \right]$.

For those $k \in \left[ \left|\cB_{i'} \cap J \right| - 1 \right] \setminus \cS_{i'}$, we have 
\[
	\hat{q}_{\marg}\left(\sigma_{i'}^{-1}(k+1)\right)- \hat{q}_{\marg}\left(\sigma_{i'}^{-1}(k)\right) \leq \frac{1}{\sqrt{|\cB_{i'} \cap J |}}.
\] 
In case both $k, k+1\in  \left[ \left|\cB_{i'} \cap J \right| - 1 \right] \setminus \cS_{i'}$, either 
$\big(i', \sigma_{i'}^{-1}(k), \sigma_{i'}^{-1}(k+1)\big) \in \cT$ or $\big(i', \sigma_{i'}^{-1}(k+1), \sigma_{i'}^{-1}(k+2)\big) \in \cT$, 
but not both; see lines 10 - 12 of Algorithm \ref{alg:setT_main}. However, $\left(i', \sigma_{i'}^{-1}(k), \sigma_{i'}^{-1}(k+1)\right) \in \cT$ 
for at least half of $k \in \big[ \left|\cB_{i'} \cap J \right| - 1 \big] \setminus \cS_{i'}$. 

From the above observations, we can see that for each $i' \in I$, there exist at least 
$ \Big\lceil \frac{1}{2} \Big( |\cB_{i'} \cap J | - 1 - \left\lfloor \sqrt{\left|\cB_{i'} \cap J \right| }\right\rfloor \Big) \Big\rceil $ 
number of $k$'s such that $\left(i', \sigma_{i'}^{-1}(k), \sigma_{i'}^{-1}(k+1)\right) \in \cT$. Moreover, for $i' \in I$, 
\begin{equation*}
	\bigg\lceil \frac{1}{2} \Big( |\cB_{i'} \cap J | - 1 - \Big\lfloor \sqrt{\left|\cB_{i'} \cap J \right| }\Big\rfloor \Big) \bigg\rceil 
		\geq \bigg\lceil \frac{1}{2} \Big( \frac{|J|p}{2} - 1 - \Big\lfloor \sqrt{\frac{|J|p}{2} }\Big\rfloor \Big) \bigg\rceil
\end{equation*}
by the definition of $I$. Therefore,
\begin{equation*}
	| \cT | \geq |I| \bigg\lceil \frac{1}{2} \Big( \frac{|J|p}{2} - 1 - \Big\lfloor \sqrt{\frac{|J|p}{2} }\Big\rfloor \Big) \bigg\rceil
\end{equation*}
and even when $i \in I$,
\begin{equation}\label{eqn:cT1.a}
	| \cT_i | \geq \big( |I| - 1 \big) \bigg\lceil \frac{1}{2} \Big( \frac{|J|p}{2} - 1 - \Big\lfloor \sqrt{\frac{|J|p}{2} }\Big\rfloor \Big) \bigg\rceil.
\end{equation}

All in all, by \eqref{eqn:cT1.a} and Lemma \ref{lem:I},
\begin{align*}
	&\Prob{ | \cT_i | <  \Bigg(\frac{m}{2} \left[ 1 - \exp\left( - \frac{n_J p}{8} \right) \right] - 1 \Bigg)	
		\Bigg\lceil \frac{ n_J p}{4} - \frac{1}{2} \bigg( \bigg\lfloor \sqrt{\frac{n_J p}{2}}\bigg\rfloor + 1 \bigg)\Bigg\rceil 	~\Bigg|~ |J| \geq n_J }\\
	&\qquad\leq	\Prob{ | \cT_i | <  \Bigg(\frac{m}{2} \left[ 1 - \exp\left( - \frac{n_J p}{8} \right) \right] - 1 \Bigg)
		\Bigg\lceil \frac{ n_J p}{4} - \frac{1}{2} \bigg( \bigg\lfloor \sqrt{\frac{n_J p}{2}}\bigg\rfloor + 1 \bigg)\Bigg\rceil 	~\Bigg|~ |J| = n_J }\\
	&\qquad\leq	\Prob{ | I | <  \frac{m}{2} \left[ 1 - \exp\left( - \frac{n_J p}{8} \right) \right]	~\Bigg|~ |J| = n_J }\\
	&\qquad\leq \exp\bigg( - \frac{m}{8} \left[ 1 - \exp\left( - \frac{n_J p}{8} \right) \right] \bigg).
\end{align*}
\end{proof}
	
We have shown that the set $\cT_i$ is sufficiently large with high probability.

\subsubsection{Completing the Proof of Lemma \ref{lem:event_Ti}}
\begin{proof}[Proof of Lemma \ref{lem:event_Ti}]
Conditioned on $\EJ$, $|J| \geq \frac{1}{4}n$. 
Since $m \geq 8$ and $np \geq 8(1+ \sqrt{3})^2 > 32 \log 2$,
\begin{equation}\label{eqn:simplification_Ti}
	\frac{m}{2} \left[ 1 - \exp\left( - \frac{np}{32} \right) \right] - 1  \geq \frac{m}{4} - 1 \geq \frac{m}{8}			\qquad \text{and}\qquad
	\Bigg\lceil \frac{1}{2} \bigg( \frac{np}{8} - \bigg\lfloor \sqrt{\frac{np}{8}}\bigg\rfloor - 1 \bigg)\Bigg\rceil
		\geq \frac{1}{2} \bigg( \frac{np}{8} - \sqrt{\frac{np}{8}} - 1 \bigg)
		\geq \frac{np}{32}.
\end{equation}
Therefore, for any $i \in [m]$, 
\begin{align*}
	\Prob{\EcTi^c \big| \EJ}
		&= \Prob{ |\cT_i| < \frac{1}{256} mnp ~\bigg|~ |J| \geq \frac{1}{4}n }\\
		&\stackrel{(a)}{\leq} \Prob{ \left| \cT_i \right| < 	\bigg(\frac{m}{2} \left[ 1 - \exp\left( - \frac{np}{32} \right) \right] - 1 \bigg)
			\bigg\lceil \frac{np}{16} - \frac{1}{2} \bigg( \bigg\lfloor \sqrt{\frac{np}{8}}\bigg\rfloor + 1 \bigg)\bigg\rceil  ~\bigg|~	|J| \geq \frac{1}{4}n }	\\
		&\stackrel{(b)}{\leq} \exp\bigg( - \frac{m}{8} \left[ 1 - \exp\left( - \frac{n p}{32} \right) \right] \bigg)\\
		&\stackrel{(c)}{\leq} \exp\bigg( - \frac{m}{16} \bigg).
\end{align*}
Here, (a) follows from \eqref{eqn:simplification_Ti}; (b) is the result of Lemma \ref{lem:T_large}; and (c) is trivial because $np \geq 32 \log 2$.
\end{proof}




\subsection{Proof of Lemma \ref{lem:event_dA}}\label{sec:proof_E4}
\subsubsection{Helper Lemma for the Proof of Lemma \ref{lem:event_dA}}
Lemma \ref{lem:dA_max} shows that $\max_{\left(i', j_1, j_2 \right) \in \cT} \big| A(i', j_1) - A(i', j_2) \big|$ diminishes as $mp, np \to \infty$ 
at the rate of $\max\{ (mp)^{-\frac{1}{2}}, (np)^{-\frac{1}{2}} \}$ with high probability.
	
\begin{lemma}\label{lem:dA_max}
	For any $i \in [m]$ and any $t > 0$,
	\begin{align*}
		\Prob{ \max_{\left(i', j_1, j_2 \right) \in \cT} \big| A(i', j_1) - A(i', j_2) \big| > t  
		+ \lipmax \bigg( \sqrt{\frac{2}{ n_J p}} + \frac{16 \sqrt{2 \pi} \constaa}{\sqrt{ n_{\cB} } }  \bigg) 
			~\bigg|~ |J| = n_J, \min_{j' \in [n]} | \cB^{j'} | = n_{\cB} }
		\leq 3n \exp \left( -\frac{nt^2}{8 \lipmax^2} \right).
	\end{align*}
\end{lemma}

\begin{proof}
First of all, we know that for any $i' \in [m]$ and any $j_1, j_2 \in [n]$,
\begin{equation}\label{eqn:F4.a}
	\big| A(i', j_1) - A(i', j_2) \big|	\leq \lipmax \left| \fcol{j_1} - \fcol{j_2} \right|.
\end{equation}
because the latent function is $\lipmax$-Lipschitz by our model assumption. Also, by the triangle inequality, we have
\begin{equation}\label{eqn:F4.b}
	\big| \fcol{j_1} - \fcol{j_2} \big|	
		\leq  \big| \fcol{j_1} - \hat{q}_{\marg}(j_1) \big| + \Big| \hat{q}_{\marg}(j_1) - \hat{q}_{\marg}(j_2) \Big| + \big| \hat{q}_{\marg}(j_2) - \fcol{j_2} \big| .
\end{equation}
Then
\begin{align}
	&\Prob{ \max_{\left(i', j_1, j_2 \right) \in \cT} \big| A(i', j_1) - A(i', j_2) \big| > t  
		+ \lipmax \bigg( \sqrt{\frac{2}{\left| J \right| p}} + \frac{16 \sqrt{2 \pi} \constaa}{\sqrt{ \min_{j' \in [n]} | \cB^{j'} | } }  \bigg) }	\nonumber\\
		&\qquad	= \Prob{ \max_{\left(i', j_1, j_2 \right) \in \cT_i} \big| A(i', j_1) - A(i', j_2) \big| > t 
				+ \lipmax \bigg( \sqrt{\frac{2}{\left| J \right| p}} + \frac{16 \sqrt{2 \pi} \constaa}{\sqrt{ \min_{j' \in [n]} | \cB^{j'} | } }  \bigg) 
					 }	\nonumber\\
		&\qquad	\stackrel{(a)}{\leq}
			\Prob{ \exists \left(i', j_1, j_2 \right) \in \cT \quad\text{such that}\quad 
			\big| \fcol{j_1} - \fcol{j_2} \big| > \frac{t}{\lipmax}  
				+ \sqrt{\frac{2}{\left| J \right| p}} + \frac{16 \sqrt{2 \pi} \constaa}{\sqrt{ \min_{j' \in [n]} | \cB^{j'} | } } }	\nonumber\\
		&\qquad\stackrel{(b)}{\leq}
		 	\Prob{ \exists \left(i', j_1, j_2 \right) \in \cT \quad\text{such that}\quad
			\Big| \hat{q}_{\marg}\left( j_1 \right) - \hat{q}_{\marg}\left( j_2 \right) \Big| > \sqrt{\frac{2}{\left| J \right| p}} }	\nonumber\\
		&\qquad\quad
			+ \Prob{ \exists j \in [n] \quad\text{such that}\quad 
				\left| \hat{q}_{\marg}(j) - \fcol{j} \right| >  \frac{t}{2 \lipmax} + \frac{8 \sqrt{2 \pi} \constaa}{\sqrt{ \min_{j' \in [n]} | \cB^{j'} | } } }
				\label{eqn:F4.term.b}
\end{align}
where (a) follows from \eqref{eqn:F4.a} and (b) follows from \eqref{eqn:F4.b}.

Next, we observe that $|\cB_i \cap J| \geq \frac{|J|p}{2}$ for any $i \in I$ by definition of $I$. Therefore, by definition\footnote{See Algorithm 
\ref{alg:setT_main} for its construction.} of $\cT$, for any $(i', j_1, j_2) \in \cT$,
\[	
	\big| \hat{q}_{\marg}\left( j_1 \right) - \hat{q}_{\marg}\left( j_2 \right) \big|	
		\leq \frac{1}{\sqrt{\left| \cB_i \cap J \right|}}	
		\leq \sqrt{\frac{2}{\left| J \right| p}}.	
\]
As a result,
\[	
	\Prob{ \exists \left(i', j_1, j_2 \right) \in \cT \quad\text{such that}\quad
			\Big| \hat{q}_{\marg}\left( j_1 \right) - \hat{q}_{\marg}\left( j_2 \right) \Big| > \sqrt{\frac{2}{\left| J \right| p}} }
		= 0.
\]
We can conclude the proof by establishing an upper bound on \eqref{eqn:F4.term.b} as
\begin{align*}
	&\Prob{ \exists j \in [n] \quad\text{such that}\quad 
		\left| \hat{q}_{\marg}(j) - \fcol{j} \right| >  \frac{t}{2 \lipmax} + \frac{8 \sqrt{2 \pi} \constaa}{\sqrt{ \min_{j' \in [n]} | \cB^{j'} | } }  }\\
		&\qquad \leq \sum_{j \in [n]} 
			\Prob{ \left| \hat{q}_{\marg}(j) - \fcol{j} \right| >  \frac{t}{2 \lipmax} + \frac{8 \sqrt{2 \pi} \constaa}{\sqrt{ \min_{j' \in [n]} | \cB^{j'} | } } }
				& \text{by the union bound}\\
		&\qquad \leq 3n \exp \left( -\frac{nt^2}{8 \lipmax^2} \right).	&\text{by Proposition \ref{prop:quantile_noisy}}
\end{align*}

\end{proof}

\subsubsection{Completing the Proof of Lemma \ref{lem:event_dA}}
\begin{proof}[Proof of Lemma \ref{lem:event_dA}]
	By definition of $\EJ, \Ecols$ and Lemma \ref{lem:dA_max}, it is easy to verify that
	\begin{align*}
		\Prob{ \EdA^c | \EJ, \Ecols }	
			&= \Prob{ \max_{\left(i', j_1, j_2 \right) \in \cT} \big| A(i', j_1) - A(i', j_2) \big| > t  
				+ \lipmax \bigg( \frac{2\sqrt{2}}{\sqrt{ n p}} + \frac{32 \sqrt{ \pi} \constaa}{\sqrt{ mp } }  \bigg)  
				~\bigg|~ |J| \geq \frac{n}{4}, \min_{j' \in [n]} |\cB^{j'}| \geq \frac{mp}{2} }\\
			&\leq 3n \exp \left( -\frac{nt^2}{8 \lipmax^2} \right).
	\end{align*}
	We conclude the proof by letting $t = 8 \lipmax \sqrt{\frac{\log n}{n}}$.
\end{proof}


\subsection{Proof of Lemma \ref{lem:event_dN}}\label{sec:proof_E5}
\begin{proof}[Proof of Lemma \ref{lem:event_dN}]
%
%
Note that $\big| N(i', j_1) - N(i', j_2) \big| \leq \big| N(i',j_1) \big| + \big| N(i',j_2) \big|$ by triangle inequality. 
Therefore, for any $t > 0$, 
\begin{align*}
	\Prob{ \max_{\left(i', j_1, j_2 \right) \in \cT} \big| N(i', j_1) - N(i', j_2) \big| > t }	
		&= \Prob{\exists (i', j_1, j_2) \in \cT \quad\text{such that}\quad \big| N(i', j_1) - N(i', j_2) \big| > t }\\
		&\stackrel{(a)}{\leq}	\Prob{\exists (i', j_1, j_2) \in \cT \quad \text{such that}\quad 
			\big| N(i',j_1) \big| \geq \frac{t}{2} \text{ or } \big| N(i',j_2) \big| \geq \frac{t}{2}}\\
		&\stackrel{(b)}{\leq} \Prob{\exists (i, j) \in [m] \times [n] \quad \text{such that}\quad 
			M(i,j) = 1 \text{ and } \big| N(i, j) \big| \geq \frac{t}{2}}\\
		&\stackrel{(c)}{\leq}	\sum_{(i,j) \in [m] \times [n] \atop M(i,j) = 1} \Prob{ \big| N(i, j) \big| \geq \frac{t}{2} } \\
		&\stackrel{(d)}{\leq}	2mn \exp \left( - \frac{t^2}{8 \sigma^2} \right).
\end{align*}
(a) follows from the observation above; (b) is trivial; (c) is obtained by the union bound; and (d) follows from the assumption of 
sub-gaussian noise. Choosing $t = 8\sigma \sqrt{\log (mn)}$ completes the proof.
\end{proof}


\subsection{Proof of Lemma \ref{lem:event_Ephi}}\label{sec:proof_E6}

\subsubsection{Helper Lemma for the proof of Lemma \ref{lem:event_Ephi}}
We present the following lemma with its proof postponed to Section \ref{sec:deferred_proof_ephi}.
\begin{lemma}\label{lem:Ephi}
	For any $i \in [m]$, for any positive integers $\Nnet$, and for any $\Lambda, s_1, s_2 \geq 0$,
	\begin{align*}
		&\Prob{ \sup_{t \in [-\Lambda, \Lambda]}  \big| \hatNest (t) - \phi_N(t) \big|^2  
			>   s_1 + s_2 + \frac{1}{\Nnet}\Big[ \Lambda^2\big(\dAi + \dNi \big)^2 + 2 \Lambda\sigma B \Big] ~ \bigg| ~ \big| \cT_i \big| = T_i } \\
			&\qquad\leq 2 \Nnet \exp \left( - \frac{T_i s_1^2}{ 2 \Lambda^2 \dAi^2 }  \right) + 2 \Nnet \exp \bigg(- \frac{T_i s_2^2}{2} \bigg).
	\end{align*}
\end{lemma}

\subsubsection{Completing the Proof of Lemma \ref{lem:event_Ephi}}

\begin{proof}[Proof of Lemma \ref{lem:event_Ephi}]
To begin with, we recall that given $i \in [m]$,
\begin{align*}
	&|\cT_i|	\geq 	\frac{1}{256}mnp,	&&\text{when conditioned on }\EcTi,	\\
	&\dAi	\leq		\max_{ (i',j_1, j_2) \in \cT} \big| A(i',j_1) - A(i',j_2) \big|
		\leq 	  \lipmax \bigg( \frac{32 \sqrt{ \pi} \constaa}{\sqrt{ mp } } + \frac{2\sqrt{2}}{\sqrt{ n p}} + 8\sqrt{\frac{\log n}{n}} \bigg),
			&&\text{when conditioned on }\EdA,	\\
	&\dNi	\leq		 \max_{ (i',j_1, j_2) \in \cT} \big| N(i',j_1) - N(i',j_2) \big|	\leq 	 8\sigma \sqrt{\log ( mn ) },	
		&&\text{when conditioned on }\EdN.
\end{align*}

Next, we let 
\begin{align*}
	\Lambda 	&= \frac{1}{h}	= (4\gamma)^{- \frac{1}{\beta}}(\log | \cB_i | )^{\frac{1}{\beta}},\\
	\Nnet 	&=	mn,\\
	s_1		&=	\frac{ 64  \lipmax}{h} \sqrt{\frac{\log(mn)}{mnp}}  
				\bigg( \frac{32 \sqrt{ \pi} \constaa}{\sqrt{ mp } } + \frac{2\sqrt{2}}{\sqrt{ n p}} + 8\sqrt{\frac{\log n}{n}} \bigg),	\\
	s_2		&=	64 \sqrt{\frac{\log(mn)}{mnp}}
\end{align*}
and plug them in Lemma \ref{lem:Ephi}. It is easy to verify that
\[
	\frac{ |\cT_i| s_1^2}{ 2 \Lambda^2 \dAi^2 } 
		\geq 8 \log(mn) 	\quad\text{and}\quad
	\frac{ |\cT_i| s_2^2}{2}
		\geq 8 \log(mn)
\]
and therefore, 
\begin{equation}\label{eqn:Ephi_tail}
	2 \Nnet \exp \left( - \frac{ |\cT_i| s_1^2}{ 2 \Lambda^2 \dAi^2 }  \right) + 2 \Nnet \exp \bigg(- \frac{ |\cT_i| s_2^2}{2} \bigg)
		\leq \frac{4}{m^7 n^7}.
\end{equation}

Lastly, we observe that
\begin{align*}
	\big(\dAi + \dNi \big)^2
		&\leq 2 \dAi^2 + 2 \dNi^2\\
		&= 2 \lipmax^2 \bigg( \frac{32 \sqrt{ \pi} \constaa}{\sqrt{ mp } } + \frac{2\sqrt{2}}{\sqrt{ n p}} + 8\sqrt{\frac{\log n}{n}} \bigg)^2
			+ 128 \sigma^2 \log(mn)\\
		&\leq 6 \lipmax^2 \bigg( \frac{1024 \pi \constaa^2}{mp} + \frac{8}{np} + \frac{64 \log n}{n} \bigg) + 128 \sigma^2 \log(mn).
\end{align*}
This observations yields that
\begin{align*}
	&\frac{1}{\Nnet}\Big[ \Lambda^2\big(\dAi + \dNi \big)^2 + 2 \Lambda\sigma B \Big]\\
		&\qquad\leq \frac{1}{mn} \Bigg\{ \frac{1}{h^2} \Bigg[  6 \lipmax^2 \bigg( \frac{1024 \pi \constaa^2}{mp} + \frac{8}{np} + \frac{64 \log n}{n} \bigg) 
			+ 128 \sigma^2 \log(mn) \Bigg] + \frac{2 \sigma B}{h}  \Bigg\}\\
		&\qquad= \frac{1}{mn} \Bigg\{ (4\gamma)^{-\frac{2}{\beta}} \Bigg[  6 \lipmax^2 \bigg( \frac{1024 \pi \constaa^2}{mp} + \frac{8}{np} 
			+ \frac{64 \log n}{n} \bigg) + 128 \sigma^2 \log(mn) \Bigg] \big( \log |\cB_i | \big)^{\frac{2}{\beta}} 
			+ 2 ( 4\gamma)^{-\frac{1}{\beta}} \sigma B \big( \log |\cB_i | \big)^{\frac{1}{\beta}}  \Bigg\}\\
		&\qquad\leq
			\frac{1}{mn} \Bigg\{ (4\gamma)^{-\frac{2}{\beta}} \Bigg[  6 \lipmax^2 \bigg( \frac{1024 \pi \constaa^2}{mp} + \frac{8}{np} 
			+ \frac{64 \log n}{n} \bigg) + 128 \sigma^2 \log(mn) \Bigg] \big( \log n \big)^{\frac{2}{\beta}} 
			+ 2 ( 4\gamma)^{-\frac{1}{\beta}} \sigma B \big( \log n \big)^{\frac{1}{\beta}}  \Bigg\}.
\end{align*}
\end{proof}

\section{Supplement 2 to the Proof of Theorem \ref{thm:ensure_condition}: \\Deferred Proof of Lemma \ref{lem:Ephi} }\label{sec:deferred_proof_ephi}

In this section, we prove Lemma \ref{lem:Ephi}. In Section \ref{sec:F5.prelim}, we sketch the outline of our proof and define some quantities 
to be used in the proof of Lemma \ref{lem:Ephi}. We present and prove intermediate lemmas in Section \ref{sec:F5.part1} and \ref{sec:F5.part2} 
and then combine them together to complete the proof of Lemma \ref{lem:Ephi} in Section \ref{sec:F5.complete}.

\subsection{Preliminary}\label{sec:F5.prelim}
Recall the definition of $\hat{\phi}_{N, i}(t)$ from \eqref{eqn:chN_est}: for $i \in [m]$, we let
\begin{equation*}
	\hat{\phi}_{N, i}(t) =	
		\Bigg| \frac{1}{\left| \cT_i \right|} \sum_{ \left(i', j_1, j_2 \right) \in \cT_i} 
			\cos \Big[ t \big( Z(i', j_1) - Z(i', j_2)  \big) \Big] \Bigg|^{\frac{1}{2}}.	
\end{equation*}
where $\cT_i$ is as defined in \eqref{eqn:set_Ti} and Algorithm \ref{alg:setT_main}. 

For the purpose of analysis, we define several functions related to $\hat{\phi}_{N, i}(t)$. For $i \in [m]$, we define
\begin{align}
	\HatNest (t)	&= \frac{1}{\left| \cT_i \right|} \sum_{ (i', j_1, j_2)  \in \cT_i} \cos \Big[ t \left( Z(i', j_1) - Z(i', j_2)  \right) \Big],	
				\label{eqn:Phi}\\
	\idealNest(t) 	&= \Bigg| \frac{1}{\left| \cT_i \right|} \sum_{ (i', j_1, j_2)  \in \cT_i} \cos \Big[ t \left( N(i', j_1) - N(i', j_2)  \right) \Big] \Bigg|^{\frac{1}{2}}, 
				\label{eqn:noise_ideal}\\
	\IdealNest(t) 	&= \frac{1}{\left| \cT_i \right|} \sum_{ (i', j_1, j_2) \in \cT_i} \cos \Big[ t \left( N(i', j_1) - N(i', j_2)  \right) \Big].		
				\label{eqn:Phi_star}
\end{align}
First, $\idealNest(t)$ defined in \eqref{eqn:noise_ideal} is the `ideal' estimator of $\phi_N$ which we would use if we had access to $N(i', j_1)$ 
and $N(i', j_2)$. However, $\idealNest(t)$ is not computable from data and thus we estimate $\phi_N$ with $\hat{\phi}_{N, i}$, instead. 
Observe that $\hat{\phi}_{N, i}(t) = \big| \HatNest (t) \big|^{\frac{1}{2}}$ and $\idealNest(t) = \big| \HatNest (t) \big|^{\frac{1}{2}}$ for all $t \in \Reals$.

We want to establish a uniform upper bound on $\big| \hatNest (t) - \phi_N(t) \big| $. Since $\phi_N(t) > 0$ by the supersmoothness assumption 
(see \eqref{eqn:model_supersmooth}) and $ \hatNest (t) \geq 0$ by its construction (see \eqref{eqn:chN_est}), we can see that
\begin{align*}
	\big| \hatNest (t) - \phi_N(t) \big|^2 
		&\leq	\big| \hatNest (t) + \phi_N(t) \big|\big| \hatNest (t) - \phi_N(t) \big|
		= \big| \hatNest (t)^2 - \phi_N(t)^2 \big|
		=	\big| |\HatNest (t)| - \phi_N(t)^2  \big| \\
		&\leq	\big| \HatNest (t) - \phi_N(t)^2  \big|	\\
		&\leq \big| \HatNest(t) - \IdealNest(t) \big| + \big| \IdealNest(t)  - \phi_N(t)^2 \big|
\end{align*}
for all $t \in \Reals$. Taking the supremum over an interval $[-\Lambda, \Lambda]$, we obtain 
\[
	\sup_{t \in [-\Lambda, \Lambda]} \big| \hatNest (t) - \phi_N(t) \big|^2  
		\leq \sup_{t \in [-\Lambda, \Lambda]} \big| \HatNest(t) - \IdealNest(t) \big| 
		+ \sup_{t \in [-\Lambda, \Lambda]} \big| \IdealNest(t)  - \phi_N(t)^2 \big|.
\]

We establish a probabilistic tail bound on $\sup_{t \in [-\Lambda, \Lambda]} \big| \HatNest(t) - \IdealNest(t) \big| $ in Section \ref{sec:F5.part1} and 
a similar upper bound on $ \sup_{t \in [-\Lambda, \Lambda]} \big| \IdealNest(t)  - \phi_N(t)^2 \big|$ in Section \ref{sec:F5.part2}, separately.

For the convenience of presenting our results, we also define the following quantities for each $i \in [m]$:
\begin{equation}\label{eqn:dNi}
	\dNi	:=   \max_{ (i', j_1, j_2) \in \cT_i} \Big| N(i', j_1) - N(i', j_2) \Big|
\end{equation}
and
\begin{equation}\label{eqn:dAi}
	\dAi	:=   \max_{ (i', j_1, j_2) \in \cT_i} \Big| A(i', j_1) - A(i', j_2) \Big|.
\end{equation}


\subsection{Intermediate Step 1: Establishing a Uniform Upper Bound on $\big| \HatNest(t) - \IdealNest(t) \big|$}\label{sec:F5.part1}	
In the proof of Lemma \ref{lem:practical_ideal} and Lemma \ref{lem:sup_noise}, we use the following shorthand notations: 
for $(i', j_1, j_2) \in [m] \times [n]^2$,
\begin{equation}\label{eqn:Delta_AN}
	\DAijj := A(i', j_1) - A(i', j_2)	\qquad\text{and}\qquad	\DNijj := N(i', j_1) - N(i', j_2).
\end{equation}

\begin{lemma}\label{lem:practical_ideal}
Given $i \in [m]$, let $\HatNest (t)$ and $\IdealNest(t)$ denote the functions as defined in \eqref{eqn:Phi} and \eqref{eqn:Phi_star}. 
Then for any $t \in \Reals$ and any $s > 0$,
\begin{align*}
	\bbP\bigg( \big| \HatNest (t) - \IdealNest(t)  \big| > s +  \max_{(i', j_1, j_2)\in \cT_i} \frac{t^2}{2} \dAi^2 ~ \bigg|~  |\cT_i| = T_i \bigg)
		\leq 2 \exp \left( - \frac{T_i s^2}{ 2 t^2 \dAi^2 }  \right).
\end{align*}
\end{lemma}

\begin{proof}
In this proof, we establish a high-probability upper bound on $\big| \HatNest (t) - \IdealNest(t) \big|$ by (1) finding an upper bound on 
its expectation and then (2) proving the concentration of $\big| \HatNest (t) - \IdealNest(t) \big|$ to its expectation.

Recall from our model that $Z(i,j) = A(i,j) + N(i,j)$ for $(i,j)$ such that $M(i,j) = 1$. For $(i', j_1, j_2) \in \cT_i$, we can write
\begin{align*}
	Z(i', j_1) - Z(i', j_2) = \big[ N(i', j_1) - N(i', j_2) \big] +  \big[ A(i', j_1) - A(i', j_2) \big].
\end{align*}
By definition of $\IdealNest$ and $\HatNest (t)$, and by the trigonometric identity $\cos a - \cos b = -2 \sin \frac{a+b}{2} \sin \frac{a-b}{2}$, 
\begin{align}
	&\HatNest (t) - \IdealNest(t) 		\nonumber\\
		&\qquad	= \frac{1}{\left| \cT_i \right|} \sum_{ \left(i', j_1, j_2) \right) \in \cT_i} 
			\bigg\{ \cos \Big( t \big[ Z(i', j_1) - Z(i', j_2) \big] \Big) - \cos\Big( t \big[ N(i', j_1) - N(i', j_2) \big] \Big)  \bigg\}		\nonumber\\
		&\qquad	= \frac{-2}{\left| \cT_i \right|} \sum_{ \left(i', j_1, j_2) \right) \in \cT_i} \sin \left( t \big[ N(i', j_1) - N(i', j_2) \big]  + \frac{ t \big[ A(i', j_1) - A(i', j_2) \big] }{2}  \right) 
			\sin \left( \frac{ t \big[ A(i', j_1) - A(i', j_2) \big] }{2}\right).	\label{eqn:Phi_difference}
\end{align}

First of all, we establish an upper bound on $\bbE \big[ \HatNest (t) - \IdealNest(t) \big]$. Note that the noise is independent of the signal 
(and hence, independent of the latent features) in our model. Therefore, $\big\{ N(i',j_1), N(i',j_2) \big\}_{(i',j_1, j_2) \in \cT_i}$ are independent 
of $\big\{ \frow{i}, \fcol{j} \big\}_{(i,j) \in [m] \times [n]}$. Now we consider the conditional expectation of $ \HatNest (t) - \IdealNest(t)$ given 
the latent features $\frow{1:m}$ and $\fcol{1:n}$.
\begin{align*}
	&\bbE \Big[ \HatNest (t) - \IdealNest(t) ~\Big|~ \frow{1:m}, \fcol{1:n} \Big]	\\
		&\qquad\stackrel{(a)}{=}  \bbE\Bigg[\frac{-2}{\left| \cT_i \right|} \sum_{ \left(i', j_1, j_2) \right) \in \cT_i} \sin \bigg( t \DNijj + \frac{t \DAijj }{2} \bigg) \sin \bigg( \frac{t \DAijj}{2}\bigg) 
			~\bigg|~ \frow{1:m}, \fcol{1:n} \Bigg]	\\
		&\qquad\stackrel{(b)}{=}  \bbE \Bigg[ \frac{-1}{\left| \cT_i \right|} \sum_{ \left(i', j_1, j_2) \right) \in \cT_i} 
			\left[ \sin \bigg( t \DNijj+ \frac{t \DAijj}{2} \bigg) + \sin \bigg( -t \DNijj+ \frac{t \DAijj}{2} \bigg) \right] 
				\sin \bigg(  \frac{t \DAijj}{2}\bigg) ~\bigg|~ \frow{1:m}, \fcol{1:n} \Bigg]		\\
		&\qquad\stackrel{(c)} {=} \bbE\Bigg[\frac{-2}{\left| \cT_i \right|} \sum_{ \left(i', j_1, j_2) \right) \in \cT_i} 
			\cos \big( t \DNijj \big) \sin^2 \bigg( \frac{t \DAijj}{2}\bigg)~\bigg|~ \frow{1:m}, \fcol{1:n} \Bigg].
\end{align*}
Here, (a) follows from \eqref{eqn:Phi_difference}; (b) follows from the symmetry of the noise distribution; and (c) follows from the trigonometric 
identity, $\sin (a + b) + \sin (a - b) = 2 \sin a \cos b$. Since $\big| \cos \big( t \DNijj \big)  \big| \leq 1$ and $ \big| \sin \big( \frac{t \DAijj}{2}\big) \big| 
\leq \big| \frac{t \DAijj}{2} \big|$, it follows that
\begin{align*}
	\bigg| \bbE \Big[ \HatNest (t) - \IdealNest(t) ~\Big|~ \frow{1:m}, \fcol{1:n} \Big] \bigg|
		&\leq \frac{2}{|\cT_i|} \sum_{(i', j_1, j_2) \in \cT_i} \bigg| \frac{t \DAijj}{2} \bigg|^2\\
		&\leq \max_{(i', j_1, j_2)\in \cT_i} \frac{t^2}{2}\big(A(i', j_1) - A(i', j_2) \big)^2.
\end{align*}
Note that this upper bound holds regardless of $\frow{1:m}, \fcol{1:n}$. Therefore,
\begin{equation}\label{eqn:expectation_term.1}
	\Big| \bbE \big[ \HatNest(t) - \IdealNest(t) \big] \Big|	\leq \max_{(i', j_1, j_2)\in \cT_i} \frac{t^2}{2} \dAi^2.
\end{equation}

Next, we show $ \HatNest(t) - \IdealNest(t) $ concentrates to $\bbE \big[ \HatNest (t) - \IdealNest(t) \big]$. Observe from \eqref{eqn:Phi_difference} 
that $\HatNest (t) - \IdealNest(t) $ is the sum of $|\cT_i|$ independent random variables where the independence is ensured due to the manner 
$\cT_i$ is constructed. Moreover, each summand is a bounded random variable as $\big| \sin x \big| \leq x \wedge 1$. Applying the Hoeffding's 
inequality (Lemma \ref{lem:Hoeffding_bounded}), we can see that for any $t \in \Reals$ and any $s \geq 0$,
\begin{align}
	\Prob{ \bigg|  \HatNest(t) - \IdealNest(t) - \Exp{ \HatNest(t) - \IdealNest(t) } \bigg| > s }
		&\leq 2 \exp \left( - \frac{ 2 s^2}{ \sum_{(i' ,j_1, j_2) \in \cT_i} \Big( \frac{ 2 }{|\cT_i|} t \DAijj \Big)^2 }  \right)		\nonumber\\
		&\leq 2 \exp \left( - \frac{ |\cT_i| s^2}{ 2 t^2 \dAi^2 }  \right).		\label{eqn:conc_bound}
\end{align}
		
We combine \eqref{eqn:expectation_term.1}, and \eqref{eqn:conc_bound} by the usual argument (triangle inequality + union bound) to 
conclude the proof. Consequently, for any $t \in \Reals$ and any $s > 0$,
\begin{align*}
	&\bbP\bigg( \big| \HatNest (t) - \IdealNest(t)  \big| > s +  \max_{(i', j_1, j_2)\in \cT_i} \frac{t^2}{2} \dAi^2 \bigg)\\
		&\qquad\leq \bbP\bigg( \Big| \bbE \big[ \HatNest(t) - \IdealNest(t) \big] \Big|	 >  \max_{(i', j_1, j_2)\in \cT_i} \frac{t^2}{2} \dAi^2 \bigg)
			+ \bbP \bigg( \Big| \big( \HatNest (t) - \IdealNest(t)  \big) - \bbE \big[ \HatNest(t) - \IdealNest(t) \big] \Big| > s \bigg)\\
		&\qquad\leq 2 \exp \left( - \frac{ |\cT_i| s^2}{ 2 t^2 \dAi^2 }  \right).
\end{align*}
\end{proof}

\begin{lemma}\label{lem:sup_noise}
Given $i \in [m]$, let $\HatNest (t)$ and $\IdealNest(t)$ denote the functions as defined in \eqref{eqn:Phi} and \eqref{eqn:Phi_star}. 
Then for any positive integer $\Nnet$ and for any $\Lambda, s > 0$,
\begin{align*}
	\Prob{ \sup_{t \in [-\Lambda, \Lambda]} \big| \HatNest (t) - \IdealNest(t)  \big|  
		> s +   \frac{\Lambda^2}{\Nnet} \Big(2 \dNi + \dAi \Big) \dAi ~ \bigg| ~ \big| \cT_i \big| = T_i }
		\leq 2 \Nnet \exp \left( - \frac{T_i s^2}{ 2 \Lambda^2 \dAi^2 }  \right).
\end{align*}
\end{lemma}


\begin{proof} [Proof of Lemma \ref{lem:sup_noise}]
First, we discretize the interval interval $[-\Lambda, \Lambda]$ by constructing an $\varepsilon$-net. For any positive integer $\Nnet$, we define
\begin{equation}\label{eqn:eps_net_Lambda}
	\cT_{\Nnet, \Lambda} \triangleq \left\{ \frac{(2k - 1 - \Nnet)\Lambda}{2\Nnet} \in \Reals ~\text{ such that }~ k \in [\Nnet] \right\}.
\end{equation}
Observe that $\cT_{\Nnet, \Lambda}$ forms a $\frac{\Lambda}{\Nnet}$-net of the interval $[-\Lambda, \Lambda]$. That is, 
\begin{enumerate}
	\item
	$\cT_{\Nnet, \Lambda} \subset [-\Lambda, \Lambda]$; and
	\item
	for any $z \in [-\Lambda, \Lambda]$, there exists $z' \in \cT_{\Nnet, \Lambda}$ such that $\left| z - z' \right| \leq \frac{\Lambda}{\Nnet}$.
\end{enumerate}
Moreover, we observe that $\left| \cT_{\Nnet, \Lambda} \right| = \Nnet$. 
		
Next, we consider the derivative of $\HatNest (t) - \IdealNest(t)$ with respect to $t$. First, we recall the notation $\DAijj$, $\DNijj$ introduced 
in \eqref{eqn:Delta_AN} and the expression of $\HatNest (t) - \IdealNest(t)$ as written in \eqref{eqn:Phi_difference}. Then we observe that
\begin{align*}
	\frac{d}{dt} \Big[ \HatNest (t) - \IdealNest(t) \Big]
	&= \frac{d}{dt} \Bigg[  \frac{-2}{| \cT_i |} \sum_{ (i', j_1, j_2) \in \cT_i} \sin \bigg( t \DNijj + \frac{t \DAijj}{2}  \bigg) \sin \bigg( \frac{t \DAijj}{2}\bigg) \Bigg]\\
	&= \frac{-2}{\left| \cT_i \right|} \sum_{ (i', j_1, j_2) \in \cT_i} \Bigg[ \bigg( \DNijj+ \frac{\DAijj}{2} \bigg)\cos \bigg( t \DNijj+ \frac{t \DAijj}{2} \bigg) \sin \bigg( \frac{t \DAijj}{2}\bigg)\\
		&\qquad\qquad\qquad\qquad\qquad\qquad\qquad + \frac{\DAijj}{2}\sin \bigg( t \DNijj+ \frac{t \DAijj}{2} \bigg) \cos \bigg( \frac{t \DAijj}{2}\bigg)\Bigg].
\end{align*}
Therefore,
\begin{align*}
	& \sup_{t \in [-\Lambda, \Lambda]} \left| \frac{d}{dt} \Big[ \HatNest (t) - \IdealNest(t) \Big] \right| \\
		&\qquad\stackrel{(a)}{\leq}  2 \sup_{t \in [-\Lambda, \Lambda]} \Bigg\{ \max_{ (i', j_1, j_2) \in \cT_i} 
				\bigg| \DNijj+ \frac{\DAijj}{2} \bigg| \bigg| \cos \bigg( t \DNijj+ \frac{t \DAijj}{2} \bigg) \bigg| \bigg| \sin \bigg( \frac{t \DAijj}{2}\bigg) \bigg|\\
			&\qquad\qquad\qquad\qquad\qquad\qquad\qquad 
				+ \max_{ (i', j_1, j_2) \in \cT_i} \bigg| \frac{\DAijj}{2} \bigg| \bigg| \sin \bigg( t \DNijj+ \frac{t \DAijj}{2} \bigg) \bigg| \bigg| \cos \bigg( \frac{t \DAijj}{2}\bigg) \bigg| \Bigg\}\\
		&\qquad\stackrel{(b)}{\leq} 2 \sup_{t \in [-\Lambda, \Lambda]} \Bigg\{  \max_{ (i', j_1, j_2) \in \cT_i } 	 \bigg| \DNijj+ \frac{\DAijj}{2} \bigg| \bigg| \frac{t \DAijj}{2}\bigg| 
			+  \max_{ (i', j_1, j_2) \in \cT_i} \bigg| \frac{\DAijj}{2} \bigg| \bigg| t \DNijj+ \frac{t \DAijj}{2} \bigg| \Bigg\}\\
		&\qquad\stackrel{(c)}{\leq} \sup_{t \in [-\Lambda, \Lambda]} \big|t\big| \Big(2 \dNi + \dAi \Big) \dAi\\
		&\qquad\leq \Lambda \Big(2 \dNi + \dAi \Big) \dAi.
\end{align*}
Here, (a) follows from the triangle inequality; (b) follows from the observation that $|\sin x | \leq |x|$ and $|\cos x| \leq 1$; and (c) follows from the definition of $\dNi, \dAi$; 
see \eqref{eqn:dNi} and \eqref{eqn:dAi}.
		
Since the function $\HatNest (t) - \IdealNest(t) $ is continuous, 
\begin{align}
	\sup_{t \in [-\Lambda, \Lambda]} \big| \HatNest (t) - \IdealNest(t)  \big| 
		&\leq \sup_{t \in \cT_{\Nnet, \Lambda}} \big| \HatNest (t) - \IdealNest(t)  \big| 
			+  \frac{\Lambda}{\Nnet}\sup_{t \in [-\Lambda, \Lambda]} \bigg| \frac{d}{dt} \left( \HatNest (t) - \IdealNest(t)  \right) \bigg|	\nonumber\\
		&\leq \sup_{t \in \cT_{\Nnet, \Lambda}} \big| \HatNest (t) - \IdealNest(t)  \big| +  \frac{\Lambda^2}{\Nnet} \Big(2 \dNi + \dAi \Big) \dAi.	\label{eqn:eps_net.2}
\end{align}
Therefore, for any $s > 0$,
\begin{align*}
	&\bbP\bigg( \sup_{t \in [-\Lambda, \Lambda]} \big| \HatNest (t) - \IdealNest(t)  \big|  
		> s +   \frac{\Lambda^2}{\Nnet} \Big(2 \dNi + \dAi \Big) \dAi ~ \bigg| ~ \big| \cT_i \big| = T_i \bigg)\\
		&\qquad\stackrel{(a)}{\leq} \bbP\bigg(  \sup_{t \in \cT_{\Nnet, \Lambda}} \big| \HatNest (t) - \IdealNest(t)  \Big| > s 
			~ \bigg| ~ \big| \cT_i \big| = T_i  \bigg)\\
		&\qquad\stackrel{(b)}{\leq} \sum_{t \in  \cT_{\Nnet, \Lambda}} \Prob{ \big| \HatNest (t) - \IdealNest(t)  \big| > s ~ \big| ~ \big| \cT_i \big| = T_i } \\
		&\qquad\stackrel{(c)}{\leq} 2 \sum_{t \in \cT_{\Nnet, \Lambda}} \exp \left( - \frac{T_i s^2}{ 2 t^2 \dAi^2 }  \right)	\\
		&\qquad\stackrel{(d)}{\leq} 2 \Nnet \exp \left( - \frac{T_i s^2}{ 2 \Lambda^2 \dAi^2 }  \right).	
\end{align*}
(a) follows from \eqref{eqn:eps_net.2}; (b) is the result of applying the union bound; (c) follows from Lemma \ref{lem:practical_ideal}; and we have
(d) because $|\cT_{\Nnet, \Lambda}| = \Nnet$.
\end{proof}


\subsection{Intermediate Step 2: Establishing a Uniform Upper Bound on $\big| \IdealNest(t)  - \phi_N(t)^2 \big|$}\label{sec:F5.part2}
\begin{lemma}\label{lem:ideal_true}
Given $i \in [m]$, let $\IdealNest(t)$ denote the function as defined in \eqref{eqn:Phi_star}. 
Then for any $t \in \Reals$ and any $s > 0$,
\begin{align*}
	\Prob{ \big|\IdealNest(t) - \phi_N(t)^2 \big| > s ~ \Big| ~ \big| \cT_i \big| = T_i } 
		\leq 2 \exp \bigg(- \frac{T_i s^2}{2} \bigg).
\end{align*}
\end{lemma}

\begin{proof}
		
From the symmetry of the noise distribution and the independence between $N(i', j_1)$ and $N(i', j_2)$ for $(i', j_1, j_2) \in \cT_i$, 
\begin{align*}
	\Exp{ \cos \big[ t \left( N(i', j_1) - N(i', j_2)  \right)}
		&= \Exp{ \frac{1}{2} \exp \big( t \left( N(i', j_1) - N(i', j_2) \right)  \big) +  \frac{1}{2} \exp \big( -t \left( N(i', j_1) - N(i', j_2) \right)  \big) }\\
		&= \frac{1}{2} \mathbb{E}\big[ \exp \big( tN(i', j_1) \big)\big] \mathbb{E}\big[\exp\big( -tN(i', j_2) \big) \big] 
			+ \frac{1}{2}\mathbb{E}\big[ \exp\big( - tN(i', j_1) \big) \big] \mathbb{E}\big[ \exp \big( tN(i', j_2) \big) \big]\\
		&= \phi_N(t)^2.
\end{align*}
Therefore, $\bbE\big[ \IdealNest(t) \big] = \phi_N(t)^2$ for all $t \in \Reals$.

Next, we consider how $\IdealNest(t)$ concentrates to $\bbE\big[ \IdealNest(t) \big]$. Since 
$\IdealNest(t) 	= \frac{1}{\left| \cT_i \right|} \sum_{ (i', j_1, j_2) \in \cT_i} \cos \Big[ t \left( N(i', j_1) - N(i', j_2)  \right) \Big]$ is the sum of $|\cT_i$| 
independent random variables, each of which is bounded within $[ -\frac{1}{|\cT_i|}, \frac{1}{|\cT_i|}]$, we can apply Hoeffding's inequality 
(Lemma \ref{lem:Hoeffding_bounded}) to achieve 
\[	
	\Prob{ \big|\IdealNest(t) - \phi_N(t)^2 \big| > s ~ \Big| ~ \big| \cT_i \big| = T_i } 
		\leq 2 \exp \bigg(- \frac{T_i s^2}{2} \bigg), \quad \text{for all }t \in \Reals. 	
\]		
\end{proof}

\begin{lemma}\label{lem:sup_ideal}
Given $i \in [m]$, let $\IdealNest(t)$ denote the function as defined in \eqref{eqn:Phi_star}. 
Then for any positive integer $\Nnet$ and for any $\Lambda, s \geq 0$,
\begin{align*}
	\Prob{ \sup_{t \in [-\Lambda, \Lambda]} \big| \idealNest(t) - \phi_N(t) \big|^2
		> s +  \frac{\Lambda}{\Nnet} \Big( \Lambda \dNi^2 + 2\sigma B \Big) ~ \bigg| ~ \big| \cT_i \big| = T_i }
		\leq 2 \Nnet \exp \bigg(- \frac{T_i s^2}{2} \bigg),
\end{align*}
where $\sigma, B$ are noise model parameters.
\end{lemma}

\begin{proof}
First, we discretize the interval interval $[-\Lambda, \Lambda]$ by constructing an $\varepsilon$-net in the same manner as in the proof of 
Lemma \ref{lem:sup_noise}, cf. \eqref{eqn:eps_net_Lambda}. For any positive integer $\Nnet$, we define
\[
	\cT_{\Nnet, \Lambda} \triangleq \left\{ \frac{(2k - 1 - \Nnet)\Lambda}{2\Nnet} \in \Reals ~\text{ such that }~ k \in [\Nnet] \right\}.
\]
Observe that $\cT_{\Nnet, \Lambda}$ forms a $\frac{\Lambda}{\Nnet}$-net of the interval $[-\Lambda, \Lambda]$. That is, 
\begin{enumerate}
	\item
	$\cT_{\Nnet, \Lambda} \subset [-\Lambda, \Lambda]$; and
	\item
	for any $z \in [-\Lambda, \Lambda]$, there exists $z' \in \cT_{\Nnet, \Lambda}$ such that $\left| z - z' \right| \leq \frac{\Lambda}{\Nnet}$.
\end{enumerate}
Moreover, we observe that $\left| \cT_{\Nnet, \Lambda} \right| = \Nnet$. 

Next, we consider the function $\IdealNest(t) - \phi_N^2(t)$ and its derivative with respect to $t$. First, we observe that
\begin{align}
	\bigg| \frac{d}{dt} \IdealNest(t) \bigg|
		&= \bigg|  \frac{1}{\left| \cT_i \right|} \sum_{ \left(i, j_1, j_2) \right) \in \cT_i} \frac{d}{dt} \cos \big[ t  \left( N(i, j_1) - N(i, j_2) \right) \big] \bigg| \nonumber\\
		&= \bigg|  \frac{-1}{\left| \cT_i \right|} \sum_{ \left(i, j_1, j_2) \right) \in \cT_i} \sin \big[ t  \left( N(i, j_1) - N(i, j_2) \right)  \big] \left( N(i, j_1) - N(i, j_2) \right) \bigg| \nonumber\\
		&\leq \max_{ \left(i, j_1, j_2) \right) \in \cT_i} \big| t \big | \big| N(i, j_1) - N(i, j_2) \big|^2	\nonumber\\
		&= |t| \dNi^2.	\label{eqn:rate_upper.1}
\end{align}
Also, we observe that
\begin{align}
	\bigg| \frac{d}{dt}  \phi_N^2(t) \bigg|	
		&= 2 \bigg| \phi_N(t) \frac{d}{dt} \phi_N(t) \bigg|	\nonumber\\
		&\leq 2 \big| \phi_N(t) \big| \bigg| \frac{d}{dt} \int_{-\infty}^{\infty} e^{-\img t x} dF_N(x) \bigg|	\nonumber\\
		&\leq 2 \big| \phi_N(t) \big| \bigg| \int_{-\infty}^{\infty} (-\img x) e^{\img t x} dF_N(x) \bigg|		\qquad\text{by definition of }\phi_N(t)	\nonumber\\
		&\leq 2 \big| \phi_N(t) \big| \int_{-\infty}^{\infty} \big| x \big|  dF_N(x)	\nonumber\\
		&\leq 2\sigma B \exp\big( - \gamma |t|^{\beta} \big).	\label{eqn:rate_upper.2}
\end{align}
The last line follows from the supersmoothness ($\phi_N(t) \leq B \exp(- \gamma |t|^{\beta})$) and the sub-gaussian assumption of the noise:
\[	
	\int_{-\infty}^{\infty} \big| x \big|  dF_N(x) = \Exp{ \big| N \big| } \leq \bbE\big[ N^2 \big]^{\frac{1}{2}} \leq \sigma.
\]

It follows from \eqref{eqn:rate_upper.1}, \eqref{eqn:rate_upper.2} and triangle inequality that
\begin{align*}
	\sup_{t \in [-\Lambda, \Lambda]} \bigg| \frac{d}{dt} \left( \IdealNest(t) - \phi_N^2(t) \right) \bigg|
		&\leq \sup_{t \in [-\Lambda, \Lambda]} \Bigg( \bigg| \frac{d}{dt} \IdealNest(t) \bigg| + \sup_{t \in [-\Lambda, \Lambda]} \bigg| \frac{d}{dt} \phi_N^2(t) \bigg| \Bigg)\\
		&\leq \sup_{t \in [-\Lambda, \Lambda]} \Big(  |t| \dNi^2 + 2\sigma B \exp\big( - \gamma |t|^{\beta} \big) \Big)\\
		&\leq \Lambda \dNi^2 + 2\sigma B.
\end{align*}		
Then by the continuity of $\IdealNest(t) - \phi_N^2(t)$, we can see that
\begin{align}
	\sup_{t \in [-\Lambda, \Lambda]} \Big| \IdealNest(t) - \phi_N^2(t) \Big| 
		&\leq \sup_{t \in \cT_{\Nnet, \Lambda}} \Big| \IdealNest(t) - \phi_N^2(t) \Big| 
			+  \frac{\Lambda}{\Nnet}\sup_{t \in [-\Lambda, \Lambda]} \bigg| \frac{d}{dt} \left( \IdealNest(t) - \phi_N^2(t) \right) \bigg|	\nonumber\\
		&\leq \sup_{t \in \cT_{\Nnet, \Lambda}} \Big| \IdealNest(t) - \phi_N^2(t) \Big| +  \frac{\Lambda}{\Nnet} \big( |\Lambda| \dNi^2 + 2\sigma B \big).	\label{eqn:eps_net.1}
\end{align}

Therefore, for any $s \geq 0$,
\begin{align*}
	&\bbP\bigg( \sup_{t \in [-\Lambda, \Lambda]} \big| \IdealNest(t) - \phi_N^2(t) \big| 
			> s +  \frac{\Lambda}{\Nnet} \big( \Lambda \dNi^2 + 2\sigma B \big) ~ \bigg| ~ \big| \cT_i \big| = T_i \bigg)\\
		&\qquad\stackrel{(a)}{\leq} \bbP\bigg( \sup_{t \in \cT_{\Nnet, \Lambda}} \big| \IdealNest(t) - \phi_N^2(t) \big| > s ~ \bigg| ~ \big| \cT_i \big| = T_i  \bigg)	\\
		&\qquad\stackrel{(b)}{\leq} \sum_{t \in  \cT_{\Nnet, \Lambda}} \Prob{ \big| \IdealNest(t) - \phi_N^2(t) \big| > s ~ \Big| ~ \big| \cT_i \big| = T_i }	\\
		&\qquad\stackrel{(c)}{\leq} 2 \sum_{t \in \cT_{\Nnet, \Lambda}} \exp \bigg(- \frac{T_i s^2}{2} \bigg)		\\
		&\qquad\stackrel{(d)}{\leq} 2 \Nnet \exp \bigg(- \frac{T_i s^2}{2} \bigg).
\end{align*}
(a) follows from \eqref{eqn:eps_net.1}; (b) is the result of applying the union bound; (c) follows from Lemma \ref{lem:ideal_true}; and we have
(d) because $|\cT_{\Nnet, \Lambda}| = \Nnet$.
\end{proof}


\subsection{Completing the Proof of Lemma \ref{lem:Ephi}}\label{sec:F5.complete}
\begin{proof}[Proof of Lemma \ref{lem:Ephi}]
We want to establish a uniform upper bound on $\big| \hatNest (t) - \phi_N(t) \big| $. Since $\phi_N(t) > 0$ by the supersmoothness assumption 
(see \eqref{eqn:model_supersmooth}) and $ \hatNest (t) \geq 0$ by its construction (see \eqref{eqn:chN_est}), we can observe that
\begin{align*}
	\big| \hatNest (t) - \phi_N(t) \big|^2 
		&\leq	\big| \hatNest (t) + \phi_N(t) \big|\big| \hatNest (t) - \phi_N(t) \big|
		= \big| \hatNest (t)^2 - \phi_N(t)^2 \big|
		=	\big| |\HatNest (t)| - \phi_N(t)^2  \big| \\
		&\leq	\big| \HatNest (t) - \phi_N(t)^2  \big|	\\
		&\leq \big| \HatNest(t) - \IdealNest(t) \big| + \big| \IdealNest(t)  - \phi_N(t)^2 \big|
\end{align*}
for all $t \in \Reals$. Taking the supremum over $t \in [-\Lambda, \Lambda]$, we obtain 
\[
	\sup_{t \in [-\Lambda, \Lambda]} \big| \hatNest (t) - \phi_N(t) \big|^2  
		\leq \sup_{t \in [-\Lambda, \Lambda]} \big| \HatNest(t) - \IdealNest(t) \big| 
		+ \sup_{t \in [-\Lambda, \Lambda]} \big| \IdealNest(t)  - \phi_N(t)^2 \big|.
\]

Applying the union bound, it follows that for any $s'_1, s'_2 > 0$,
\begin{align*}
	\Prob{ \sup_{t \in [-\Lambda, \Lambda]}  \big| \hatNest (t) - \phi_N(t) \big|^2  > s'_1 + s'_2 }
		&\leq \Prob{ \sup_{t \in [-\Lambda, \Lambda]}  \big| \HatNest(t) - \IdealNest(t) \big| > s'_1}\\
			&\quad+\Prob{ \sup_{t \in [-\Lambda, \Lambda]}  \big| \IdealNest(t)  - \phi_N(t)^2 \big| > s'_2}.
\end{align*}
We conclude the proof by applying Lemma \ref{lem:sup_noise} and Lemma \ref{lem:sup_ideal} with the choice of 
\[	
	s'_1 = s_1 +  \frac{\Lambda^2}{\Nnet}  \Big(2 \dNi + \dAi \Big) \dAi		, \qquad\text{and}\qquad
	s'_2 = s_2 + \frac{\Lambda}{\Nnet} \big( \Lambda \dNi^2 + 2\sigma B \big).
\]
\end{proof}

\section{Proof of Proposition \ref{prop:quantile_noisy}}\label{sec:proof_col_noisy}
\subsection{Helper Lemma}
Recall that we defined $\constaa = \frac{1}{ \lipmin } \big( \fmax - \fmin + 2\sigma \big)$.
\begin{lemma}\label{lem:helper_w}
	For $j, j' \in [n]$, let $W_{j, j'} $ denote the Bernoulli random variable such that 
	\[	W_{j, j'} = 1	\quad\text{if and only if}\quad 
			\Big| \big[ \heavi{ Z_{\marg}(j) - Z_{\marg}(j') } - \heavi{ \fcol{j} - \fcol{j'} } \big] \Big| \neq 0.
	\] 
	Then
	\[
		\Prob{W_{j,j'} =1 ~\Big|~ |\cB^j|=k_1, |\cB^{j'}| = k_2 } \leq 4 \sqrt{2 \pi} \constaa \bigg( \frac{1}{\sqrt{k_1}} + \frac{1}{\sqrt{k_2}} \bigg).
	\]
\end{lemma}

\begin{proof}
Let $g_{\marg}(y) := \int_0^1 g(x,y) dx$. Note that $g_{\marg}$ is $(\lipmin, \lipmax)$-biLipschitz, and hence, invertible. 
For $j \in [n]$, let $\zeta_j = g_{\marg}^{-1}\left( Z_{\marg}(j)\right)$ for the purpose of analysis. 
Note that $\zeta_j$ are quantities that are solely used for analysis.

Next, we note that $W_{j, j'} = 1$ if and only if $\sign \big( Z_{\marg}(j) - Z_{\marg}(j') \big) \neq \sign \big( \fcol{j} - \fcol{j'} \big)$ by definition of $W_{j,j'}$.  
Moreover, $\sign \big( Z_{\marg}(j) - Z_{\marg}(j') \big) = \sign \big( \zeta_j - \zeta_{j'} \big)$ because $g_{\marg}$ is strictly monotone increasing. 
Therefore, we focus on identifying the probability of the event that $\sign \big( \zeta_j - \zeta_{j'} \big) \neq  \sign \big( \fcol{j} - \fcol{j'} \big)$. 

For each $j \in [n]$, define $X_j := \zeta_j - \fcol{j}$. Since $g_{\marg}$ is $( \lipmin, \lipmax )$-biLipschitz, for any $s > 0$,
\begin{align}
	\Prob{X_j \geq s}	&\leq \Prob{ g_{\marg}\big( \zeta_j \big) - g_{\marg}\big(\fcol{j}\big) \geq \lipmin s}	\nonumber\\
		&\leq \Prob{ Z_{\marg}(j) - g_{\marg}\big(\fcol{j}\big) \geq \lipmin s}		\nonumber\\
		&= \Prob{	\frac{1}{|\cB^j|} \sum_{i' \in \cB^j} Z(i',j) - g_{\marg}\big(\fcol{j}\big) \geq \lipmin s }		\nonumber\\
		&\leq \Prob{	\frac{1}{|\cB^j|} \sum_{i' \in \cB^j} A(i',j) - g_{\marg}\big(\fcol{j}\big) \geq \frac{\fmax - \fmin }{\fmax - \fmin + 2 \sigma}\lipmin s }
			+ \Prob{	\frac{1}{|\cB^j|} \sum_{i' \in \cB^j} N(i',j) \geq \frac{2\sigma}{\fmax - \fmin + 2\sigma}\lipmin s }	\nonumber\\
		&\leq 2 \exp\left( - \frac{\big| \cB^j \big| \lipmin^2 s^2}{2( \fmax - \fmin + 2\sigma )^2} \right).	\label{eqn:conc.1}
\end{align}
Here, all the probabilities are conditional probabilities conditioned on $|\cB^j|$. We can achieve the same upper bound for $\Prob{X_j \leq -s}$. 

Since $X_j - X_{j'} = \big( \zeta_j - \zeta_{j'} \big) - \big( \fcol{j} - \fcol{j'} \big)$, we can see that 
$\sign \big( \zeta_j - \zeta_{j'} \big) \neq  \sign \big( \fcol{j} - \fcol{j'} \big)$ if and only if
\begin{align*}
	\begin{cases}
		X_j - X_{j'}	< - \big( \fcol{j} - \fcol{j'}\big),	& \text{when } \fcol{j} - \fcol{j'} \geq 0,\\
		X_j - X_{j'} > - \big( \fcol{j} - \fcol{j'}\big),	& \text{when } \fcol{j} - \fcol{j'} < 0. 
	\end{cases}
\end{align*}
Given $\fcol{j}$, observe that $\Prob{\fcol{j} - \fcol{j'} \geq 0} = \fcol{j}$ for any $j' \neq j$. Therefore, by the law of total probability, we can write
\begin{align}
	&\Prob{ \sign\big( \zeta^{(j)} - \zeta^{(j')} \big) \neq \sign\big( \fcol{j} - \fcol{j'} \big) ~\Big|~ |\cB^j| =k_1, |\cB^{j'}| =k_2 }	\nonumber\\
		&\qquad= \Prob{\fcol{j} - \fcol{j'} \geq 0 }
			\Prob{ X_j - X_{j'} < - \left( \fcol{j} - \fcol{j'}\right) \Big|~ \fcol{j} - \fcol{j'} \geq 0, |\cB^j| =k_1, |\cB^{j'}| =k_2} 	\label{eqn:proof_noisy_term.3}\\
		&\qquad\quad+  \Prob{\fcol{j} - \fcol{j'} < 0 }
			\Prob{ X_j - X_{j'} > - \left( \fcol{j} - \fcol{j'}\right) \Big|~ \fcol{j} - \fcol{j'} < 0, |\cB^j| =k_1, |\cB^{j'}| =k_2}.	\label{eqn:proof_noisy_term.4}
\end{align}
Note that $\Prob{\fcol{j} - \fcol{j'} \geq 0 } = \Prob{\fcol{j} - \fcol{j'} \geq 0 \big|~  |\cB^j| =k}$ by the independence between $\fcol{j}, \fcol{j'}$ and $M$.

Next, we establish an upper bound on \eqref{eqn:proof_noisy_term.3}. Since $X_j - X_{j'} < -2\tau$ implies either $X_j < -\tau$ or $X_{j'} > \tau$,
the conditional probability in \eqref{eqn:proof_noisy_term.3} can be upper bounded by
\begin{align*}
	&\Prob{ X_j - X_{j'} < - \left( \fcol{j} - \fcol{j'}\right) \bigg|~ \fcol{j} - \fcol{j'} \geq 0, |\cB^j| =k_1, |\cB^{j'}| =k_2}\\
	 	&\qquad\leq \Prob{ X_j  < - \frac{1}{2} \big( \fcol{j} - \fcol{j'} \big) \bigg|~ \fcol{j} - \fcol{j'} \geq 0, |\cB^j| =k_1 }\\
	 	&\qquad\quad 	+  \Prob{ X_{j'} > \frac{ 1}{2} \big(\fcol{j} - \fcol{j'} \big) \bigg|~ \fcol{j} - \fcol{j'} \geq 0, |\cB^{j'}| =k_2}.
\end{align*}
%

We obtain an upper bound on Eq. \eqref{eqn:proof_noisy_term.3} by finding upper bounds on each terms and then taking the union bound.
For that purpose, we observe that $\frac{d}{d\tau}\Prob{ \fcol{j} - \fcol{j'} \leq 2\tau ~\big|~ \fcol{j} - \fcol{j'} \geq 0 } = \frac{2}{\fcol{j}} \bI \big\{0 \leq \tau \leq \frac{\fcol{j}}{2}  \big\}$ 
and $\Prob{\fcol{j} - \fcol{j'} \geq 0 }=\fcol{j}$.
\begin{align}
	&\Prob{\fcol{j} - \fcol{j'} \geq 0 }
			\Prob{ X_j < - \frac{1}{2} \left( \fcol{j} - \fcol{j'}\right) \Big|~ \fcol{j} - \fcol{j'} \geq 0, |\cB^j| =k_1}	\nonumber\\
	&= \Prob{\fcol{j} - \fcol{j'} \geq 0 } \int_{\tau} \Prob{ X_j  <  -\tau ~\big|~ \fcol{j} - \fcol{j'} = 2\tau, |\cB^j| =k_1 } 
			\frac{d}{d\tau}\Prob{ \fcol{j} - \fcol{j'} \leq 2\tau ~\big|~ \fcol{j} - \fcol{j'} \geq 0 } d \tau	\nonumber\\
	&= 2\int_0^{\frac{\fcol{j}}{2}}  \Prob{ X_j  < - \tau ~\big|~ \fcol{j} - \fcol{j'} = 2\tau, |\cB^j| =k_1 } d\tau		\nonumber\\
	&\stackrel{(a)}{=} 2\int_0^{\frac{\fcol{j}}{2}}  \Prob{ X_j  < - \tau ~\big|~  |\cB^j| =k_1 } d\tau	 \nonumber\\
	&\stackrel{(b)}{\leq} 4\int_0^{\frac{\fcol{j}}{2}} \exp\left( - \frac{k_1 \tau^2}{ 2 \constaa^2} \right) d\tau	\nonumber\\
	&\leq 4\int_0^{\infty} \exp\left( - \frac{k_1 \tau^2}{2 \constaa^2} \right) d\tau	\nonumber\\
	&\stackrel{(c)}{=} \frac{ 2\sqrt{2 \pi} \constaa}{\sqrt{k_1}}.				\label{eqn:proof_noisy_term.5}
\end{align}
(a) follows from the observation that $X_j$ is independent of $\fcol{j'} $; (b) follows from \eqref{eqn:conc.1}; and (c) follows from the identity 
$\int_0^{\infty} e^{-ax^2} dx = \frac{1}{2}\sqrt{\frac{\pi}{a}}$.

We can obtain an upper bound for the other half of \eqref{eqn:proof_noisy_term.3} in a similar fashion.
\begin{equation}\label{eqn:proof_noisy_term.6}
	\Prob{\fcol{j} - \fcol{j'} \geq 0 } \Prob{ X_{j'} > \frac{1}{2} \left( \fcol{j} - \fcol{j'}\right) \Big|~ \fcol{j} - \fcol{j'} \geq 0, |\cB^{j'}| =k_2} 
		\leq \frac{ 2\sqrt{2\pi} \constaa}{\sqrt{k_2}}.
\end{equation}

Using \eqref{eqn:proof_noisy_term.5} and \eqref{eqn:proof_noisy_term.6}, we can find an upper bound on the term in \eqref{eqn:proof_noisy_term.3} as
\begin{align*}
	\Prob{\fcol{j} - \fcol{j'} \geq 0 }\Prob{ X_j - X_{j'} < - \left( \fcol{j} - \fcol{j'}\right) ~\Big|~ \fcol{j} - \fcol{j'} \geq 0, |\cB^j| =k_1, |\cB^{j'}| =k_2}
		&\leq   2\sqrt{2\pi} \constaa \bigg( \frac{1}{\sqrt{k_1}} + \frac{1}{\sqrt{k_2}} \bigg).
\end{align*}
		
We can obtain the same upper bound on the term in \eqref{eqn:proof_noisy_term.4} by noticing that
\begin{align*}
	&\Prob{\fcol{j} - \fcol{j'} < 0 } = 1-\fcol{j},	\qquad\text{and}\\
	&\frac{d}{d\tau} \Prob{ \fcol{j} - \fcol{j'} \geq -2\tau ~\big|~ \fcol{j} - \fcol{j'} < 0 } = \frac{2}{1 - \fcol{j}} \Ind{0 \leq \tau \leq \frac{1 - \fcol{j}}{2} }.
\end{align*}
		
Consequently, we can conclude that
\begin{align*}
	\Prob{ W_{j,j'} = 1 ~\Big|~ |\cB^j| =k_1, |\cB^{j'}| =k_2 } 
		&= \Prob{ \sign\left( \zeta_j - \zeta_{j'} \right) \neq \sign\left( \fcol{j} - \fcol{j'} \right)  ~\Big|~ |\cB^j| =k_1, |\cB^{j'}| =k_2 }\\
		&\leq 4 \sqrt{2 \pi} \constaa \bigg( \frac{1}{\sqrt{k_1}} + \frac{1}{\sqrt{k_2}} \bigg) .
\end{align*}
\end{proof}

\subsection{Completing the Proof of Proposition \ref{prop:quantile_noisy}}

\begin{proof}[Proof of Proposition \ref{prop:quantile_noisy}]
Recall the definition of $\hat{q}_{\marg}(j)$ from \eqref{eqn:Z_marg} and \eqref{eqn:estimate_marg}: for $j \in [n]$, we defined
\[
	\hat{q}_{\marg}(j) = \frac{1}{n}\sum_{j' = 1}^n \heavi{ Z_{\marg}(j) - Z_{\marg}(j') },
\]
where
\[
	Z_{\marg}(j) = \begin{cases}
		\frac{ \sum_{i=1}^m M(i,j) Z(i,j)}{\sum_{i=1}^m M(i,j)},	& \text{if } \cB^j \neq \emptyset\\
		\frac{1}{2},									&	\text{if } \cB^j = \emptyset.
	\end{cases}
\]
For the purpose of analysis, we define an imaginary estimator for $\fcol{j}$ as
\[
	\hat{q}_{*}(j) = \frac{1}{n}\sum_{j' = 1}^n \heavi{ \fcol{j} - \fcol{j'} }.
\]
		
By triangle inequality, the error in quantile estimation is upper bounded as
\[	\left| \hat{q}_{\marg}(j) - \fcol{j} \right| \leq \Big| \hat{q}_{\marg}(j) - \hat{q}_*(j) \Big| + \left| \hat{q}_*(j) - \fcol{j}\right|.	\]
If both $\big| \hat{q}_{\marg}(j) - \hat{q}_*(j) \big| \leq t_1$ and $\big| \hat{q}_*(j) - \fcol{j}\big| \leq t_2$ are satisfied, then 
$\big| \hat{q}_{\marg}(j) - \fcol{j} \big| \leq t_1 + t_2$. Therefore, for any $t_1, t_2 > 0$,
\begin{align}
	\Prob{ \big| \hat{q}_{\marg}(j) - \fcol{j} \big| > t_1 + t_2 }	
		&\leq \mathbb{P}\Big({\big| \hat{q}_{\marg}(j) - \hat{q}_*(j) \big| > t_1 }\Big) + \Prob{ \big| \hat{q}_*(j) - \fcol{j} \big| > t_2 }.	\label{eqn:partition1}	
\end{align}
It is easy to verify that $\hat{q}_{*}(j)$ exponentially concentrates to $\fcol{j}$ as $n \to \infty$, e.g., by McDiarmid's inequality:
\begin{equation}\label{eqn:quant_a}
	 \Prob{ \big| \hat{q}_*(j) - \fcol{j} \big| > t_2 }  	\leq 2 \exp \left( - 2 nt_2^2 \right).
\end{equation}
Therefore, it suffices to establish an upper bound for the first term in \eqref{eqn:partition1}, i.e., a probabilistic tail upper bound for 
$\left| \hat{q}_{\marg}(j) - \hat{q}_*(j) \right|$.

We observe that
\begin{align*}
	\left| \hat{q}_{\marg}(j) - \hat{q}_*(j) \right| 
		&= \bigg| \frac{1}{n} \sum_{j' = 1}^n  \left[ \heavi{ Z_{\marg}(j) - Z_{\marg}(j') } - \heavi{ \fcol{j} - \fcol{j'} } \right] \bigg|\\
		&\leq \frac{1}{n} \sum_{j' = 1}^n \bigg| \left[ \heavi{ Z_{\marg}(j) - Z_{\marg}(j') } - \heavi{ \fcol{j} - \fcol{j'} } \right] \bigg|.
\end{align*}
For each pair $(j, j') \in [n]^2$, define a Bernoulli random variable $W_{j, j'}$ such that 
\[
	W_{j, j'} = 1 \quad\text{if and only if}\quad
	\Big| \big[ \heavi{ Z_{\marg}(j) - Z_{\marg}(j') } - \heavi{ \fcol{j} - \fcol{j'} } \big] \Big| \neq 0.
\]
Then we can observe that $0 \leq \Big| \big[ \heavi{ Z_{\marg}(j) - Z_{\marg}(j') } - \heavi{ \fcol{j} - \fcol{j'} } \big] \Big| \leq W_{j,j'}$ and therefore,
\begin{align*}
	\mathbb{P}\Big( \left| \hat{q}_{\marg}(j) - \hat{q}_*(j) \right| > t_1 \Big) 
		& \leq \mathbb{P} \bigg( \sum_{j'=1}^n W_{j,j'} > nt_1 \bigg).
\end{align*}
By Lemma \ref{lem:helper_w}, we have
\[
	\Prob{W_{j,j'} =1 ~\Big|~ |\cB^j|=k_1, |\cB^{j'}| = k_2 } \leq 4 \sqrt{2 \pi} \constaa \bigg( \frac{1}{\sqrt{k_1}} + \frac{1}{\sqrt{k_2}} \bigg).
\]
Therefore, we may write
\[
	\Prob{W_{j,j'} =1 ~\Big|~ M } \leq \frac{8 \sqrt{2 \pi} \constaa}{\sqrt{k_*}}.
\]
for all $j, j' \in [n]$ with $k_* = \min_{j' \in [n]} | \cB^{j'}|$.

Applying the binomial Chernoff bound,
\begin{align}
	\mathbb{P} \bigg( \sum_{j'=1}^n W_{j,j'} > nt_1  ~\Big|~ M  \bigg)
		&= \mathbb{P} \bigg( \sum_{j'=1}^n W_{j,j'} - \bbE \Big[ \sum_{j'=1}^n W_{j,j'} \Big] > nt_1 - \bbE \Big[ \sum_{j'=1}^n W_{j,j'} \Big]  ~\Big|~ M  \bigg)	\nonumber\\
		&\leq \mathbb{P} \bigg( \sum_{j'=1}^n W_{j,j'} - \bbE \Big[ \sum_{j'=1}^n W_{j,j'} \Big] > n \bigg( t_1 - \frac{8 \sqrt{2 \pi} \constaa}{\sqrt{k_*} } \bigg)  ~\Big|~ M  \bigg)	\nonumber\\
		&\leq \exp \left( -2 n \bigg( t_1 - \frac{8 \sqrt{2 \pi} \constaa}{\sqrt{k_*} } \bigg)^2  \right).	\label{eqn:quant_b}
\end{align}

All in all, we can conclude that for $t > 0$, 
\begin{align*}
	&\Prob{ \big| \hat{q}_{\marg}(j) - \fcol{j} \big| > t + \frac{8 \sqrt{2 \pi} \constaa}{\sqrt{ \min_{j' \in [n]} | \cB^{j'} | } } ~\bigg|~ M }\\
		&\qquad\leq \Prob{ \left| \hat{q}_{\marg}(j) - \hat{q}_*(j) \right| > \frac{t}{2} + \frac{8 \sqrt{2 \pi} \constaa}{\sqrt{ \min_{j' \in [n]} | \cB^{j'} | } } ~\bigg|~ M }
			+  \Prob{ \big| \hat{q}_*(j) - \fcol{j} \big| > \frac{t}{2} ~\bigg|~ M } \\
		&\qquad\leq 3\exp \left( -\frac{nt^2}{2} \right)
\end{align*}
by plugging \eqref{eqn:quant_a} and \eqref{eqn:quant_b} back to \eqref{eqn:partition1} with the choice of 
$t_1 = \frac{t}{2} + \frac{8 \sqrt{2 \pi} \constaa}{\sqrt{ \min_{j' \in [n]} | \cB^{j'} | } }$ and $t_2 = \frac{t}{2}$.

\end{proof}
\section{Proof of Corollary \ref{coro:answer.2}}\label{sec:proof_ME}

\subsection{Helper Lemma}
In this section, we establish a probabilistic tail bound on $ \sup_{i \in [m]} \sup_{j \in [n]} | \hat{A}(i,j) - A(i,j) | $.

\begin{lemma}\label{lem:prob_max_norm}
For $i \in [m]$, let $\hat{F}_i$ be defined as in \eqref{eqn:ECDF_unknown_noise} with $\hat{\phi}_{N}(t) = \hat{\phi}_{N, i}(t)$ as described 
in Section \ref{sec:noise_estimation}, cf. \eqref{eqn:chN_est}. Suppose that the kernel bandwidth 
$h = (4\gamma)^{\frac{1}{\beta}}(\log |\cB_i|)^{-\frac{1}{\beta}}$ and the ridge parameter $\rho = \frac{1}{B} |\cB_i|^{-\frac{9}{20}}$. 
For $j \in [n]$, let $\hat{q}_{\marg}(j)$ be defined as in \eqref{eqn:estimate_marg}.

 If $|\cB_i| \geq 1024$ and $mp$ and $n$ are sufficiently large so that $\valphiAA + \valphiBB \leq \frac{1}{B} |\cB_i|^{-\frac{9}{20}}$, 
 then for any $t > 0$,
	\begin{align*}
		&\bbP\Bigg(	\sup_{(i, j) \in [m] \times [n] } \big| \hat{F}_i^{-1} \big( \hat{q}_{\marg}(j) \big) - A(i,j)\big| 
			>  t + \constdd 	\Bigg) \\
		&\qquad\leq
			\Big( 2 m ( 2np )^{\frac{9}{20}} \big[ \log (2np) \big]^{\frac{2}{\beta}} + 3n \Big)
			\exp \Bigg( - \frac{1}{2}\bigg( \frac{t}{\constee } \bigg)^2 \Bigg)\\
		&\qquad\quad
			+ \frac{3}{n^7} + \frac{6}{m^7n^7} + m \exp \left( - \frac{np}{8} \right) + 2n \exp \left( - \frac{mp}{8} \right) + \exp \left( - \frac{m}{16} \right) + 2 \exp \left( - \frac{n}{16} \right).
	\end{align*}
	where
	\begin{align*}
		\constdd	&=	\lipmax \bigg\{ \big( \constbb + \constcc \big) \Big[ \log \big( \frac{np}{2} \big) \Big]^{-\frac{1}{\beta}} 
				+ 4\constcc  \frac{\big[ \log ( 2np) \big]^{\frac{1}{\beta}}}{ \big( \frac{np}{2} \big)^{\frac{1}{5}} } 
				+  \frac{8 \sqrt{2 \pi} \constaa}{\sqrt{ \frac{mp}{2} } } \bigg\}\\
		\constee	&=	\lipmax \Big[ \frac{\constcc [ \log (2np) ]^{\frac{1}{\beta}}}{( \frac{np}{2} )^{\frac{1}{20}}} + \frac{1}{\sqrt{n}} \Big].
	\end{align*}
\end{lemma}

\begin{proof}
Fix $(i,j) \in [m] \times [n]$. Let $\theta^* \equiv F_i \Big( \hat{F}_i^{-1} \big( \hat{q}_{\marg}(j) \big) \Big)$. 
Since $\hat{F}_i^{-1} \big( \hat{q}_{\marg}(j) \big) = g\big( \frow{i}, \theta^* \big)$ and $A(i,j) = g \big( \frow{i}, \fcol{j} \big)$,
\begin{align*}
	\Big| \hat{F}_i^{-1} \big( \hat{q}_{\marg}(j) \big) - A(i,j) \Big|
		&= \Big| g\big( \frow{i}, \theta^* \big) - g \big( \frow{i}, \fcol{j} \big) \Big|\\
		&\stackrel{(a)}{\leq} \lipmax \big| \theta^* - \fcol{j} \big|\\
		&\leq \lipmax \Big( \big| \theta^* - \hat{q}_{\marg}(j) \big| + \big| \hat{q}_{\marg}(j) - \fcol{j} \big| \Big)\\
		&\stackrel{(b)}{\leq} \lipmax \Big( \sup_{z \in [\fmin, \fmax] } \big| \hat{F}_i (z) - F_i(z) \big| + \big| \hat{q}_{\marg}(j) - \fcol{j} \big| \Big)
\end{align*}
where (a) follows from the assumption that $g$ is $(\lipmin, \lipmax)$-bi-Lipschitz and (b) is the result of the following observation: since 
$ \hat{F}_i^{-1} \big( \hat{q}_{\marg}(j) \big) \in [\fmin, \fmax]$ by definition of $\hat{F}_i$, and therefore,
\[
	\big| \hat{q}_{\marg}(j) - \theta^*\big| 	\leq \sup_{z \in [\fmin, \fmax] } \big| \hat{F}_i (z) - F_i(z) \big|.
\]

Observe that
\[
	\sup_{z \in [\fmin, \fmax] } \big| \hat{F}_i (z) - F_i(z) \big| \leq s_1	\quad\text{and}\quad
	\big| \hat{q}_{\marg}(j) - \fcol{j} \big| \leq s_2	\qquad\implies\qquad 
	\big| \hat{F}_i^{-1} \big( \hat{q}_{\marg}(j) \big) - A(i,j)\big| \leq \lipmax \big( s_1 + s_2 \big).
\]
The contraposition of the above proposition reads as
\begin{align*}
	&\exists(i,j) \in [m] \times [n] 	\quad\text{such that}\quad		\big| \hat{F}_i^{-1} \big( \hat{q}_{\marg}(j) \big) - A(i,j)\big| > \lipmax \big( s_1 + s_2 \big)	\\
	&\implies	
		\exists i \in [m]	\quad\text{such that}\quad 	\sup_{z \in [\fmin, \fmax] } \big| \hat{F}_i (z) - F_i(z) \big| > s_1	\qquad\text{or}\qquad
		\exists j \in [n]	\quad\text{such that}\quad 	\big| \hat{q}_{\marg}(j) - \fcol{j} \big| > s_2.	 
\end{align*}
Therefore, for any $s_1, s_2 \geq 0$,
\begin{align}
	&\bbP\bigg(	\sup_{(i, j) \in [m] \times [n] } \big| \hat{F}_i^{-1} \big( \hat{q}_{\marg}(j) \big) - A(i,j)\big| > \lipmax \big( s_1 + s_2 \big)	\bigg)	\nonumber\\
	&\qquad\leq
		 \bbP\bigg(	\sup_{i \in [m]}	\sup_{z \in [\fmin, \fmax] } \big| \hat{F}_i (z) - F_i(z) \big| > s_1	\bigg)
		 +	\bbP\bigg(		\sup_{j \in [n]}	\big| \hat{q}_{\marg}(j) - \fcol{j} \big| > s_2	\bigg).	\label{eqn:term.ME.a}
\end{align}
It remains to further simplify \eqref{eqn:term.ME.a} with an appropriate choice of $s_1$ and $s_2$.

We pause and define a new event for conditioning. Recall that we defined $\Erow := \cap_{i=1}^m \{ \frac{np}{2} \leq |\cB_i| \leq 2np \}$ and observed 
$ \Prob{\Erow^c} \leq 2m \exp \big( -\frac{np}{8} \big)$ in the proof of Corollary \ref{coro:answer.1}, cf. \eqref{eqn:upper_row}. 
Let $\Ecol := \cap_{j=1}^n \{ |\cB^j| \geq \frac{mp}{2} \}$. We observe that $|\cB^j| = \sum_{i=1}^m \Ind{M_{ij} = 1}$ is the sum of $m$ independent 
Bernoulli random variables for each $j \in [n]$. We have $\Prob{ |\cB^j| < \frac{mp}{2} } \leq \exp \big( -\frac{mp}{8} \big)$ by the binomial Chernoff bound. 
Applying the union bound, 
\begin{equation}\label{eqn:upper_col}
	\Prob{\Ecol^c} \leq \sum_{j=1}^n \Prob{ |\cB^j| < \frac{mp}{2} }  \leq n \exp \Big( -\frac{mp}{8} \Big).
\end{equation}

With this observation, we further simplify \eqref{eqn:term.ME.a} as
\begin{align}
	& \bbP\bigg(	\sup_{i \in [m]}	\sup_{z \in [\fmin, \fmax] } \big| \hat{F}_i (z) - F_i(z) \big| > s_1	\bigg)
		 +	\bbP\bigg(		\sup_{j \in [n]}	\big| \hat{q}_{\marg}(j) - \fcol{j} \big| > s_2	\bigg)	\nonumber\\
	&\qquad\leq
		\bbP\bigg(	\sup_{i \in [m]}	\sup_{z \in [\fmin, \fmax] } \big| \hat{F}_i (z) - F_i(z) \big| > s_1	~\Big|~	\Ephi \cap \Erow \bigg) +	\Prob{ \Ephi^c \cup \Erow^c}	\nonumber\\
		&\qquad\qquad	+	\bbP\bigg(		\sup_{j \in [n]}	\big| \hat{q}_{\marg}(j) - \fcol{j} \big| > s_2	~\Big|~	\Ecol	\bigg) + \Prob{\Ecol^c}	\nonumber\\
	&\qquad\leq
		\sum_{i=1}^m \bbP\bigg( \sup_{z \in [\fmin, \fmax] } \big| \hat{F}_i (z) - F_i(z) \big| > s_1	~\Big|~	\Ephi \cap \Erow \bigg)
		+	\sum_{j=1}^n \bbP\bigg(	\big| \hat{q}_{\marg}(j) - \fcol{j} \big| > s_2	~\Big|~	\Ecol	\bigg)	\label{eqn:term.ME.b}\\
		&\qquad\qquad		+ 	\Prob{ \Ephi^c}	+ \Prob{\Erow^c} + \Prob{\Ecol^c}.		\nonumber
\end{align}

Let $\gamma$ denote a parameter in $[0,1]$ and let 
\begin{align*}
	s_1 &= \gamma \frac{ t}{\lipmax} + \big( \constbb + \constcc \big) \Big[ \log \big( \frac{np}{2} \big) \Big]^{-\frac{1}{\beta}} 
			+ 4\constcc  \frac{\big[ \log ( 2np) \big]^{\frac{1}{\beta}}}{ \big( \frac{np}{2} \big)^{\frac{1}{5}} }	\qquad\text{and}\\
	s_2 &= (1- \gamma )\frac{ t}{\lipmax} + \frac{8 \sqrt{2 \pi} \constaa}{\sqrt{ \frac{mp}{2} } }.
\end{align*}
With the choice of $s_1, s_2$, we obtain the following upper bound on \eqref{eqn:term.ME.b}:
\begin{align}
	&\sum_{i=1}^m \bbP\bigg( \sup_{z \in [\fmin, \fmax] } \big| \hat{F}_i (z) - F_i(z) \big| > s_1	~\Big|~	\Ephi \cap \Erow \bigg)
		+	\sum_{j=1}^n \bbP\bigg(	\big| \hat{q}_{\marg}(j) - \fcol{j} \big| > s_2	~\Big|~	\Ecol	\bigg)	\nonumber\\
	&\qquad\leq
		2 m ( 2np )^{\frac{9}{20}} \big[ \log (2np) \big]^{\frac{2}{\beta}}
				\exp\Bigg( - \frac{ \big( \frac{np}{2} \big)^{\frac{1}{10}} }{2 \constcc^2 \big[ \log (2np) \big]^{\frac{2}{\beta}} }  \frac{ \gamma^2 t^2}{\lipmax^2}  \Bigg)
		+	3n \exp \bigg( -\frac{n}{2}  \frac{ (1-\gamma)^2 t^2}{\lipmax^2} \bigg).		\label{eqn:term.ME.c}
\end{align}
Now we choose $\gamma \in [0,1]$ so that the two terms in the upper bound in \eqref{eqn:term.ME.c} are balanced. Equating the exponents in the two terms, 
we obtain a quadratic equation in $\gamma$. Letting $C_1 := \frac{ \big( \frac{np}{2} \big)^{\frac{1}{10}} }{2 \constcc^2 \big[ \log (2np) \big]^{\frac{2}{\beta}} }$ and 
$C_2 := \frac{n}{2}$, we may write the quadratic equation as $- A \gamma^2 t^2 = - B (1-\gamma)^2 t^2$, or equivalently, $(B-A)\gamma^2 - 2B \gamma + B = 0$. 
Since $\gamma \in [0,1]$, this equation admits one valid root:
\begin{align*}
	\gamma &= \frac{B - \sqrt{AB}}{B-A}
		= \frac{\sqrt{B}}{\sqrt{A} + \sqrt{B}}.
\end{align*}
With the choice of $\gamma$, \eqref{eqn:term.ME.c} simplifies to 
\begin{equation}\label{eqn:term.ME.d}
	\Big( 2 m ( 2np )^{\frac{9}{20}} \big[ \log (2np) \big]^{\frac{2}{\beta}} + 3n \Big)
		\exp \Bigg( - \frac{1}{2}\bigg( \frac{t}{\constee } \bigg)^2 \Bigg).
\end{equation}
where $\constee =  \lipmax \Big[ \frac{\constcc [ \log (2np) ]^{\frac{1}{\beta}}}{( \frac{np}{2} )^{\frac{1}{20}}} + \frac{1}{\sqrt{n}} \Big]$.

With \eqref{eqn:term.ME.d} as an upper bound on \eqref{eqn:term.ME.b} and the upper bounds on $\Prob{ \Ephi^c}	+ \Prob{\Erow^c} + \Prob{\Ecol^c}$ 
from Theorem \ref{thm:ensure_condition}, \eqref{eqn:upper_row}, and \eqref{eqn:upper_col}, we can complete the proof using  \eqref{eqn:term.ME.a}.
\end{proof}

\subsection{Proof of Corollary \ref{coro:answer.2}}
\begin{proof}[Proof of Corollary \ref{coro:answer.2}]
Letting $\hat{A} = \psi(Z)$, we have $\hat{A}(i,j) = \hat{F}_i^{-1} \big( \hat{q}_{\marg}(j) \big)$ for $(i, j) \in [m] \times [n]$. 
We recall the definition of $\RiskME(\psi)$ from \eqref{eqn:loss_psi} and see that
\begin{align*}
	\RiskME(\psi) &=  \bbE_{Z} \bigg[ \sup_{(i,j) \in [m] \times [n]} | \hat{A}(i,j) - A(i,j) |^2 \bigg].
\end{align*}
Since $0 \leq | \hat{A}(i,j) - A(i,j) |^2 \leq (\fmax - \fmin)^2$, it follows that 
\begin{align}
	\bbE_{Z} \bigg[ \sup_{(i, j) \in [m] \times [n]} | \hat{A}(i,j) - A(i,j) |^2 \bigg]
		= \int_{0}^{(\fmax - \fmin)^2} \Prob{ \sup_{(i,j) \in [m] \times [n]} | \hat{A}(i,j) - A(i,j) |^2 > t} dt		\nonumber\\
		= \int_{0}^{\fmax - \fmin} 2 s \Prob{ \sup_{(i,j) \in [m] \times [n]} | \hat{A}(i,j) - A(i,j) | > s} ds	\label{eqn:ME_int.a}
\end{align}
by the changing of variables $s = \sqrt{t}$.

Next, we use the upper bound obtained in Lemma \ref{lem:prob_max_norm} to find an upper bound on \eqref{eqn:ME_int.a}.
\begin{align}
	&\int_{0}^{\fmax - \fmin} 2 s \Prob{ \sup_{(i,j) \in [m] \times [n]} | \hat{A}(i,j) - A(i,j) | > s} ds		\nonumber\\
	 	&\stackrel{(a)}{\leq} \int_0^{\constdd} 2s ~ds + \int_{\constdd}^{\fmax - \fmin} 2 s \Prob{ \sup_{(i,j) \in [m] \times [n]} | \hat{A}(i,j) - A(i,j) | > s} ds		\nonumber\\
		&\stackrel{(b)}{\leq} \int_0^{\constdd} 2s ~ds		\label{eqn:ME_int.b}\\
		&\qquad	+ 2 \Big( 2 m ( 2np )^{\frac{9}{20}} \big[ \log (2np) \big]^{\frac{2}{\beta}} + 3n \Big) 
			\int_{0}^{\fmax - \fmin - \constdd} \big( s + \constdd \big) \exp \Bigg( - \frac{1}{2}\bigg( \frac{s}{\constee } \bigg)^2 \Bigg) ds		\label{eqn:ME_int.c}\\
		&\qquad	+ 2 \bigg[  \frac{3}{n^7} + \frac{6}{m^7n^7} + m \exp \left( - \frac{np}{8} \right) + 2n \exp \left( - \frac{mp}{8} \right) 
			+ \exp \left( - \frac{m}{16} \right) + 2 \exp \left( - \frac{n}{16} \right) \bigg] \int_{\constdd}^{\fmax - \fmin} 2 s ~ds.	\label{eqn:ME_int.d}
\end{align}
Here, (a) follows from the trivial upper bound on probability, i.e., $\Prob{ \sup_{(i,j) \in [m] \times [n]} | \hat{A}(i,j) - A(i,j) | > s} \leq 1$; and (b) follows from the upper bound 
in Lemma \ref{lem:prob_max_norm}.

Now we establish upper bounds on the integral in \eqref{eqn:ME_int.b}, \eqref{eqn:ME_int.c}, and \eqref{eqn:ME_int.d} separately.
\begin{itemize}
	\item
		First, it is easy to compute the integral in \eqref{eqn:ME_int.b}: 
		\begin{align*}
			 \int_0^{\constdd} 2s ~ds = \constdd^2.
		\end{align*}
	
	\item
		Second, we recall the well known facts that for $a > 0$,
		\begin{align*}
			\int_0^{\infty} \exp \bigg( - \frac{s^2}{2a^2} \bigg) ds = a \sqrt{\frac{\pi}{2}}		\qquad\text{and}\qquad
			\int_0^{\infty} s \exp \bigg( - \frac{s^2}{2a^2} \bigg) ds = a^2.
		\end{align*}
		Then we observe that the integral in \eqref{eqn:ME_int.c} is bounded above as
		\begin{align*}
			&\int_{0}^{\fmax - \fmin - \constdd} \big( s + \constdd \big) \exp \Bigg( - \frac{1}{2}\bigg( \frac{s}{\constee } \bigg)^2 \Bigg) ds	\\
				&\qquad\leq \int_{0}^{\infty} \big( s + \constdd \big) \exp \Bigg( - \frac{1}{2}\bigg( \frac{s}{\constee } \bigg)^2 \Bigg) ds	\\
				&\qquad = \constee^2 + \sqrt{\frac{\pi}{2}}\constdd \constee.
		\end{align*}
	\item
		Lastly, we use the following simple upper bound on the integral in \eqref{eqn:ME_int.d}:
		\begin{align*}
			\int_{\constdd}^{\fmax - \fmin} 2 s ~ds
				\leq \int_{0}^{\fmax - \fmin} 2 s ~ds
				= (\fmax - \fmin )^2.
		\end{align*}
\end{itemize}

All in all, we establish the following upper bound:
\begin{align*}
	\RiskME(\psi)
		&\leq \constdd^2 
			+ 2 \Big( 2 m ( 2np )^{\frac{9}{20}} \big[ \log (2np) \big]^{\frac{2}{\beta}} + 3n \Big) \bigg( \sqrt{\frac{\pi}{2}}\constdd + \constee \bigg) \constee\\
		&\quad	+ 2 \big( \fmax - \fmin \big)^2	\bigg[  \frac{3}{n^7} + \frac{6}{m^7n^7} + m \exp \left( - \frac{np}{8} \right) + 2n \exp \left( - \frac{mp}{8} \right) 
			+ \exp \left( - \frac{m}{16} \right) + 2 \exp \left( - \frac{n}{16} \right) \bigg].
\end{align*}
\end{proof}

\section{Some Known Facts from Literature}

\subsection{Well-known Facts about Distribution}\label{appx:distribution}
\subsubsection{Basic Definitions}
In this section, we briefly restate some basic facts about random variables and their associated distributions. 
We let $(\Omega, \cF, P)$ denote the probability space of interest.
\begin{definition}[Random variable]
	A random variable $X: \Omega \to E$ is a measurable function from a set of possible outcomes $\Omega$ to 
	a measurable space $E$. When $E= \Reals$, we call $X$ a real-valued random variable.
\end{definition}

For a real-valued random variable $X$, we can define its distribution function, whose evaluation at $x$ is 
the probability that $X$ will take a value less than or equal to $x$.
\begin{definition}[Cumulative distribution function (CDF)]\label{defn:CDF}
	The cumulative distribution function of a real-valued random variable $X$ is defined as a function 
	$F_X: \Reals \to [0,1]$ such that
	\[	F_X(x) = \Prob{X \leq x}.	\] 
\end{definition}

Every cumulative distribution function $F$ is non-decreasing, right-continuous, $\lim_{x \to -\infty} F(x) = 0$, 
and $\lim_{x \to \infty} F(x) = 1$. Conversely, every function with these four properties is a CDF, i.e., 
a random variable can be defined so that the function is the CDF of that random variable.

We define a pseudo-inverse of the distribution function as follows and call it the quantile function.
\begin{definition}[Quantile function]\label{defn:Quantile}
	Given a distribution function $F: \Reals \to [0,1]$, the associated quantile function $Q:(0,1) \to \Reals$ is defined as
	\[	Q(p) = \inf \left\{ x \in \Reals: p \leq F(x) \right\}.	\]
\end{definition}
If the function $F$ is continuous and strictly monotone increasing, then the infimum can be replaced by the minimum and 
$Q = F^{-1}$, i.e., $p = F(x)$ if and only if $x = Q(p)$.

Note that the CDF can be expressed as the expectation of an indicator function, $F_X(x) = \Exp{\Ind{X \leq x}}$. In particular, 
when $F$ is absolutely continuous, then there exists a Lebesgue-integrable function $f(x)$ such that
\[	F(b) - F(a) = \Prob{a < X \leq b} = \int_a^b f(x) dx,	\]
for all real numbers $a$ and $b$. The function $f$ is the (Radon-Nikodym) derivative of $F$, and it is called the 
probability density function of distribution of $X$.

Also, there is an alternative way to describe a random variable (in the Fourier domain).
\begin{definition}[Characteristic function]\label{defn:ch_ftn}
	The characteristic function $\phi_X: \Reals \to \Cx$ for a real-valued random variable is defined as the expected 
	value of $e^{itX}$, where $i$ is the imaginary unit, and $t \in \Reals$ is the argument of the characteristic function:
	\begin{align*}
		\phi_X(t) &= \Exp{e^{itX}}
			= \int_{\Reals} e^{itx} dF_X(x)
			= \int_{\Reals} e^{itx} f_X(x) dx
			= \int_0^1 e^{itQ_X(p)} dp.
	\end{align*}
\end{definition}

If random variable $X$ has a probability density function $f_X$, then the characteristic function is the Fourier transform 
with sign reversal in the complex exponential (note that the constant differs from the usual convention for the Fourier transform).

\subsubsection{Empirical Distribution}
\begin{definition}[Empirical CDF]
	Suppose that $X_1, \ldots, X_n$ ($n$ is a natural number) are real-valued independent and identically distributed 
	random variables with common cumulative distribution function $F$. We let $F_n$ denote the empirical distribution 
	function associated with $\{X_1, \ldots, X_n\}$, which is defined as
	\[	F_n(x) = \frac{1}{n} \sum_{i=1}^n \Ind{X_i \leq x}, \quad, \forall x \in \Reals.	\]
\end{definition}

It is known that the empirical distribution function converges to the true underlying distribution function, which the samples 
are drawn from. The following concentration results known as the Dvoretzky-Kiefer-Wolfowitz (DKW) inequality quantifies 
the rate of convergence of $F_n$ to $F$ with respect to the uniform norm as $n$ tends to infinity. 
This result strengthens the Glivenko-Cantelli theorem.
\begin{lemma}[Dvoretzky-Kiefer-Wolfowitz]\label{lem:DKW}
	Given a natural number $n$, let $X_1, \ldots, X_n$ be real-valued independent and identically distributed random variables 
	with common cumulative distribution function $F$. Then for every $t > 0$,
	\[	\Prob{\sup_{x \in \Reals} \left| F_n(x) - F(x) \right| > \eps} \leq 2 e^{-2n t^2}.	\]
\end{lemma}

\subsection{Sub-Gaussian Random Variable and the Chernoff Bound}
We define a class of random variables, whose tail behavior is easy to control.
\begin{definition}[Sub-Gaussian random variable]
	A random variable $X$ with mean $\mu = \Exp{X}$ is called sub-Gaussian with parameter $\sigma$ if there is 
	a positive constant $\sigma$ such that
	\[	\bbE\big[e^{\lambda(X-\mu)}\big] \leq e^{\frac{\lambda^2 \sigma^2}{2}}, \quad \forall \lambda \in \Reals.	\]
	We will call $\sigma$ the sub-Gaussian parameter of $X$.
\end{definition}
An application of the Chernoff bound leads to
\[	\Prob{X-\mu \geq t} \leq \inf_{\lambda > 0} \frac{\bbE\big[e^{\lambda(X-\mu)}\big]}{e^{\lambda t}}.	\]
It is possible to achieve the same upper bound for $\Prob{ X - \mu \leq -t } = \Prob{ -(X-\mu) \geq t }$.
We can conclude that a sub-Gaussian random variable satisfies that for all $t \in \Reals$,
\[	\Prob{|X - \mu| \geq t} \leq 2 e^{-\frac{t^2}{2\sigma^2}}.	\]
The class of sub-Gaussian random variables subsumes Gaussian random variable and any bounded random variables.

\subsubsection{Hoeffding-type Inequalities}

Now, we present several forms of concentration inequalities for the sum of independent random variables. Essentially they are 
all Chernoff bounds, tailored to specific random variable assumptions.
\begin{lemma}[Binomial Chernoff bound]\label{lem:Chernoff}
	Let $X = \sum_{i=1}^n X_i$, where $X_i = 1$ with probability $p_i$, and $X_i = 0$ with probability $1 - p_i$, and $X_i$'s 
	are independent. Let $\mu = \Exp{X} = \sum_{i=1}^n p_i$. Then
	\begin{enumerate}
		\item Upper tail:	$ \Prob{X \geq (1+\delta) \mu} \leq \exp\left(-\frac{\delta^2}{2+\delta}\mu \right)$ for all $\delta > 0$.
		\item Lower tail: 	$ \Prob{X \leq (1-\delta) \mu} \leq \exp\left(-\frac{\delta^2}{2}\mu \right)$ for all $0 < \delta < 1$.
	\end{enumerate}
\end{lemma}

There is a more general version of concentration inequality that applies to sub-gaussian random variables.
\begin{lemma}[Hoeffding's inequality for sub-Gaussian ranom variables]\label{lem:Hoeffding_subG}
	Let $X_1, \ldots, X_n$ be $n$ independent random variables such that $X_i$ has mean $\mu_i$ and sub-Gaussian parameter 
	$\sigma_i$ and let $X = \sum_{i=1}^n X_i$. Then for any $t > 0$,
	\[	\Prob{X - \Exp{X} \geq t} \leq \exp\left( - \frac{t^2}{2\sum_{i=1}^n \sigma_i^2} \right).	\]
	The same upper bound holds for $\Prob{X - \Exp{X} \leq -t}$.
\end{lemma}

Oftentimes, Hoeffding's inequality is presented in the following form, which is specialized for bounded random variables.
\begin{lemma}[Hoeffding's inequality for bounded ranom variables]\label{lem:Hoeffding_bounded}
	Let $X_1, \ldots, X_n$ be $n$ independent random variables such that $X_i \in [a_i, b_i]$ almost surely for all $i$ 
	and let $X = \sum_{i=1}^n X_i$. Then for any $t > 0$,
	\[	\Prob{X - \Exp{X} \geq t} \leq \exp\left( - \frac{2t^2}{\sum_{i=1}^n (b_i - a_i)^2} \right).	\]
	The same upper bound holds for $\Prob{X - \Exp{X} \leq -t}$.
\end{lemma}

\subsubsection{Bounded Difference Condition}
Note that the inequalities in the previous section ensure concentration for the sum of independent random variables whose tail behavior 
is well-behaved. It is possible to obtain a similar concentration for a more general class of functions of independent random variables 
as long as the function does not depend on a single random variable too heavily. This is so-called the ``bounded difference'' condition. 
We formally state this result in the following lemma. 
\begin{lemma}[McDiarmid's inequality]\label{lem:McDiarmid}
	Let $X_1, \ldots, X_n$ be independent random variables such that for each $i \in [n]$, $X_i \in X$. Let $\xi: \prod_{i=1}^n X_i \to \Reals$ 
	be a function of $(X_1, \ldots, X_n)$ that satisfies for all $x_1, \ldots, x_n$, for all $i$, and for all $x_i'$,
	\begin{equation}\label{eqn:bounded_difference}
		\left| \xi\left(x_1, \ldots, x_i, \ldots, x_n\right) - \xi\left(x_1, \ldots, x_i', \ldots, x_n\right) \right| \leq c_i.
	\end{equation}
	Then for all $t > 0$,
	\[	\Prob{\xi - \Exp{\xi} \geq t} \leq \exp \left( \frac{-2 t^2}{\sum_{i=1}^n c_i^2} \right).	\]
\end{lemma}
Note that one can obtain the same tail bound for the opposite direction by considering $-\xi$ in lieu of $\xi$.

\subsection{Some Known Results from Deconvolution Literature }\label{appx:deconvolution}
In this section, we introduce some known results for estimating the unknown density $f_X$ of random variable $X$ 
using deconvolution techniques. Suppose that $Z = X+N$ is a measurement of $X$ with additive noise $N$ and that 
we have $n$ i.i.d. observations $Z_1, \ldots, Z_n$. Fan reported that we can achieve an asymptotically consistent estimate 
for the density $f_X$ when the noise density $f_N$ is known and $f_X$ satisfies certain smoothness conditions \cite{Fan1991}. 
Later, Delaigle et al. showed that consistent estimation is possible even when the noise distribution is unknown, 
with aid of repeated measurements \citet{Delaigle2008}.

Their estimators and proof techniques rely on the kernel smoothing method (kernel deconvolution estimator). Here we only 
present the abridged version of the concepts, the estimator, and the results to the minimum amount we need. We would refer 
interested readers to relevant references for more details; for example, \citet{Carroll1988, Fan1991, Delaigle2008}. 
	
\subsubsection{Deconvolution Kernel Density Estimator}
Our goal is to recover distribution of random variable $X$, but we observe samples of $Z = X + N$ instead of $X$. We assume 
we know the distribution of $N$. Due to the independence between $X$ and $N$, we know that $\phi_Z(t) = \phi_X(t) \phi_N(t)$ 
for all $t \in \Reals$, where $\phi_Z, \phi_X, \phi_N$ denote the characteristic function of random variable $Z, X$ and $N$, respectively. 
	
Let $\cF$ denote Fourier transformation operator and $\cF^{-1}$ denote the inverse Fourier transformation operator. 
By applying these operators, we obtain the deconvolution estimate for $f_X(x)$, namely, $\hat{f}_X(x)$ as
\begin{equation}\label{eqn:kernel.est}
	\hat{f}_X(x) = \cF^{-1} \left\{ \frac{ \cF\{ \hat{f}_Z(x) \} (t)}{\phi_N(t)} \right\} 
		=  \frac{1}{hn} \sum_{i=1}^n L\Big(\frac{x - Z_i}{h}\Big),
\end{equation}
where we define
\begin{align*}	
	L & \equiv  \cF^{-1} \left\{ \frac{ \phi_K(\, \cdot \,) }{\phi_N(\, \cdot \, h^{-1})} \right\}, \quad \text{i.e.,} \quad
	L(z) = \frac{1}{2\pi} \int \exp(- \img\, t z ) \frac{\phi_K(t)}{\phi_N\left(\frac{t}{h}\right)} dt, ~~z \in \Reals.
\end{align*}

Indeed, this is known as deconvolution kernel density estimator in literature. We shall adopt prior results of Fan \citet{Fan1991} 
on its consistency to establish our results. We refer interested readers to the textbook by Wand and Jones \citet{WandJones94} 
for more details and properties of kernel density estimation.

\subsubsection{Usual Assumptions Made for Deconvolution}\label{sec:deconv_assumptions}
	\paragraph{Assumptions on the Signal Density, $f_X$} 
	Given constants $m,B \geq0$, and $\alpha \in [0,1)$, we define a class of densities following Fan \cite{Fan1991} as
	\begin{equation}\label{eqn:Fan_class}
		\cC_{m, \alpha, B} = \{ f_X(x): \left| f_X^{(m)}(x) - f_X^{(m)}(x + \delta) \right| \leq B \delta^{\alpha} \}.
	\end{equation}
	Intuitively, that implies that the signal density, $f_X$, is sufficiently ``smooth' (slowly varying with respect to $x$) so that 
	there is a hope to reconstruct it from a finite number of samples by interpolating the empirical density. 
	
	\paragraph{Assumptions on the Noise Density, $f_N$}\label{appx:noise}	
	Fan showed that the hardness of deconvolution depends on the smoothness of the noise distribution as well as the smoothness 
	of the signal density to be estimated \citet{Fan1991}. Here, the term `smoothness' means the order (the rate of decay) of 
	the characteristic function as $t \to \infty$. In short, deconvolution becomes more difficult as it is corrupted by 
	smoother\footnote{Smoother noise has faster decaying tail in the Fourier domain (characteristic function). Intuitively, 
	one may consider the smoother noise has heavier tail in the original domain, e.g., due to the uncertainty principle.} additive noise. 
	Following Fan, we call the distribution of a random variable $N$ smooth of order $\beta$ if its characteristic function $\phi_N$ satisfies
	\begin{equation}\label{eqn:ord_smooth}
		B^{-1} \left( 1 + |t| \right)^{-\beta} \leq \left| \phi_N (t) \right| \leq B \left( 1 + |t| \right)^{-\beta},
	\end{equation}
	for some positive constants $\beta, B > 0$, and for all real $t$ \citet{Fan1991}. 
	This class of densities is called ordinary-smooth and such densities have polynomially decaying tails in the Fourier domain. 
	Some examples of the ordinary-smooth error distributions include symmetric Gamma and double exponential distributions. 
	
	There is another interesting class of error distributions, whose tails decay much faster in the Fourier domain. 
	We will call the distribution of a random variable $N$ super-smooth of order $\beta$ if its characteristic function $\phi_N$ satisfies
	\begin{equation}\label{eqn:supersmooth}
		B^{-1} \exp \left( -\gamma |t|^{\beta} \right) \leq \left| \phi_N (t) \right| \leq B\exp \left( -\gamma |t|^{\beta} \right),
	\end{equation}
	for some positive constants $\beta, \gamma>0$ and $B > 1$, and for all real $t$. 
	Normal, mixture normal, Cauchy distributions belong to the super-smooth class. 

	\paragraph{Assumptions on the Kernel, $K$} 
	Typically, the kernel used in kernel deconvolution is assumed to satisfy the following four properties:
	\begin{enumerate}
		\item[(K1)] $\phi_K(t)$ is symmetric
		\item[(K2)] $\phi_K(t)$ has bounded integrable derivatives up to order $m+2$ on $\Reals$, where $m$ is 
			the signal parameter as in \eqref{eqn:Fan_class};
		\item[(K3)] $\phi_K(t) = 1 + \cO\left(|t|^m\right)$ as $t \to 0$;
		\item[(K4)] $\phi_K(t) = 0$, for $|t| > 1$.
	\end{enumerate}

\subsubsection{Some Known Results from Deconvolution Literature}\label{sec:deconv_results}
Here we summarize two theorems from Fan's seminal paper on deconvolution \cite{Fan1991}. The following theorems provide 
the convergence rate of the kernel deconvolution estimator as well as its consistency under the setup where the noise density is known. 
Specifically, the signal density is assumed to belong to Fan's $\cC_{m, \alpha, B}$ class for some $m,B \geq0$, and $\alpha \in [0,1)$ 
\eqref{eqn:kernel.est} and the noise density is assumed supersmooth \eqref{eqn:supersmooth}.
 
We use the subscript $n$ in $\hat{f}_n$ to emphasize that $\hat{f}_X$ is an estimator for $f_X$ based on $n$ samples. 
\begin{theorem}[\citet{Fan1991}, Theorem 1]\label{thm:Fan1}
	Suppose that the noise density is known and super-smooth as defined in \eqref{eqn:supersmooth}. 
	Given a kernel that satisfies (K1), (K2), (K3), (K4), it is possible to achieve
	\[	\sup_{f \in \cC_{m,\alpha,B}}\sup_{x \in \Reals} \Exp{\left( \hat{f}_n (x) - f(x)\right)^2} 
			= \cO \left( \left( \log n \right)^{-2(m+\alpha)/\beta} \right)	\]
	by the kernel deconvolution estimator with the choice of kernel bandwidth parameter 
	$h_n = \left(4\gamma\right)^{\frac{1}{\beta}}\left( \log n \right)^{-\frac{1}{\beta}}$.
\end{theorem}

The same paper has another theorem (which is presented as a corollary of Theorem \ref{thm:Fan1} in the original paper), which fits 
our purpose better. With $\hat{f}_n$, it is possible to define $\hat{F}_n$, an estimator of the CDF of $X$ by integrating $\hat{f}_n$: 
\begin{equation}\label{eqn:kernel_CDF}
	\hat{F}_n(x) = \int_{-M_n}^{x} \hat{f}_n(z) dz.
\end{equation}
$M_n$ is a sequence of constants, which tends to $-\infty$ as $n \to \infty$. The following theorem provides a convergence rate, 
which is better than na\"ively integrating that bound from Theorem \ref{thm:Fan1}.
	
\begin{theorem}[\citet{Fan1991}, Theorem 3]\label{thm:Fan2}
	Let the same assumptions hold as in Theorem \ref{thm:Fan1} except for that we require the kernel to satisfy (K2) and (K3) 
	with parameter $m+1$ instead of $m$. Then it is possible to achieve 
	\[	\sup_{f \in \cC'_{m,\alpha,B}}\sup_{x \in \Reals} \Exp{\left( \tilde{F}_n (x) - F(x)\right)^2} 
			= \cO \left( \left( \log n \right)^{-2(m+\alpha+1)/\beta} \right).		\]
	by the kernel deconvolution estimator with the same 
	choice of the bandwidth parameter $h_n = \left(4\gamma\right)^{\frac{1}{\beta}}\left( \log n \right)^{-\frac{1}{\beta}}$ 
	and $M_n = n^{\frac{1}{3}}$. Here, $\cC'_{m,\alpha,B}= \left\{ f \in \cC_{m,\alpha,B}: F(-n) \leq 
	D \left( \log n \right)^{-(m+2)/\beta} \right\}$.
\end{theorem}


\end{document}